\documentclass[letterpaper,11pt]{article}
\usepackage[margin=1in]{geometry}

\usepackage{amsmath, amsthm, amssymb}
\usepackage{mathtools,hyperref,graphicx,tikz,colortbl,enumitem,cancel}
\usepackage{algorithm,algorithmicx,algpseudocode}
\usepackage{multirow}

\definecolor{gray1}{gray}{0.9}
\definecolor{gray2}{gray}{0.8}
\definecolor{gray3}{gray}{0.7}

\newtheorem{theo}{Theorem}[section]

\newtheorem{lemm}[theo]{Lemma}
\newtheorem{corol}[theo]{Corollary}

\theoremstyle{definition}

\theoremstyle{remark}

\theoremstyle{remark}

\numberwithin{equation}{section}

\def\PLR{\mathrm{PLR}}
\def\PLS{\mathrm{PLS}}

\def\smallgraphnodedistance{0.25}

\begin{document}

\title{\bf \Large Enumerating partial Latin rectangles}

\author{
Ra\'{u}l M.\ Falc\'{o}n\\
\small  School of Building Engineering, University of Seville, Spain.\\ \small \url{rafalgan@us.es}\\ \\
Rebecca J.\ Stones\\ \small College of Computer Science, Nankai University, Tianjin, China.\\
\small School of Mathematical Sciences and Faculty of Information Technology, Monash University, Australia.\\
\small Department of Mathematics and Statistics, Dalhousie University, Halifax, Canada.\\
\small \url{rebecca.stones82@gmail.com}
}
\date{\today}

\maketitle

\begin{abstract}
This paper deals with distinct computational methods to enumerate the set $\PLR(r,s,n;m)$ of $r \times s$ partial Latin rectangles on $n$ symbols with $m$ non-empty cells. For fixed $r$, $s$, and $n$, we prove that the size of this set is a symmetric polynomial of degree $3m$, and we determine the leading terms (the monomials of degree $3m$ through $3m-9$) using inclusion-exclusion. For $m \leq 13$, exact formulas for these symmetric polynomials are determined using a chromatic polynomial method. Adapting Sade's method for enumerating Latin squares, we compute the exact size of $\PLR(r,s,n;m)$, for all $r \leq s \leq n \leq 7$, and all $r \leq s \leq 6$ when $n=8$.  Using an algebraic geometry method together with Burnside's Lemma, we enumerate isomorphism, isotopism, and main classes when $r \leq s \leq n \leq 6$.  Numerical results have been cross-checked where possible.
\end{abstract}

{\bf Keywords}: Partial Latin rectangle, \, isomorphism,\, isotopism,\,  main class,\, inclusion-exclusion\, chromatic polynomial\, algebraic geometry.

{\bf MSC[2010]}: 05B15.

\section{Introduction}\label{sec:intro}

Let $[n]:=\{1,2,\ldots,n\}$. An $r \times s$ \emph{partial Latin rectangle} $L=(l_{ij})$ on the symbol set $[n] \cup \{\cdot\}$ is an $r \times s$ matrix such that each row and each column has at most one copy of any symbol in $[n]$.  Here, $r$, $s$, and $n$ are arbitrary positive integers, and we admit the possibility that $n < \min\{r, s\}$. If $r=s=n$, then this constitutes a {\em partial Latin square} of order $n$.  The cells containing the symbol $\cdot$ are considered \emph{empty}, and we say that $l_{ij}$ is \emph{undefined}. An \emph{entry} of $L$ is any triple $(i,j,l_{ij})\in [r]\times [s]\times [n]$. The set of all entries of $L$ is called its {\em entry set}, which is denoted $E(L)$. The \emph{weight} of $L$ is its number of non-empty cells, that is, the size of its entry set. Let $\PLR(r,s,n;m)$ denote the set of $r \times s$ partial Latin rectangles on the symbol set $[n] \cup \{\cdot\}$ of weight $m$ and let $\PLR(r,s,n)=\cup_{0\leq m \leq rs} \PLR(r,s,n;m)$. Let $\PLS(n;m)=\PLR(n,n,n;m)$ be the set of partial Latin squares of weight $m$. For $m=n^2$, this is the set of {\em Latin squares of order $n$}.

For each positive integer $t\in\mathbb{Z}^+$, let $S_t$ denote the symmetric group on the set $[t]$.
\begin{itemize}
 \item The {\em isotopism group} $\mathfrak{I}_{r,s,n}:=S_r\times S_s\times S_n$ acts on the set $\PLR(r,s,n;m)$, with the \emph{isotopism} $\Theta=(\alpha,\beta,\gamma)$ permuting the rows according to $\alpha$, the columns according to $\beta$, and the symbols according to $\gamma$. This gives the \emph{isotopic} partial Latin rectangle $L^{\Theta}\in \PLR(r,s,n;m)$, whose entry set is $E(L^{\Theta})=\left\{(\alpha(i),\beta(j),\gamma(l_{i,j}))\colon (i,j,l_{ij})\in E(L)\right\}$.
\item The symmetric group $S_n$ is isomorphic to the subgroup $\{(\alpha,\alpha,\alpha):\alpha \in S_n\}$ of the isotopism group $\mathfrak{I}_{n,n,n}$ via the isomorphism $\alpha \mapsto (\alpha,\alpha,\alpha)$. In this regard, the \emph{isomorphism group} $S_n$ acts on the set $\PLS(n;m)$, with $\alpha \in S_n$ mapping $L$ to $L^{(\alpha,\alpha,\alpha)}$.
\item Let $\pi\in S_3$ and $L \in \PLR(d_1,d_2,d_3;m)$. The \emph{parastrophic} partial Latin rectangle $L^{\pi} \in \PLR(d_{\pi(1)},d_{\pi(2)},d_{\pi(3)};m)$ is defined so that its entry set is $E(L^{\pi})=\{(p_{\pi(1)},p_{\pi(2)},p_{\pi(3)})\colon$ $(p_1,p_2,p_3)\in E(L)\}$.  The permutation $\pi$ is said to be a {\em parastrophism}. Since parastrophisms may not preserve the dimensions of partial Latin rectangles, the \emph{parastrophism group} $S_{r,s,n}$ is defined as the stabilizer of the ordered triple $(r,s,n)$ under the action $(d_1,d_2,d_3) \xmapsto{\pi} (d_{\pi(1)},d_{\pi(2)},d_{\pi(3)})$ by $S_3$.
\item The {\em paratopism group} $\mathfrak{P}_{r,s,n}:=\mathfrak{I}_{r,s,n}\rtimes S_{r,s,n}$ acts on the set $\PLR(r,s,n;m)$ so that each {\em paratopism} $(\Theta,\pi)$ maps $L$ to the {\em paratopic} partial Latin rectangle $L^{(\Theta,\pi)}=(L^{\pi})^{\Theta}$. When $\pi=\mathrm{Id}$ (i.e., the trivial permutation in $S_3$) the isotopism group arises as a normal subgroup of the paratopism group.
\end{itemize}

Orbits of $\PLR(r,s,n;m)$ under the isotopism, isomorphism, and paratopism groups are equivalence classes, called \emph{isotopism}, \emph{isomorphism}, and \emph{main} classes, respectively. The stabilizer subgroups under these groups are called {\em autotopism}, {\em automorphism}, and \emph{autoparatopism} groups, respectively. Let $\PLR((\Theta,\pi))$ and $\PLR((\Theta,\pi);m)$ denote, respectively, the subsets of partial Latin rectangles in the sets $\PLR(r,s,n)$ and $\PLR(r,s,n;m)$ that admit an autoparatopism $(\Theta,\pi) \in \mathfrak{P}_{r,s,n}$.

The goal of this paper is to find methods for computing the size of $\PLR(r,s,n;m)$, along with its equivalence class sizes. It is unrealistic to expect a succinct solution to both problems for arbitrary $r$, $s$, $n$, and $m$, since they include in particular the number of Latin squares of given order $n$, which is a long-standing research problem in combinatorics. This is known only for order $n\leq 11$ \cite{HulpkeKaskiOstergard2011,McKayWanless2005}; see \cite{Stones2009b,Stones2009,StonesLinLiuWang2016} for some related results on Latin rectangles. Currently, the number of partial Latin rectangles is known only for $r,s,n\leq 6$ \cite{Falcon2013, Falcon2015, Falcon2018}. To make progress, we need to restrict these parameters in some way. In particular, we enumerate (a) fixed-weight partial Latin rectangles, (b) partial Latin rectangles for small $m$, and (c) partial Latin rectangles for small $r$, $s$, and $n$.

The number of isotopism, isomorphism, and main classes of Latin squares has been determined \cite{HulpkeKaskiOstergard2011, McKayMeynertMyrvold2007} for order $n\leq 11$, whereas for partial Latin rectangles, these numbers were computed \cite{FalconStones2015} for $r,s,n \leq 6$. Adams, Bean, and Khodkar \cite{AdamsBeanKhodkar2003} enumerated main classes of partial Latin squares of order $n\leq 6$ that constitute critical sets. More recently, it has been obtained  \cite{DietrichWanless2018,WanlessWebb2017} the number of main classes of partial Latin rectangles with at most $12$ entries.

Autoparatopisms and asymmetry for partial Latin squares were studied in \cite{FalconStones,Stones2013} and several constructions of partial Latin rectangles with trivial autotopism groups for various autoparatopism groups was given in \cite{FalconStones2017}.  Computational methods to determine autotopism groups of partial Latin rectangles were compared in \cite{FalconKotlarStones2,FalconKotlarStones}.  For Latin squares of order $n\leq 17$, identifying when $\#\PLR((\Theta,\pi)) \neq 0$ (throughout this paper $\#$ denotes the cardinality of a set) was done for isotopisms in \cite{StonesVojtechovskyWanless2012} and paratopisms in \cite{MendisWanless}, with prior work in \cite{Falcon2009,FalconMartinMorales2007}.

Symmetries of Latin squares and rectangles have been studied in a wide range of contexts, e.g., enumeration \cite{MengZhengZheng2008,StonesWanless2009a,StonesWanless2009b}, subsquares \cite{Browning2013,Mendis2016}, the Alon-Tarsi Conjecture \cite{Drisko1997b,StonesWanless2009d}, quasigroups and loops \cite{Artzy1954,KerbySmith2009, KerbySmith2010,McKayWanlessZhang2015}, special kinds of symmetries \cite{CavenaghStones2010a,FalconNunez2007,IhrigIhrig2008,WanlessIhrig2005}, and in their own right \cite{Bailey1982,BryantBuchananWanless2009,Falcon2007,Falcon2008}.  They are beginning to find applications in secret sharing schemes \cite{Ganfornina2006, StonesSuLiuWangLin,KongEtAl}, erasure codes \cite{YiEtAl2019}, and graph coloring games \cite{FalconAndres2018a,FalconAndres2018b}.

The remainder of the paper is organized as follows. The three following sections deal with distinct combinatorial methods that enable us to determine the size of the set $\PLR(r,s,n;m)$: (a) Section~\ref{sec:inc-exc}: an inclusion-exclusion method that demonstrates $\#\PLR(r,s,n;m)$ for fixed $m$ is a symmetric polynomial of degree $3m$; (b) Section~\ref{sec:chrom_poly}: a chromatic polynomial method that gives exact formulas for this symmetric polynomial, which we compute for $m\leq 13$; and (c) Section~\ref{sec:Sade}: an adaptation of Sade's method (which efficiently enumerates Latin squares) to partial Latin rectangles, which enables us to determine explicitly the number $\#\PLR(r,s,n;m)$ for all $r\leq s\leq n\leq 7$, and all $r\leq s\leq 6$ when $n=8$. Section~\ref{sec:alg_geo} describes an algebraic method for computing $\#\PLR((\Theta,\pi);m)$ and also the number of isotopisms between two given partial Latin rectangles. In Section~\ref{sec:equiv} we use the Orbit-Stabilizer Theorem and Burnside's Lemma to compute the size of isomorphisms, isotopism and main classes. Section~\ref{sec:comp_results} describes the computational results and the implementations of the various methods. In Section~\ref{sec:verify} we comment how these computational results have been cross-checked in order to ensure their accuracy. A glossary of the most common symbols that are used throughout the paper is shown in Appendix~\ref{sec:symbols}. To improve the readability of the paper, tables are in Appendix~\ref{sec:tables}.

\section{Inclusion-exclusion method}\label{sec:inc-exc}

For fixed $m \geq 1$, we find formulas for the size of $\PLR(r,s,n;m)$ by modifying the method for enumerating partial orthomorphisms of finite cyclic groups given in \cite{StonesWanless2009b}.  At first, this is a surprising claim, as partial Latin rectangles and partial orthomorphisms are largely unrelated (unless we impose some symmetry, which we don't in the context of this section). The similarity between these two types of objects is that both partial Latin rectangles of weight $m$ and partial orthomorphisms with domain size $m$ are equivalent to non-clashing $m$-sets of ordered triples (the difference is what constitutes a ``clash'').

\subsection{Generalized ordered partial Latin rectangles}

Let $\mathcal{S}_m=\mathcal{S}(r,s,n;m)$ be the set of sequences $\mathbf{e}=(e_i)_{i=1}^m$, where each $e_i=(e_i[1],e_i[2],$ $e_i[3])$ is a $3$-tuple in $[r] \times [s] \times [n]$.  From any $\mathbf{e} \in \mathcal{S}_m$, we construct an $r \times s$ matrix $M=M(\mathbf{e})$ as follows:
\begin{itemize}
 \item We begin with each cell in $M$ containing the empty multiset $\emptyset$.
 \item For $i \in [m]$, we add symbol $e_i[3]$ in the multiset in cell $(e_i[1],e_i[2])$.
\end{itemize}
For example, if $r=s=n=m=3$ and $\mathbf{e}=\big((1,1,1),(1,2,3),(1,1,1)\big)$, then
\[M(\mathbf{e})=\begin{array}{|ccc|}
\hline
\{1,1\} & \{3\} & \emptyset \\
\emptyset & \emptyset & \emptyset \\
\emptyset & \emptyset & \emptyset \\
\hline
\end{array}.\]
If it turns out that every non-empty multiset in $M$ has cardinality $1$ and there are no repeated elements in any row or column of $M$, then $M$ is essentially a partial Latin rectangle (formally, we need to map $\emptyset \mapsto \cdot$ and $\{i\} \mapsto i$).  For example, if $r=s=n=m=3$ and $\mathbf{e}=\big((1,1,1),(1,2,3),(2,2,2)\big)$, then

\[M(\mathbf{e})=\begin{array}{|ccc|}
\hline
\{1\} & \{3\} & \emptyset \\
\emptyset & \{2\} & \emptyset \\
\emptyset & \emptyset & \emptyset \\
\hline
\end{array}
\longleftrightarrow
\begin{array}{|ccc|}
\hline
1 & 3 & \cdot \\
\cdot & 2 & \cdot \\
\cdot & \cdot & \cdot \\
\hline
\end{array}.\]
Thus, sequences in $\mathcal{S}_m$ are generalized partial Latin rectangles consisting of $m$ ordered entries.

Let $\mathcal{A}_m$ be the subset of $\mathcal{S}_m$ that gives rise to partial Latin rectangles.  Hence, \[|\mathcal{A}_m|=m!\,\#\PLR(r,s,n;m)\] because we can order the entries in a partial Latin rectangle in $m!$ ways.  For fixed $m$, we define
\[f_m(r,s,n):=|\mathcal{A}_m|.\]
To find a formula for $f_m(r,s,n)$ by using inclusion-exclusion on the number of ``clashes'', let
\[C_m:=\{[i,j,k] \colon 1 \leq i < j \leq m \text{ and } k\in \{1,2,3\}\},\]
which we use to index the possible clashes in $\mathbf{e} \in \mathcal{S}_m$ as follows:
\begin{itemize}[leftmargin=1.2in]
 \item[Clash {$[i,j,1]$}:] When $e_i[2]=e_j[2]$ and $e_i[3]=e_j[3]$.  This would result in two copies of the same symbol in the same column in $M$ (not necessarily in distinct cells).
 \item[Clash {$[i,j,2]$}:] When $e_i[1]=e_j[1]$ and $e_i[3]=e_j[3]$.  This would result in two copies of the same symbol in the same row in $M$ (not necessarily in distinct cells).
 \item[Clash {$[i,j,3]$}:] When $e_i[1]=e_j[1]$ and $e_i[2]=e_j[2]$.  This would result in two (not necessarily distinct) symbols in the same cell in $M$.
\end{itemize}
Any $\mathbf{e} \in \mathcal{S}_m$ therefore has a corresponding set of clashes $C_{\mathbf{e}} \subseteq C_m$. For any $U \subseteq C_m$, define
\[\mathcal{B}_U:=\{\mathbf{e} \in \mathcal{S}_m \colon U \subseteq C_{\mathbf{e}}\},\]
i.e., the sequences in $\mathcal{S}_m$ that have the clashes in $U$ (and possibly more clashes), and
\[\mathcal{D}_U:=\{\mathbf{e} \in \mathcal{S}_m \colon U = C_{\mathbf{e}}\},\]
i.e., the sequences in $\mathcal{S}_m$ that have precisely those clashes in $U$ (and no more clashes).  By definition, \[\mathcal{D}_U=\mathcal{B}_U \setminus \bigcup_{\substack{V \subseteq C_m \\ V \supsetneq U}} \mathcal{B}_V.\]
Hence, by inclusion-exclusion,
\begin{align*}
|\mathcal{D}_U| &= |\mathcal{B}_U|-\Biggl|\bigcup_{\substack{V \subseteq C_m \\ V \supsetneq U}} \mathcal{B}_V \Biggr| \\
      &= |\mathcal{B}_U|+\sum_{\substack{V \subseteq C_m \\ V \supsetneq U}} (-1)^{|V|-|U|} |\mathcal{B}_V| \\
      &= \sum_{\substack{V \subseteq C_m \\ V \supseteq U}} (-1)^{|V|-|U|} |\mathcal{B}_V|.
\end{align*}
When $U=\emptyset$, we have $|\mathcal{D}_U|=|\mathcal{A}_m|$ and consequently the following lemma.

\begin{lemm}\label{lm:PLRIncExc}
For all $m,r,s,n \geq 1$, we have
$$f_m(r,s,n)=\sum_{V \subseteq C_m} (-1)^{|V|} |\mathcal{B}_V|.$$
\end{lemm}

\subsection{Graph colorings}

Our next goal is to find an equation for $|\mathcal{B}_V|$ in terms of the number of vertex colorings of an edge-colored graph, satisfying some additional constraints (neither vertex colorings nor edge colorings are required to be proper in the ordinary sense).  Given $V \subseteq C_m$, we define a graph $G=G(V)$ with an edge coloring $\delta=\delta(V)$ by the following process.  We start with the null graph on the vertex set $[m]$, and for each $[i,j,k] \in V$:
\begin{itemize}
  \item[I:] If $k=1$, then add a red edge between $i$ and $j$.
  \item[II:] If $k=2$, then add a blue edge between $i$ and $j$.
  \item[III:] If $k=3$, then add a green edge between $i$ and $j$.
  \item[IV:] Replace any parallel edges resulting from I--III with a single black edge.
\end{itemize}
We denote the graph together with its edge coloring generated from $V$ by $(G,\delta)_V$.  An example of an edge-colored graph generated in this way is given in Figure~\ref{FIGraphExample}.

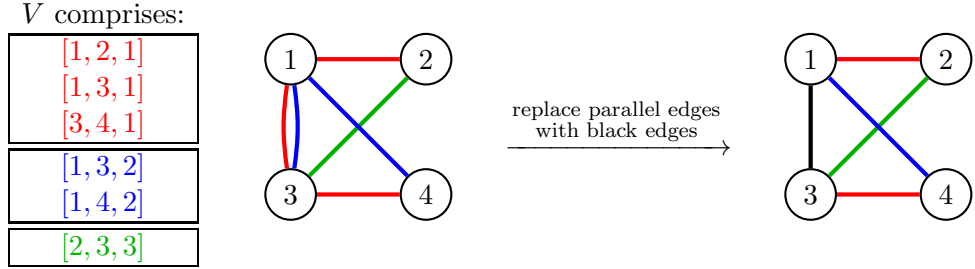
\begin{figure}[ht]
\centering
$
\begin{array}{ccccc}
\begin{array}{|c|}
\multicolumn{1}{c}{V \text{ comprises:}} \\
\hline
\color{red} [1,2,1] \\
\color{red} [1,3,1] \\
\color{red} [3,4,1] \\
\hline
\hline
\color{blue} [1,3,2] \\
\color{blue} [1,4,2] \\
\hline
\hline
\color{green!70!black} [2,3,3] \\
\hline
\end{array}
&
\qquad
&
\begin{array}{c}
\begin{tikzpicture}[scale=0.6]
\tikzstyle{every node}=[draw,thick,circle,minimum size=8pt];
\draw (0,0) node (a) {1};
\draw (3,0) node (b) {2};
\draw (0,-3) node (c) {3};
\draw (3,-3) node (d) {4};

\draw[ultra thick,color=red] (a) to  (b);
\draw[ultra thick,color=green!70!black] (b) to (c);
\draw[ultra thick,color=red] (c) to (d);
\draw[ultra thick,color=blue] (a) to (d);
\draw[ultra thick,color=red] (c) to[bend left=8] (a);
\draw[ultra thick,color=blue] (c) to[bend right=8] (a);
\end{tikzpicture}
\end{array}
&
\begin{array}{c}
\xrightarrow{\substack{\text{replace parallel edges}\\\text{with black edges}}}
\end{array}
&
\begin{array}{c}
\begin{tikzpicture}[scale=0.6]
\tikzstyle{every node}=[draw,thick,circle,minimum size=8pt];
\draw (0,0) node (a) {1};
\draw (3,0) node (b) {2};
\draw (0,-3) node (c) {3};
\draw (3,-3) node (d) {4};

\draw[ultra thick,color=red] (a) to  (b);
\draw[ultra thick,color=green!70!black] (b) to (c);
\draw[ultra thick,color=red] (c) to (d);
\draw[ultra thick,color=blue] (a) to (d);
\draw[ultra thick] (c) to (a);
\end{tikzpicture}
\end{array}
\end{array}
$
\caption{An example of the graph $G=G(V)$, and its edge coloring $\delta=\delta(V)$ (right), for the set of clashes $V \subseteq C_4$.}\label{FIGraphExample}
\end{figure}

Sequences $\mathbf{e} \in \mathcal{B}_V$ are equivalent to a special type of vertex coloring $\phi$ of $(G,\delta)_V$, for which we assign to vertex $i \in [m]$ the color
$$\big(\phi_1(i),\phi_2(i),\phi_3(i)\big):=\big(e_i[1],e_i[2],e_i[3]\big) \in [r] \times [s] \times [n].$$
This coloring satisfies the properties:
\begin{itemize}
 \item If there is a red edge between vertices $i$ and $j$, then $\phi_2(i)=\phi_2(j)$ and $\phi_3(i)=\phi_3(j)$.
 \item If there is a blue edge between vertices $i$ and $j$, then $\phi_1(i)=\phi_1(j)$ and $\phi_3(i)=\phi_3(j)$.
 \item If there is a green edge between vertices $i$ and $j$, then
     $\phi_1(i)=\phi_1(j)$ and $\phi_2(i)=\phi_2(j)$.
 \item If there is a black edge between vertices $i$ and $j$, then $\phi_1(i)=\phi_1(j)$, $\phi_2(i)=\phi_2(j)$ and $\phi_3(i)=\phi_3(j)$.
\end{itemize}
We call such a vertex coloring of $(G,\delta)_V$ \emph{suitable}.  Conversely, any suitable vertex coloring of $(G,\delta)_V$ with the vertex color set $[r] \times [s] \times [n]$ that satisfies the above four properties is a member of $\mathcal{B}_V$, thus giving the following lemma.

\begin{lemm}
For all $V \subseteq C_m$, the set $\mathcal{B}_V$ is the set of suitable vertex colorings of $(G,\delta)_V$, hence $|\mathcal{B}_V|$ is the number of suitable proper vertex colorings of $(G,\delta)_V$.
\end{lemm}

We can find a simple formula (Lemma~\ref{lm:BForm}) for the number of suitable colorings of $(G,\delta)_V$ since each of the three coordinates can be accounted for separately.  Let $H_1$, $H_2$, and $H_3$ respectively be the graphs formed by deleting the red, blue, and green edges from $(G,\delta)_V$, then ignoring the edge colors.  For any graph $H$, let $c(H)$ denote the number of connected components in $H$.

\begin{lemm}\label{lm:BForm}
The number of proper suitable colorings of $(G,\delta)_V$ is
\[|\mathcal{B}_V|=r^{c(H_1)}s^{c(H_2)}n^{c(H_3)}.\]
\end{lemm}

\begin{proof}
In order to be a suitable vertex coloring, the vertices in each component of $H_1$ must be assigned colors in $G$ that agree at the first coordinate.  We can thus assign the first coordinates of a suitable vertex coloring in $r^{c(H_1)}$ ways. Similar claims hold for $H_2$ and $H_3$.
\end{proof}

We are now ready to make the following fundamental observation about the polynomials $f_m$.

\begin{theo}\label{th:poly}
For fixed $m$, we have that $f_m=f_m(r,s,n)$ is given by a $3$-variable symmetric polynomial with integer coefficients of degree $3m$.
\end{theo}

\begin{proof}
Lemmas~\ref{lm:PLRIncExc} and~\ref{lm:BForm} imply that $f_m(r,s,n)=\sum_{V \subseteq C_m} (-1)^{|V|} |\mathcal{B}_V|$ where $|\mathcal{B}_V|$ is given by $r^{c(H_1)}s^{c(H_2)}n^{c(H_3)}$ for the graph $(G,\delta)_V$.  This ensures that $f_m(r,s,n)$ is a polynomial in variables $r,s,n$ and has integer coefficients.  The leading term is $(rsn)^m$, which arises when $V=\emptyset$; for all other $V \subseteq C_m$, we see $|\mathcal{B}_V|$ has degree less than $3m$. Finally, to verify that $f_m(r,s,n)$ is a symmetric polynomial, we observe that we can permute the colors red, blue, and green (or equivalently, permute the third coordinate of the elements in $C_m$).  Each equivalence class under this action contributes
\begin{align*}
 & r^{c(H_1)}s^{c(H_2)}n^{c(H_3)}
+r^{c(H_1)}s^{c(H_3)}n^{c(H_2)}
+r^{c(H_2)}s^{c(H_1)}n^{c(H_3)} \\
+{}& r^{c(H_2)}s^{c(H_3)}n^{c(H_1)}
+r^{c(H_3)}s^{c(H_1)}n^{c(H_2)}
+r^{c(H_3)}s^{c(H_2)}n^{c(H_1)}
\end{align*}
to the sum in Lemma~\ref{lm:PLRIncExc}, which is symmetric.  We conclude that $f_m(r,s,n)$ is the sum of symmetric polynomials, and is also symmetric.
\end{proof}

\subsection{A simplified equation}

For a $4$-edge-colored graph $(G,\delta)$, with possible edge colors red, blue, green, and black, let $|G|$ be the number of vertices in $G$, let $|E(G)|$ be the number of edges in $G$, and let $b(\delta)$ be the number of black edges in $\delta$. There are $4^{|b(\delta)|}$ sets $V \subseteq C_m$ for which $(G,\delta)=(G,\delta)_V$, since a black edge can be formed in $4$ possible ways: (a) when exactly two of properties I, II and III hold, or (b) when all three of properties I, II and III hold. From Lemmas~\ref{lm:PLRIncExc} and~\ref{lm:BForm}, we have

\begin{align*}
f_m(r,s,n) &= \sum_{V \subseteq C_m} (-1)^{|V|} |\mathcal{B}_V| \\
 &= \sum_{\substack{(G,\delta) \\ |G|=m}} \sum_{\substack{V \subseteq C_m: \\ G(V)=G, \\ \delta(V)=\delta}} (-1)^{|V|} |\mathcal{B}_V| \\
 &= \sum_{\substack{(G,\delta) \\ |G|=m}} r^{c(H_1)}s^{c(H_2)}n^{c(H_3)} \sum_{\substack{V \subseteq C_m: \\ G(V)=G, \\ \delta(V)=\delta}} (-1)^{|V|}.
\end{align*}
From here, we use the following identity from \cite{StonesWanless2009b}: For any $(G,\delta)$, we have
\begin{align*}
\sum_{\substack{V \subseteq C_m: \\ G(V)=G, \\ \delta(V)=\delta}} (-1)^{|V|} &= \sum_{x \geq 0} \binom{b(\delta)}{x} (-1)^{|E(G)|+b(\delta)+x} 3^{b(\delta)-x} \\
 &= (-1)^{|E(G)|} (-2)^{b(\delta)}
\end{align*}
using the Binomial Theorem.  The local variable $x$ counts the number of black edges where I, II and III all hold.  This yields the following theorem:
\begin{theo}\label{th:PolySimp}
For all $r,s,n,m \geq 1$, we have
$$f_m(r,s,n) = \sum_{\substack{(G,\delta) \\ |G|=m}} (-1)^{|E(G)|} (-2)^{b(\delta)} r^{c(H_1)}s^{c(H_2)}n^{c(H_3)}.$$
\end{theo}
The advantage of Theorem~\ref{th:PolySimp} is that it eliminates the need for accounting for clashes (via the variable $V$).  Instead, we are now working solely with graphs. For computational purposes, it is easier to work with isomorphism classes of graphs (rather than labeled graphs).

We will also account for isolated vertices mathematically.  For $v \geq 0$ and $e \geq 0$, let $\Gamma_{e,v}$ denote the set of unlabeled $e$-edge $v$-vertex graphs without isolated vertices (the set $\Gamma_{0,0}$ contains the empty graph, whereas $\Gamma_{e,1}=\emptyset$).  We can split Theorem~\ref{th:PolySimp} according to $e$, $v$, and $\Gamma_{e,v}$ to give the following theorem.

\begin{theo}\label{th:isosplit}
For all $r,s,n,m \geq 1$, we have
$$f_m(r,s,n) =(rsn)^m+\sum_{v \geq 2} \binom{m}{v} (rsn)^{m-v+1} \sum_{e \geq 1} (-1)^e  \sum_{G \in \Gamma_{e,v}} \frac{v!}{|\mathrm{Aut}(G)|} P(G)$$
where $$P(G)=P(G;r,s,n):=\sum_{\delta} (-2)^{b(\delta)} r^{c(H_1)-1}s^{c(H_2)-1}n^{c(H_3)-1}$$ where the sum is over all edge colorings $\delta$ of $G$.
\end{theo}

\begin{proof}
Given a graph $G \in \Gamma_{e,v}$, there are $\binom{m}{v} \frac{v!}{|\mathrm{Aut}(G)|}$ labeled graphs on the vertex set $[m]$ that are isomorphic to $G$ together with $m-v$ isolated vertices.  Thus,
\begin{align*}
f_m(r,s,n) &= \sum_{\substack{(G,\delta) \\ |G|=m}} (-1)^{|E(G)|} (-2)^{b(\delta)} r^{c(H_1)}s^{c(H_2)}n^{c(H_3)} \\
 &= \sum_{\substack{G \in \Gamma_{e,v}\\v \geq 0 \\ e \geq 0}}\sum_\delta \binom{m}{v} \frac{v!}{|\mathrm{Aut}(G)|} (-1)^e (-2)^{b(\delta)} (rsn)^{m-v} r^{c(H_1)}s^{c(H_2)}n^{c(H_3)}.
\end{align*}
We obtain the theorem by rearranging this equation.
\end{proof}

The following corollary follows straightforwardly from Theorem~\ref{th:isosplit}.

\begin{corol}
For fixed $m \geq 1$, the polynomial $f_m(r,s,n)$ is divisible by $rsn$.
\end{corol}

Furthermore, we use the next result to reduce the required computation.

\begin{lemm} Let $G_1$ and $G_2$ be two graphs. Then,
\begin{enumerate}
\item if both graphs are disjoint, then, $P(G_1 \cup G_2)=rsn\,P(G_1)P(G_2)$; and
\item if both graphs meet at a single vertex, then $P(G_1 \cup G_2)=P(G_1)P(G_2)$.
\end{enumerate}
\end{lemm}

% \begin{proof}
% For $i \in \{1,2\}$, given some coloring $\delta$ of $G_1 \cup G_2$, let $\delta_i$ be the coloring $\delta$ restricted to $G_i$; let $b^{i}$ be the number of black edges in $G_i$; and let $H^{(i})_1$, $H^{(i})_2$, and $H^{(i})_3$ respectively be the graphs formed by deleting the red, blue, and green edges from $(G_i,\delta_i)$, then ignoring the edge colors.  It follows that
% \begin{align*}
% P(G_1 \cup G_2) &= \sum_{\delta} (-2)^{b(\delta)} r^{c(H_3)-1}s^{c(H_2)-1}n^{c(H_1)-1} \\
%  &= \sum_{\delta} (-2)^{b^{(1)}(\delta)+b^{(2)}(\delta)} r^{c(H^{(1)}_3)+c(H^{(2)}_3)-1}s^{c(H^{(1)}_2)+c(H^{(2)}_2)-1}n^{c(H^{(1)}_1)+c(H^{(2)}_1)-1} \\
%  &= rcs \left( \sum_{\delta} (-2)^{b^{(1)}(\delta)} r^{c(H^{(1)}_3)-1}s^{c(H^{(1)}_2)-1}n^{c(H^{(1)}_1)-1} \right) \left( \sum_{\delta} (-2)^{b^{(2)}(\delta)} r^{c(H^{(2)}_3)-1}s^{c(H^{(2)}_2)-1}n^{c(H^{(2)}_1)-1} \right) \\
%  &= rsn\,P(G_1)P(G_2).
% \end{align*}
% \end{proof}

Finally, the following lemma is useful for finding which graphs have to be included when computing the leading terms in $f_m(r,s,n)$.

\begin{lemm}\label{lm:maxdeg}
For any graph $G$ on $v$ vertices, the degree of $(rsn)^{m-v+1}P(G)$ in Theorem~\ref{th:isosplit} is at most $3m-2v+2c(G)$.
\end{lemm}

\begin{proof}
From Theorem~\ref{th:isosplit}, the degree of $(rsn)^{m-v+1}P(G)$ is at most $3m-3v+\max_\delta(c(H_1)+c(H_2)+c(H_3))$. Let us show, by induction on the number of edges, that
\begin{equation}\label{eq:compineq}
c(H_1)+c(H_2)+c(H_3) \leq v+2c(G)
\end{equation}
for any $4$-edge-coloring $\delta$ (with equality when all the edges are red, say). If $G$ has no edges, then we have equality in \eqref{eq:compineq}.  Next, assume \eqref{eq:compineq} holds for some $4$-edge-coloring, and add a colored edge $xy$.  Adding this edge will not increase $c(H_1)$, $c(H_2)$, and $c(H_3)$, so \eqref{eq:compineq} continues to hold unless possibly if adding $xy$ affects $c(G)$.  Adding $xy$ decreases $c(G)$ by $1$ if and only if $x$ and $y$ belong to separate components of $G$.  In this case, $x$ and $y$ also belong to separate components of $H_1$, $H_2$, and $H_3$.  If $xy$ is a red edge, then $c(H_2)$ and $c(H_3)$ both decrease by $1$, and \eqref{eq:compineq} holds for the new graph.  The same argument works if $xy$ is a blue or green edge.  If $xy$ is a black edge, then $c(H_1)$, $c(H_2)$, and $c(H_3)$ all decrease by $1$, and \eqref{eq:compineq} holds for the new graph.
\end{proof}

\section{Chromatic polynomial method}\label{sec:chrom_poly}

Let $R_{r,s}$ be the $r \times s$ \emph{rook's graph}, i.e., the Cartesian product of $K_r$ and $K_s$.  The graph $R_{3,4}$ is drawn in Figure~\ref{fi:R34}.  Any partial Latin rectangle in $\PLR(r,s,n;m)$ can be interpreted as a proper $n$-coloring of an $m$-vertex induced subgraph of $R_{r,s}$.  An example of this correspondence is also given in Figure~\ref{fi:R34}.

\begin{figure}[ht]
\centering
% \resizebox{0.4\textwidth}{!}{
\begin{tikzpicture}[scale=0.7]
\matrix[nodes={draw, thick, fill=white!100},row sep=1cm,column sep=1cm] {
  \node[circle](11){}; &
  \node[circle](12){}; &
  \node[circle](13){}; &
  \node[circle](14){}; \\
  \node[circle](21){}; &
  \node[circle](22){}; &
  \node[circle](23){}; &
  \node[circle](24){}; \\
  \node[circle](31){}; &
  \node[circle](32){}; &
  \node[circle](33){}; &
  \node[circle](34){}; \\
};
\draw (11) to (12) to (13) to (14); \draw (11) to[bend left] (13); \draw (12) to[bend left] (14); \draw (11) to[bend left] (14);
\draw (21) to (22) to (23) to (24); \draw (21) to[bend left] (23); \draw (22) to[bend left] (24); \draw (21) to[bend left] (24);
\draw (31) to (32) to (33) to (34); \draw (31) to[bend left] (33); \draw (32) to[bend left] (34); \draw (31) to[bend left] (34);
\draw (11) to (21) to (31) to[bend left] (11);  \draw (12) to (22) to (32) to[bend left] (12);  \draw (13) to (23) to (33) to[bend left] (13);  \draw (14) to (24) to (34) to[bend left] (14);
\end{tikzpicture}
% }
\qquad
% \resizebox{0.4\textwidth}{!}{
\begin{tikzpicture}[scale=0.7]
\matrix[nodes={draw, thick, fill=white!100},row sep=0.7cm,column sep=0.7cm] {
  \node[draw=none](11){$\cdot$}; &
  \node[draw=none](12){$\cdot$}; &
  \node[draw=none](13){$\cdot$}; &
  \node[circle](14){$1$}; \\
  \node[draw=none](21){$\cdot$}; &
%   \node[circle](21){$1$}; &
  \node[circle](22){$2$}; &
  \node[circle](23){$4$}; &
  \node[draw=none](24){$\cdot$}; \\
  \node[circle](31){$2$}; &
  \node[draw=none](32){$\cdot$}; &
  \node[draw=none](33){$\cdot$}; &
  \node[circle](34){$3$}; \\
};
\draw (-4,3) rectangle (4,-3);

\draw (31) to[bend left] (34) to[bend left] (14);
\draw (22) to (23);
\end{tikzpicture}
% }
\caption{The graph $R_{3,4}$ along with a proper $4$-coloring of an induced $5$-vertex subgraph of $R_{3,4}$.  This illustrates the corresponding partial Latin rectangle in $\PLR(3,4,4;5)$.}\label{fi:R34}
\end{figure}
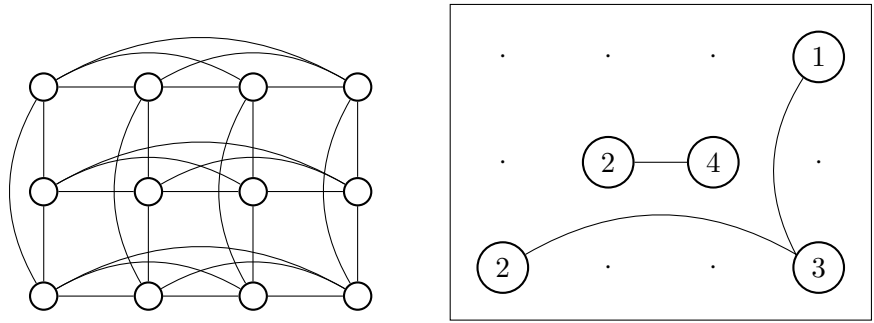

We can naturally think of (labeled) induced subgraphs of $R_{r,s}$ as $(0,1)$-matrices, with a $1$ in cell $(i,j)$ whenever vertex $(i,j)$ is present.  Under this equivalence, we talk of the \emph{rows} and \emph{columns} of such graphs, and of the chromatic polynomial of their corresponding induced subgraphs. If $\Pi$ denotes this chromatic polynomial, then
\begin{equation}\label{eq:sumchrom}
\#\PLR(r,s,n;m) = \sum_{M} \Pi(M;n)
\end{equation}
where the sum is over all $r \times s$ $(0,1)$-matrices $M$ with exactly $m$ ones (or equivalently over all $m$-vertex induced subgraphs of $R_{r,s}$).

% Let $\mathcal{Q}_{r,s}=\mathcal{Q}_{r,s}(m)$ denote the set of $m$-vertex induced subgraphs of $R_{r,s}$, let $\mathcal{Q}^*_{r,s}=\mathcal{Q}^*_{r,s}(m)$ denote the set of $m$-vertex induced subgraphs of $R_{r,s}$ with at least one vertex present in each row and each column, and let $\mathcal{Q}^\text{conn}_{r,s}=\mathcal{Q}^\text{conn}_{r,s}(m)$ denote the set of connected $m$-vertex induced subgraphs of $R_{r,s}$.

% If $\Pi$ denotes the chromatic polynomial of a graph, we thus have
% \begin{align}
% \#\PLR(r,s,n;m) &= \sum_{H \in \mathcal{Q}_{r,s}} \Pi(H;n) \nonumber\\
%  &= \sum_{i \leq r} \sum_{j \leq s} \binom{r}{i} \binom{s}{j} \sum_{H \in \mathcal{Q}^*_{i,j}} \Pi(H;n).\label{eq:chrompolystep1}
% \end{align}

Given any $r \times s$ $(0,1)$-matrix $M$, we define a \emph{general block} of $M$ to be a submatrix $H$ in which: (a) every row and every column of $H$ contains a $1$; (b) in $M$, there are no $1$'s in the rows of $H$ outside of $H$; and (c) in $M$, there are no $1$'s in the columns of $H$ outside of $H$.  We define a \emph{block} as a general block which has no proper submatrix which is general block in itself.  Blocks correspond to components of the induced subgraph of $R_{r,s}$.

We act on the set of $r \times s$ $(0,1)$-matrices by permuting rows and columns.  Under this action, we choose representatives from each orbit and call them \emph{canonical}. After that, we define a function $C$ such that $C(M)$ is the canonical matrix in the orbit of each $(0,1)$-matrix $M$. From $M$, we can also construct a multiset $\{C(M_i)\}_{i=1}^k$ where $M_1,M_2,\ldots,M_k$ are the blocks of $M$.

In the other direction, let $\mathcal{K}_{r,s,m,k}$ denote the set of multisets $\mathbf{K}=\{K_i\}_{i=1}^k$ of canonical blocks such that (a) the number of $1$'s in the blocks sum to $m$; (b) the number of rows in the blocks sum to $\leq r$; and (c) the number of columns in the blocks sum to $\leq s$.  Given any $\mathbf{K} \in \mathcal{K}_{r,s,m,k}$, we can arrange the blocks as follows:
$$
\begin{array}{|c|c|c|c|c|}
\hline
K_1 & \emptyset & \cdots & \emptyset & \emptyset \\
\hline
\emptyset & K_2 & \cdots & \emptyset & \emptyset \\
\hline
\vdots & \vdots & \ddots & \vdots & \vdots \\
\hline
\emptyset & \emptyset & \cdots & K_k & \emptyset \\
\hline
\emptyset & \emptyset & \cdots & \emptyset & \emptyset \\
\hline
\end{array},
$$
where $\emptyset$ denotes an all-$0$ submatrix, so that there are $r$ rows and $c$ columns.  Call this matrix $M(\mathbf{K})$.  If we permute the rows and columns of this matrix, we generate every $r \times s$ $(0,1)$-matrix $M$ that has $\{C(M_i)\}_{i=1}^k=\mathbf{K}$ some number of times, $\Gamma$ say, by the Orbit-Stabilizer Theorem.  It follows from \eqref{eq:sumchrom} that
\begin{equation}\label{eq:PLSchrom}
\#\PLR(r,s,n;m) = \sum_{k \geq 0} \sum_{ \mathbf{K} \in \mathcal{K}_{r,s,m,k} } \frac{r! s!}{\Gamma}\,\Pi(M(\mathbf{K});n).
\end{equation}
If $\mathbf{K}=\{K_i\}_{i=1}^k$, then $$\Pi(M(\mathbf{K});n)=\prod_{i=1}^k \Pi(K_i;n)$$ since each $K_i$ corresponds to a disjoint component in the induced subgraph of $R_{r,s}$.

Let $e_\text{row}$ and $e_\text{col}$ denote the number of non-empty rows and columns in the matrix $M(\mathbf{K})$, respectively.  If there are $\ell$ distinct matrices in the multiset $\mathbf{K}$, let $k_i$, for $i \in [\ell]$, be the number of copies of the $i$-th distinct matrix.  The elements in the stabilizer of $M(\mathbf{K})$ are those which permute the all-$0$ rows and columns, permute the identical blocks amongst themselves, and stabilize each $K_i$ individually.  Thus since every $K_i \in \mathbf{K}$ is canonical,
\begin{equation}\label{eq:gamma}
\Gamma = (r-e_\text{row})! (s-e_\text{col})! \left( \prod_{i=1}^k |\mathrm{Stab}(K_i)| \right) \left( \prod_{i=1}^{\ell} k_i! \right).
\end{equation}

% Equation \eqref{eq:gamma} is indeed an equality, because every $K_i$ is canonical since generic $(0,1)$-matrices can certainly be equal under row and column permutations without being distinct.

The stabilizer of a $(0,1)$-matrix $M=(M_{ij})$ under row and column permutations is isomorphic to the automorphism group of the vertex-colored bipartite graph $G_M$ with vertex set $$\overbrace{\{\mathfrak{r}_1,\ldots,\mathfrak{r}_r\}}^{\text{color } 1} \cup \overbrace{\{\mathfrak{c}_1,\ldots,\mathfrak{c}_s\}}^{\text{color } 2}$$ and edges $\mathfrak{r}_i \mathfrak{c}_j$ if and only if $M_{ij}=1$.

% We use a method similar to one used in \cite{McKayMeynertMyrvold2007} to describe, for any $(0,1)$-matrix $H$, a graph $G_H$ whose automorphism group is isomorphic to the stabilizer of $H$ under permuting the rows and columns.  Specifically, $G_H$ is the vertex-colored graph with vertex set $$\overbrace{\{R_1,\ldots,R_r\}}^{\text{color } 1} \cup \overbrace{\{C_1,\ldots,C_s\}}^{\text{color } 2} \cup \overbrace{\{S\}}^{\text{color } 3} \cup \overbrace{\{V_{ij}:i \in [r] \text{ and } j \in [s]\}}^{\text{color } 4}$$ and edge set $$\{R_i V_{ij}:i \in [r] \text{ and } j \in [s]\} \cup \{C_j V_{ij}:i \in [r] \text{ and } j \in [s]\} \cup \{S V_{ij}:(i,j) \text{ appears in } H\}.$$  The purpose of this graph is for computational reasons; it enables us to use \texttt{nauty} to compute $$|\mathrm{Stab}(K_i)|=|\mathrm{Aut}(G_{K_i})|$$ in \eqref{eq:gamma}.

To compute $\#\PLR(r,s,n;m)$ for small $m$, we thus:
\begin{itemize}
 \item Generate a list of possible blocks $K$ with up to $m$ ones, inequivalent under row and column permutations, and compute, for each block, the size of $\mathrm{Aut}(G_K)$, and the chromatic polynomial of $K$.  Table~\ref{ta:blocks} lists the results of this computation for $m \leq 5$.
 \item Iterate through each $\mathbf{K} \in \cup_{k \geq 0} \mathcal{K}_{r,s,m,k}$, computing its contribution to \eqref{eq:PLSchrom} from the table generated in the first step.
\end{itemize}
To further reduce the computation, we only store blocks with no more rows than columns.  This requires the modification of \eqref{eq:gamma} to account for transposing the blocks.  By ordering the set of all blocks, a multiset $\mathbf{K}=\{K_i\}_{i=1}^k$ is equivalent to a unique ordered sequence $(K_i)_{i=1}^k$.  We use a $(0,1)$-sequence $(t_i)_{i=1}^k$ to keep track of which $K_i$ we transpose, with $1$ meaning ``transpose'' and $0$ meaning ``don't transpose''.  We define $(t_i)_{i=1}^k$ as \emph{good} if (a) $t_i=0$ whenever $K_i$ is a square matrix, and (b) if $K_i=K_{i+1}$ and $t_i=0$, then $t_{i+1}=0$.  We choose not to transpose square matrices at this stage, as it adds the task of identifying when the transpose of a matrix can be formed by permuting its rows and columns, hence we have condition (a).  Condition (b) prevents overcounting in cases such as
$$
\begin{array}{|ccc|}
\hline
1 & 0 & 0 \\
1 & 0 & 0 \\
0 & 1 & 1 \\
\hline
\end{array}
\qquad\text{and}\qquad
\begin{array}{|ccc|}
\hline
1 & 1 & 0 \\
0 & 0 & 1 \\
0 & 0 & 1 \\
\hline
\end{array}.
$$
Define $$\overline{K_i}=\begin{cases} K_i & \text{if } t_i=0, \\ (K_i)^T & \text{if } t_i=1. \\  \end{cases}$$  Thus, \eqref{eq:PLSchrom} can be rephrased to give the following theorem.

\begin{theo}\label{th:PLRcomp}
For all $m \geq 1$, we have
$$
\#\PLR(r,s,n;m) = \sum_{k \geq 0} \sum_{ \mathbf{K} \in \mathcal{K}_{r,s,m,k} } \sum_{\substack{(t_i)_{i=1}^k \\ \text{\textup{good}}}} [r]_{e_\text{\textrm{row}}} [s]_{e_\text{\textup{col}}} \frac{\prod_{i=1}^k \Pi(\overline{K_i};n)}{\left(\prod_{i=1}^k |\mathrm{Aut}(G_{K_i})| \right) \left( \prod_{i=1}^{\ell} k_i! \right)}
$$
where $[r]_{e_\text{\textup{row}}}=r!/(r-e_\text{\textup{row}})!$ and $[s]_{e_\text{\textup{col}}}=s!/(s-e_\text{\textup{col}})!$.
\end{theo}

% the Cartesian product of K_k and K_s.  (I discuss the non-partial case in my survey paper, under "Rook's graphs".)  Thus, we can find these polynomials by (a) generating all inequivalent shapes with no empty rows/columns, (b) finding the size of the equivalence class via nauty, (c) finding the chromatic polynomial of the corresponding graphs, (d) summing up the product of these (multiplied by the number of ways to embed the shape in an kxs matrix).  I attached the results of my computation for #PLR(n,n,n;m) for m<=12 (which shows your conjecture is true up to m=12).  These results match yours, except there seems to be a typo in item (iv) after Conjecture 23 in your paper.

\section{Sade's method}\label{sec:Sade}

Sade's method \cite{Sade1948a} outstrips all other methods for finding the number of Latin squares \cite{Stones2009b}.  Subsequent authors \cite{BammelRothstein1975,McKayWanless2005,McKayRogoyski1995,Wells1967} who found the number of Latin squares of orders $n\in\{8,9,10,11\}$ implemented optimized computerized versions of Sade's method.  We generalize Sade's method to partial Latin rectangles:

\begin{lemm}
Let $L, M\in\PLR(r,s,n)$.  If
\begin{enumerate}
 \item $L$ and $M$ have the same set of symbols in each column, or
 \item $L$ is isotopic to $M$,
\end{enumerate}
then they can be extended in the same number of ways to $(r+1) \times s$ partial Latin rectangles of a given weight $m$ by adding an $(r+1)$-th row.
\end{lemm}

\begin{proof}
In the first case, the possible $(r+1)$-th rows for $L$ and $M$ are the same.  In the second case, if $M=L^{\Theta}$, then the possible $(r+1)$-th rows of $L^{\Theta}$ are precisely the possible $(r+1)$-th rows of $L$ after applying $\Theta$.
\end{proof}

Let $L, M\in\PLR(r,s,n)$. We say that $L$ and $M$ are \emph{Sade equivalent} if $L$ is isotopic to a partial Latin rectangle $L'$ such that the columns of $L'$ and $M$ have the same sets of symbols. Particularly, $L$ and $M$ must have the same weight. Thus, for instance, the following four partial Latin rectangles in $\PLR(2,3,3;4)$ are Sade equivalent.
\[
\begin{array}{|ccc|}
\hline
1 & 3 & 2 \\
\cdot & 2 & \cdot \\
\hline
\end{array},
\qquad
\begin{array}{|ccc|}
\hline
3 & 1 & 2 \\
2 & \cdot & \cdot \\
\hline
\end{array},
\qquad
\begin{array}{|ccc|}
\hline
1 & 3 & 2 \\
2 & \cdot & \cdot \\
\hline
\end{array}
\qquad \text{and} \qquad
\begin{array}{|ccc|}
\hline
2 & 3 & \cdot \\
1 & \cdot & 2 \\
\hline
\end{array}.
\]

Practically, we need a fast method for checking whether a large number of partial Latin rectangles are Sade equivalent.  To this end, for each partial Latin rectangle $L\in\PLR(r,s,n)$, we perform the following steps, which we illustrate for this example $\PLR(3,4,3;6)$:
\begin{equation}\label{eq:examPLR}
\begin{array}{|cccc|}
\hline
1 & \cdot & 2 & \cdot \\
2 & 1 & 3 & \cdot \\
\cdot & \cdot & \cdot & 3 \\
\hline
\end{array}.
\end{equation}
\begin{enumerate}
 \item Construct a vertex-colored bipartite graph with vertex set $$\overbrace{\{\mathfrak{c}_1,\ldots,\mathfrak{c}_s\}}^{\text{color } 1} \cup \overbrace{\{\mathfrak{n}_1,\ldots,\mathfrak{n}_n\}}^{\text{color } 2}$$ and edge set $$\{\mathfrak{c}_i \mathfrak{n}_j:\text{symbol } j \text{ occurs in column } i \text{ in } L\}.$$  In our running example \eqref{eq:examPLR}, we obtain:

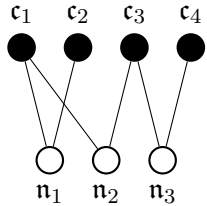
\begin{center}
\begin{tikzpicture}[rotate=-90,scale=1.5]
\tikzstyle{vertex}=[draw,thick,circle,fill=black,minimum size=10pt,inner sep=0pt]
\tikzstyle{vertex2}=[draw,thick,circle,fill=white,minimum size=10pt,inner sep=0pt]

\draw node[vertex,label={$\mathfrak{c}_1$}] (c1) at (0,0) {};
\draw node[vertex,label={$\mathfrak{c}_2$}] (c2) at (0,0.5) {};
\draw node[vertex,label={$\mathfrak{c}_3$}] (c3) at (0,1) {};
\draw node[vertex,label={$\mathfrak{c}_4$}] (c4) at (0,1.5) {};

\draw node[vertex2,label=below:{$\mathfrak{n}_1$}] (s1) at (1,0.25) {};
\draw node[vertex2,label=below:{$\mathfrak{n}_2$}] (s2) at (1,0.75) {};
\draw node[vertex2,label=below:{$\mathfrak{n}_3$}] (s3) at (1,1.25) {};

\draw (c1) -- (s1);
\draw (c1) -- (s2);

\draw (c2) -- (s1);

\draw (c3) -- (s2);
\draw (c3) -- (s3);

\draw (c4) -- (s3);
\end{tikzpicture}
\end{center}

 \item Canonically label the graph in a way that preserves the vertex colors (to this end we use \texttt{nauty} \cite{nauty}, for which such a labeling is an internal procedure).  In our running example, we obtain:
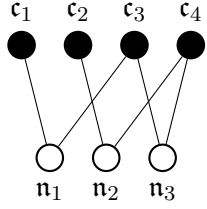
\begin{center}
\begin{tikzpicture}[rotate=-90,scale=1.5]
\tikzstyle{vertex}=[draw,thick,circle,fill=black,minimum size=10pt,inner sep=0pt]
\tikzstyle{vertex2}=[draw,thick,circle,fill=white,minimum size=10pt,inner sep=0pt]

\draw node[vertex,label={$\mathfrak{c}_1$}] (v0) at (0,0) {};
\draw node[vertex,label={$\mathfrak{c}_2$}] (v1) at (0,0.5) {};
\draw node[vertex,label={$\mathfrak{c}_3$}] (v2) at (0,1) {};
\draw node[vertex,label={$\mathfrak{c}_4$}] (v3) at (0,1.5) {};

\draw node[vertex2,label=below:{$\mathfrak{n}_1$}] (v4) at (1,0.25) {};
\draw node[vertex2,label=below:{$\mathfrak{n}_2$}] (v5) at (1,0.75) {};
\draw node[vertex2,label=below:{$\mathfrak{n}_3$}] (v6) at (1,1.25) {};

\draw (v0) -- (v4);
\draw (v1) -- (v5);
\draw (v2) -- (v4);
\draw (v2) -- (v6);
\draw (v3) -- (v5);
\draw (v3) -- (v6);
\end{tikzpicture}
\end{center}
 \item Find the submatrix of the adjacency matrix formed by the rows indexed by $\{\mathfrak{c}_i\}_{i=1}^s$ and columns indexed $\{\mathfrak{n}_i\}_{i=1}^n$, and read it as a binary number.  We call the result the \emph{Sade number}, denoted $sn_L$, of the partial Latin rectangle $L$.  In our running example, we obtain:

\[\begin{array}{|ccccccc|}
\hline
0 & 0 & 0 & 0 & \cellcolor{lightgray} 1 & \cellcolor{lightgray} 0 & \cellcolor{lightgray} 0 \\
0 & 0 & 0 & 0 & \cellcolor{lightgray} 0 & \cellcolor{lightgray} 1 & \cellcolor{lightgray} 0 \\
0 & 0 & 0 & 0 & \cellcolor{lightgray} 1 & \cellcolor{lightgray} 0 & \cellcolor{lightgray} 1 \\
0 & 0 & 0 & 0 & \cellcolor{lightgray} 0 & \cellcolor{lightgray} 1 & \cellcolor{lightgray} 1 \\
1 & 0 & 1 & 0 & 0 & 0 & 0 \\
0 & 1 & 0 & 1 & 0 & 0 & 0 \\
0 & 0 & 1 & 1 & 0 & 0 & 0 \\
\hline
\end{array}\]
which gives the Sade number of \eqref{eq:examPLR} as $100010101011$ in binary, or $2219$ in decimal.
\end{enumerate}

We describe how to implement Sade's method for partial Latin rectangles in Algorithm~\ref{al:Sade}.  Along with a partial Latin rectangle $L$ itself, we store its Sade number $sn_L$ and the number of Sade equivalent partial Latin rectangles, which we call the \emph{Sade multiplier} $sm_L$.  For each $i \in \{0,1,\ldots,s\}$, we maintain a database of Sade inequivalent $i \times s$ partial Latin rectangles $PLRs[i]$.  We compute $PLRs[i]$ by extending the partial Latin rectangles in $PLRs[i-1]$ in all possible ways, then filtering out Sade equivalent partial Latin rectangles.

\begin{algorithm}
\caption{Sade's method for partial Latin rectangles}
\label{al:Sade}
\begin{algorithmic}[1]
\State{Set $PLRs[0]=\langle (L^0,0,1) \rangle$ where $L^0$ is the $0 \times s$ partial Latin rectangle}
\For{$i$ from $1$ to $s$}
  \State{Set $PLRs[i]=\langle \rangle$}
  \ForAll{$(L,sn_L,sm_L) \in PLRs[i-1]$}
    \ForAll{extensions $L^{\text{ext}}$ of $L$ to an $i \times s$ partial Latin rectangle}
      \State{Compute its Sade number $sn_{L^{\text{ext}}}$}
      \State{Binary search for $sn_{L^{\text{ext}}}$ among the Sade numbers in $PLRs[i]$}
      \If{$sn_{L^{\text{ext}}}$ is found}
        \State{We have $sn_{L^{\text{ext}}}=sn_M$ for some $(M,sn_M,sm_M) \in PLRs[i]$}
        \State{Increase $sm_M$ by $sm_L$}
      \Else
        \State{Insert $(L^{\text{ext}},sn_{L^{\text{ext}}},sm_L)$ into $PLRs[i]$}
      \EndIf
    \EndFor
  \EndFor
\EndFor
\end{algorithmic}
\end{algorithm}

Importantly, $PLRs[i]$ is sorted according to Sade numbers.  This enables the use of binary search when checking for equivalent partial Latin rectangles, thereby greatly reducing the number of pairwise equivalence comparisons we need to make.  We iterate through extensions of partial Latin rectangles by use of a backtracking algorithm.

Other practical improvements can be made:
\begin{itemize}
 \item In enumerating up to $8 \times 8$ partial Latin rectangles, the Sade number will be less than $2^{64}$, which can thus be stored as $64$-bit unsigned integers.
 \item When processing the last few rows, we can forgo Sade's method and instead use a simple backtracking algorithm to count the number of extensions up to completions of each partial Latin rectangle.
 \item Although $\#\PLR(r,s,n;m)=\#\PLR(r,n,s;m)$, it is significantly faster to compute the value of $\#\PLR(r,s,n;m)$ when $s \leq n$.
\end{itemize}

\section{Algebraic geometry method}\label{sec:alg_geo}

In this section, we review how sets of partial Latin rectangles are identified with the algebraic sets of certain ideals. This follows the idea of Bayer \cite{Bayer1982} and Adams and Loustaunau \cite{Adams1994} to solve the problem of $n$-coloring a graph by means of algebraic geometry, since every Latin square of order $n$ is equivalent to an $n$-colored bipartite graph $K_{n,n}$ \cite{Laywine1998}. Much more recently, this algebraic method has been adapted to solve sudokus \cite{Arnold2010,Gago2006,Sato2011}, enumerate quasigroup rings derived from partial Latin squares \cite{FalconJCAM2017}, enumerate partial Latin rectangles that admit a given autotopism \cite{Falcon2013,Falcon2015,Falcon2018,FalconMartinMorales2007} or autoparatopism \cite{FalconStones2017}, and also the number of isotopisms between two given partial Latin rectangles \cite{FalconStones2015}, thereby enabling us to compute in Section~\ref{sec:equiv} the numbers of equivalence classes by means of the Orbit-Stabilizer Theorem and Burnside's Lemma.  See \cite{Cox2007, Kreuzer2000} for more details on algebraic geometry.

Let $\mathbb{K}[{\bf x}]=\mathbb{K}[x_1,\ldots,x_n]$ be a polynomial ring in $n$ variables over a field $\mathbb{K}$.  An {\em ideal} of $\mathbb{K}[{\bf x}]$ is any subset $I\subseteq \mathbb{K}[{\bf x}]$ that (a) contains the zero polynomial; (b) is closed under polynomial addition; and (c) is closed under multiplication by polynomials in $q\in \mathbb{K}[{\bf x}]$.  The ideal generated by a finite set of polynomials $\{p_1,\ldots,p_m\} \subseteq \mathbb{K}[{\bf x}]$ is defined as
\[\langle\, p_1,\ldots,p_m \,\rangle := \left\{p\in \mathbb{K}[{\bf x}]\colon p=\textstyle\sum_{i=1}^m q_ip_i,\, \text{ where } q_i\in \mathbb{K}[{\bf x}], \text{ for all } i\leq m\right\}.\]
The {\em algebraic set} of $I$ is the set of points
\[\mathcal{V}(I):=\{(a_1,\ldots,a_n)\in k^n \colon p(a_1,\ldots,a_n)=0, \text{ for all } p\in I\}.\]

For this paper, interest in this topic arises from the polynomial ring $\mathbb{Q}[\mathbf{x}]=\mathbb{Q}[x_{111},\ldots,$ $x_{rsn}]$ and the ideal
\begin{align*}
I_{r,s,n;m} := {} & \langle\,x_{ijk}^2-x_{ijk} \colon  (i,j,k)\in [r]\times [s]\times [n]\,\rangle\\
      & +\langle\, x_{ijk}x_{i'jk}\colon (i,j,k)\in [r]\times [s]\times [n] ,\ i'\in [r],\ i<i'\,\rangle\\
      & +\langle\, x_{ijk}x_{ij'k}\colon (i,j,k)\in [r]\times [s]\times [n] ,\ j'\in [s],\ j<j'\,\rangle\\
      & +\langle\, \,x_{ijk}x_{ijk'}\colon (i,j,k)\in [r]\times [s]\times [n] ,\ k'\in [n],\ k<k'\,\rangle\\
      & +\langle\, \,m-\sum_{i \in [r]} \sum_{j \in [s]} \sum_{k \in [n]}x_{ijk}\,\rangle.
%      & +\langle\, \,m-\sum_{(i,j,k)\in [r]\times [s]\times [n]}x_{ijk}\,\rangle.
\end{align*}

There is a bijection between partial Latin rectangles $L=(l_{ij})\in \PLR(r,s,n;m)$ and elements of the algebraic set of $I_{r,s,n;m}$: we have $l_{ij}=k$ whenever $x_{ijk}=1$, and $l_{ij}$ is undefined otherwise.  More specifically:
\begin{itemize}
 \item Having $x_{ijk}^2-x_{ijk}=0$ implies that the algebraic set is contained in $\{0,1\}^{rsn}$.
 \item Having $x_{ijk}x_{i'jk}=0$ implies that the symbol $k$ does not appear twice in the column $j$.
 \item Having $x_{ijk}x_{ij'k}=0$ implies that the symbol $k$ does not to appear twice in the row $i$.
 \item Having $x_{ijk}x_{ijk'}=0$ implies that there is at most one symbol in the cell $(i,j)$.
 \item Having $\sum_{i \in [r]} \sum_{j \in [s]} \sum_{k \in [n]} x_{ijk}=m$ implies that the weight of the partial Latin rectangle $L$ is $m$.
\end{itemize}
Since the algebraic set $\mathcal{V}(I_{r,s,n;m})$ is finite and the ideal $I_{r,s,n;m}\cap\mathbb{Q}[x_{ijk}]$ is generated by the polynomial $x_{ijk}^2-x_{ijk}$, which is contained in $I_{r,s,n;m}$, Seidenberg's Lemma and \cite[Theorem 3.7.19]{Kreuzer2000} imply $$\#\PLR(r,s,n;m)=|\mathcal{V}(I_{r,s,n;m})|=\mathrm{dim}_{\mathbb{Q}}\big(\mathbb{Q}[{\bf x}]/I_{r,s,n;m}\big).$$
This fact has recently been used in \cite{Falcon2018} to compute $\#\PLR(r,s,n;m)$, for all $r,s,n\leq 6$ and $m\leq rs$.

This algebraic geometry enumeration method can be generalized to include cases in which a certain autoparatopism is imposed as follows.

\begin{theo}\label{theoPLR}
Let $\Theta=(\delta_1,\delta_2,\delta_3)\in\mathfrak{I}_{r,s,n}$ and $\pi\in S_3$. Define
\[I_{(\Theta,\pi);m}:= I_{r,s,n;m} + \langle\,x_{i_1 i_2 i_3} - x_{\delta_{\pi(1)}(i_{\pi(1)})\delta_{\pi(2)}(i_{\pi(2)})\delta_{\pi(3)}(i_{\pi(3)})}\colon i_1\in [r],\, i_2\in [s],\, i_3\in [n]\,\rangle.\]
Then, the set $\PLR((\Theta,\pi);m)$ has a natural bijection with $\mathcal{V}(I_{(\Theta,\pi);m})$ and $$\#\PLR((\Theta,\pi);m)= \mathrm{dim}_{\mathbb{Q}}(\mathbb{Q}[{\bf x}]/I_{(\Theta,\pi);m}).$$
\end{theo}

\begin{proof}
Since $\mathcal{V}(I_{(\Theta,\pi);m})\subseteq \mathcal{V}(I_{r,s,n;m})$, any point in $\mathcal{V}(I_{(\Theta,\pi);m})$ corresponds to a partial Latin rectangle $L=(l_{i_1i_2})\in\PLR(r,s,n;m)$ as previously described. The addition of the polynomials $x_{i_1i_2i_3} - x_{\delta_{\pi(1)}(i_{\pi(1)})\delta_{\pi(2)}(i_{\pi(2)})\delta_{\pi(3)}(i_{\pi(3)})}$ implies that if $l_{i_1i_2}$ is defined and $l_{i_1i_2}=i_3$, then $l_{\delta_{\pi(1)}(i_{\pi(1)})\delta_{\pi(2)}(i_{\pi(2)})}=\delta_{\pi(3)}(i_{\pi(3)})$.  Thus, $(\Theta,\pi)$ is an autoparatopism of $L$.
\end{proof}

A similar approach enables us to determine the set of isotopisms between two partial Latin rectangles as follows.

\begin{theo} \label{theoPLR2} Let $\mathbb{Q}[{\bf x}]=\mathbb{Q}[x_{11},$ $\ldots,x_{rr},y_{11},\ldots,y_{ss},z_{11},\ldots,z_{nn}]$ be a polynomial ring in $r^2+s^2+n^2$ variables. The set $\mathfrak{I}(P,Q)$ of isotopisms between two partial Latin rectangles $P=(p_{ij})$ and $Q=(q_{ij})$ in $\PLR(r,s,n)$ has a natural bijection with the algebraic set of the ideal

\begin{align*}
I_{P,Q}= & \langle\,x_{ij}^2-x_{ij}\colon i,j \in [r]\,\rangle + \langle\,y_{ij}^2-y_{ij}\colon i,j \in [s]\,\rangle + \langle\,z_{ij}^2-z_{ij}\colon i,j \in [n]\,\rangle \\ & + \langle\,1-\sum_{j\in [r]}x_{ij}\colon i\in [r]\,\rangle + \langle\,1-\sum_{i\in [r]}x_{ij}\colon j\in [r]\,\rangle \\ & + \langle\,1-\sum_{j\in [s]}y_{ij}\colon i\in [s]\,\rangle + \langle\,1-\sum_{i\in [s]}y_{ij}\colon j\in [s]\,\rangle \\ & + \langle\,1-\sum_{j\in [n]}z_{ij}\colon i\in [n]\,\rangle +\langle\,1-\sum_{i\in [n]}z_{ij}\colon j\in [n]\,\rangle \\ & + \langle\,x_{ik} y_{jl}(z_{p_{ij}q_{kl}}-1)\colon i,k\in [r],\ j,l\in [s],\text{ such that } p_{ij}, q_{kl}\in [n]\,\rangle \\ & + \langle\,x_{ik}y_{jl} \colon i,k\in [r],\ j,l \in [s],\ p_{ij} \text{ and/or } q_{kl} \text{ is undefined}\,\rangle.
\end{align*}
Consequently, the number of isotopisms from $P$ to $Q$ is given by \[\#\mathfrak{I}(P,Q)= \mathrm{dim}_{\mathbb{Q}}(\mathbb{Q}[{\bf x}]/I_{P,Q}).\]
\end{theo}

\begin{proof} The first three subideals of $I_{P,Q}$ imply $\mathcal{V}(I_{P,Q})\subseteq \{0,1\}^{r^2+s^2+n^2}$. Every isotopism $\Theta=(\alpha,\beta,\gamma)\in\mathfrak{I}_{r,s,n}$ uniquely corresponds to a zero $(x^{\Theta}_{11},\ldots,x^{\Theta}_{rr},$ $y^{\Theta}_{11},\ldots, y^{\Theta}_{ss},z^{\Theta}_{11},\ldots,z^{\Theta}_{nn})$ in $\mathcal{V}(I_{P,Q})$, where $x^{\Theta}_{ij}=1$, (resp., $y^{\Theta}_{ij}=1$ and $z^{\Theta}_{ij}=1$) if $\alpha(i)=j$ (resp., $\beta(i)=j$ and $\gamma(i)=j$) and $0$, otherwise. Specifically, the fourth and fifth subideals of $I_{P,Q}$ (in the statement of Theorem~\ref{theoPLR2}) imply $\alpha$ is a permutation of $S_r$ (the fourth one ensures the injectivity, while the fifth one ensures surjectivity). Similarly, the next two pairs of subideals imply $\beta$ and $\gamma$ are permutations of $S_s$ and $S_n$, respectively, and hence, $\Theta$ is an isotopism of $\PLR(r,s,n)$. The last two subideals imply $\Theta$ is a bijection between the entry sets $E(P)$ and $E(Q)$. Further, since $I_{P,Q}\cap\mathbb{Q}[x_{ijk}]=\langle\,x_{ijk}^2-x_{ijk}\,\rangle \subseteq I_{P,Q}$, Seidenberg's Lemma and Theorem 3.7.19 in \cite{Kreuzer2000} imply the theorem statement.
\end{proof}

\section{Counting equivalence classes of partial Latin rectangles}\label{sec:equiv}

Theorem \ref{theoPLR2} can be used to determine not only the size of the autotopism group $\mathfrak{I}(P,P)$ of a partial Latin rectangle $P$, but also that of its autoparatopism group $\mathfrak{P}(P,P)$, because
\[\#\mathfrak{P}(P,P)=\sum_{\pi\in S_3}\#\mathfrak{I}(P,P^{\pi}).\]
The following result shows how the computation of both values enables us to determine the size of the isotopism and main classes containing $P$ by means of the Orbit-Stabilizer Theorem.

\begin{theo}\label{th:isomisot2} Let $P\in\PLR(r,s,n)$. Then,
\begin{enumerate}
\item the number of partial Latin rectangles that are isotopic to $P$, i.e., the size of the isotopism class containing $P$, is
$$\frac {r!\ s!\ n!}{\#\mathfrak{I}(P,P)};$$
\item the number of partial Latin rectangles that are paratopic to $P$, i.e., the size of the main class containing $P$, is
$$\# S_{r,s,n}\ \frac {r!\ s!\ n!}{\#\mathfrak{P}(P,P)};$$ and
\item the number of isotopism classes in the main class of $P$ is
$$\# S_{r,s,n}\ \frac{\#\mathfrak{I}(P,P)}{\#\mathfrak{P}(P,P)}.$$
\end{enumerate}
\end{theo}

\begin{proof}
The first two claims follow from the Orbit-Stabilizer Theorem.  For the third claim, we observe that paratopic partial Latin rectangles have autotopism groups of the same size, because $\Theta$ is an autotopism of $P$ if and only if $\Lambda^{-1}\Theta\Lambda$ is an autotopism of $P^{\Lambda}$ for any paratopism $\Lambda$.  They thus also have isotopism classes of the same size, which partition the main class, so the first two claims imply the third.
\end{proof}

From here on, let $\mathrm{Isom}(n;m)$, $\mathrm{Isot}(r,s,n;m)$ and $\mathrm{MC}(r,s,n;m)$ respectively denote the set of isomorphism classes of $\PLS(n;m)$ and the sets of isotopism and main classes of $\PLR(r,s,n;m)$. The following result follows straightforwardly from Burnside's Lemma and the fact of acting the isomorphism, isotopism, and paratopism groups on the set of partial Latin rectangles of a given order.

\begin{theo}\label{theo:isomisot} Let $r,s,n \geq 1$ and $m\leq rs$. Then,
\begin{enumerate}
\item the number of isomorphism classes in $\PLS(n;m)$ is
\[\#\mathrm{Isom}(n;m)=\frac 1{n!}\sum_{\pi\in S_n} \#\PLR(((\pi,\pi,\pi),\mathrm{Id});m);\]
\item the number of isotopism classes in $\PLR(r,s,n;m)$ is
\[\#\mathrm{Isot}(r,s,n;m) = \frac 1{r!\,s!\,n!}\sum_{\Theta\in \mathfrak{I}_{r,s,n}} \#\PLR((\Theta,\mathrm{Id});m);\] and
\item the number of main classes in $\PLR(r,s,n;m)$ is
\[\#\mathrm{MC}(r,s,n;m) =\frac 1{r!\,s!\,n!\,\#S_{r,s,n}} \sum_{(\Theta,\pi)\in \mathfrak{P}_{r,s,n}} \#\PLR((\Theta,\pi);m).\]
\end{enumerate}
\end{theo}

In practice, it is not necessary to compute all the values $\#\PLR((\Theta,\pi);m)$ in order to determine $\#\mathrm{Isom}(n;m)$, $\#\mathrm{Isot}(r,s,n;m)$, or $\#\mathrm{MC}(r,s,n;m)$ in Theorem~\ref{theo:isomisot}. The following lemma implies that $\#\PLR((\Theta,\pi);m)$ only depends on the conjugacy class of the corresponding paratopism. Recall that two permutations $\alpha,\beta\in S_n$ are {\em conjugate} if there exists a third permutation $\gamma\in S_n$ such that $\alpha=\gamma^{-1}\beta\gamma$, which naturally generalizes to isotopisms and paratopisms under componentwise conjugacy. In what follows, conjugacy is denoted $\sim$.

\begin{lemm}\label{lmm_Conjugacy}
Let $(\Theta_1,\pi_1)$ and $(\Theta_2,\pi_2)$ be two conjugate paratopisms in $\mathfrak{P}_{r,s,n}$. Then, \[\#\PLR((\Theta_1,\pi_1);m)=\#\PLR((\Theta_2,\pi_2);m).\]
\end{lemm}

\begin{proof}
Since $(\Theta_1,\pi_1)\sim (\Theta_2,\pi_2)$, there exists $(\Theta_3,\pi_3)\in \mathfrak{P}_{r,s,n}$ such that $(\Theta_2,\pi_2)=(\Theta_3,\pi_3)^{-1}$ $(\Theta_1,\pi_1)(\Theta_3,\pi_3)$.  Let $L\in\PLR((\Theta_1,\pi_1);m)$. Then, $(L^{(\Theta_3,\pi_3)})^{(\Theta_2,\pi_2)}=(L^{(\Theta_1,\pi_1)})^{(\Theta_3,\pi_3)}=L^{(\Theta_3,\pi_3)}$ and so, $L^{(\Theta_3,\pi_3)}\in \PLR((\Theta_2,\pi_2);m)$. The result holds because $\PLR((\Theta_2,\pi_2);m)=\{L^{(\Theta_3,\pi_3)}\colon L\in \PLR((\Theta_1,\pi_1);m)\}$.
\end{proof}

The following result is shown using similar reasoning to that used by Mendis and Wanless in the proof of Theorem~2.2 in \cite{Mendis2016} for paratopisms of Latin squares.

\begin{theo}\label{theo:Conjugacy_MW} Two paratopisms $((\alpha_1,\alpha_2,\alpha_3),\pi_1)$ and $((\beta_1,\beta_2,\beta_3),\pi_2)$ in $\mathfrak{P}_{r,s,n}$ are conjugate if and only if there is a length preserving bijection $\eta$ from the cycles of $\pi_1$ to those of $\pi_2$ such that, if $\eta$ maps a cycle $(a_1,\ldots,a_k)$ to a cycle $(b_1,\ldots,b_k)$, both of them in the symmetric group $S_3$, then  $\alpha_{a_1}\cdots\alpha_{a_k}\sim\beta_{b_1}\cdots\beta_{b_k}$.

As a consequence, any paratopism $((\alpha,\beta,\gamma),(12))\in \mathfrak{P}_{r,r,n}$ is conjugate to $((\mathrm{Id},
\alpha\beta,\gamma),(12))\in \mathfrak{P}_{r,r,n}$, and any $((\alpha,\beta,\gamma),(123))\in \mathfrak{P}_{r,r,r}$ is conjugate to both $((\mathrm{Id},
\mathrm{Id},\alpha\beta\gamma),(123))\in \mathfrak{P}_{r,r,r}$ and $((\mathrm{Id},
\mathrm{Id},\alpha\beta\gamma),(132))\in \mathfrak{P}_{r,r,r}$.
\end{theo}

\begin{proof}
If $r=s=n$, or if $\pi_1 = \pi_2$, the proof of Theorem~2.2 in \cite{Mendis2016} suffices the prove the theorem.  Otherwise, up to equivalence, we have $\pi_1=(12)$ and $\pi_2=\mathrm{Id}$ and $r=s \neq n$.  Clearly, $\eta$ does not exist.  The two paratopisms are not conjugate since conjugation in $\mathfrak{P}_{r,s,n}$ preserves the conjugacy class of the parastrophe permutation.
\end{proof}

\begin{theo}\label{theo:Conjugacy}
For any paratopism $((\alpha,\beta,\gamma),(12))\in \mathfrak{P}_{r,r,n}$, we have
\begin{align*}
    & \#\PLR(((\alpha,\beta,\gamma),(12));m) = \#\PLR(((\beta,\alpha,\gamma),(12));m) \\
={} & \#\PLR(((\alpha,\gamma,\beta),(13));m) = \#\PLR(((\beta,\gamma,\alpha),(13));m) \\
={} & \#\PLR(((\gamma,\alpha,\beta),(23));m) = \#\PLR(((\gamma,\beta,\alpha),(23));m).
\end{align*}
Moreover, any $((\alpha,\beta,\gamma),(12)) \in \mathfrak{P}_{r,r,n}$ is conjugate to a paratopism $((\mathrm{Id},\beta',\gamma'), (12))$, where $\alpha\beta\sim\beta'$ and $\gamma\sim\gamma'$.  If $\pi$ is a $3$-cycle, then any paratopism $((\alpha,\beta,\gamma),\pi)\in \mathfrak{P}_{r,r,r}$ is conjugate to a paratopism $((\mathrm{Id},\mathrm{Id},\gamma),(123)) \in \mathfrak{P}_{r,s,n}$, where $\alpha\beta\gamma\sim\gamma'$.
\end{theo}

\begin{proof}
We have $((\alpha,\beta,\gamma),(12))
\in \mathrm{Atop}(L)$ if and only if $((\alpha,\gamma,\beta),(13))\in \mathrm{Atop}(L^{(23)})$.  Consequently, $\#\PLR(((\alpha,\beta,\gamma),(12));m) = \#\PLR(((\alpha,\gamma,\beta),(13));m)$.  Lemma~\ref{lmm_Conjugacy} implies \[\#\PLR(((\alpha,\beta,\gamma),(12));m) = \#\PLR(((\beta,\alpha,\gamma),(12));m).\]
And we similarly prove the other equalities.  Theorem~\ref{theo:Conjugacy_MW} implies the conjugacy claims.
\end{proof}

Conjugacy in symmetric groups constitutes an equivalence relation in which each conjugacy class is characterized by the common cycle structure of their elements. Recall that the {\em cycle structure} of a permutation $\pi \in S_m$ is the expression $z_{\pi}:=m^{d_m^{\pi}}\cdots 1^{d_1^{\pi}}$, where $d_i^{\pi}$ denotes the number of cycles of length $i$ in the unique cycle decomposition of the permutation $\pi$.  Thus, for instance, the cycle structure of the permutation $(12)(345)(78)(9)$ is $3^12^21^1$. From here on, we denote the set of cycle structures of the symmetric group $S_m$ as $\mathcal{CS}_m$. The number of permutations in $S_m$ with cycle structure $m^{d_m}\cdots 1^{d_1}\in \mathcal{CS}_m$
\begin{equation}\label{eq:permcyc}
\frac{m!}{\prod_{i \in [m]} d_i!\,i^{d_i}},
\end{equation}
as in \cite[Theorem~13.2]{Andrews1984}.
Further, the {\em cycle structure} of an isotopism $(\alpha,\beta,\gamma)\in \mathfrak{I}_{r,s,n}$ is defined as the triple $(z_{\alpha},z_{\beta},z_{\gamma})$ formed by the respective cycle structures of $\alpha$, $\beta$, and $\gamma$.  Keeping in mind Lemma~\ref{lmm_Conjugacy} and Theorem~\ref{theo:Conjugacy}, the following values are well-defined:
\begin{itemize}
\item
$\Delta_m(z_1,z_2,z_3):=\#\PLR((\Theta,\mathrm{Id});m)$, for any $\Theta\in \mathfrak{I}_{r,s,n}$ with cycle structure $(z_1,z_2,z_3)$;
\item $\Delta^{(12)}_m(z_2,z_3):=\#\PLR((\Theta,(12));m)$, for any $\Theta=(\alpha,\beta,\gamma)\in \mathfrak{I}_{r,s,n}$ where $r=s$, such that $z_{\alpha\beta}=z_2$ and $z_{\gamma}=z_3$; and
\item $\Delta^{(123)}_m(z_3):=\#\PLR((\Theta,(123));m)$, for any $\Theta=(\alpha,\beta,\gamma)\in \mathfrak{I}_{r,s,n}$ where $r=s=n$, such that $z_{\alpha\beta\gamma}=z_3$.
\end{itemize}

Given a cycle structure $z\in \mathcal{CS}_m$, define $d_i^z:=d_i^{\pi}$ for any permutation $\pi\in S_m$ with cycle structure $z$. The next theorem follows straightforwardly from Theorem~\ref{theo:isomisot}, Lemma \ref{lmm_Conjugacy} and \eqref{eq:permcyc}.

\begin{theo}\label{th:isomisot} Let $r,s,n \geq 1$ and $m\leq rs$. Then,
\begin{enumerate}
\item the number of isomorphism classes in $\PLS(n;m)$ is
\[\#\mathrm{Isom}(n;m) =\sum_{z\in\mathcal{CS}_n}\frac {\Delta_m(z,z,z)}{\prod_{i\in [n]}d_i^z!\,i^{d_i^z}},\] and
\item\label{iso_classes} the number of isotopism classes in $\PLR(r,s,n;m)$ is
\[\#\mathrm{Isot}(r,s,n;m) = \sum_{\substack{z_1\in\mathcal{CS}_r\\z_2\in\mathcal{CS}_s\\ z_3\in\mathcal{CS}_n}}\frac {\Delta_m(z_1,z_2,z_3)}{\prod_{\substack{i\in [r]\\j\in [s]\\k\in [n]}}\,d_i^{z_1}!\,d_j^{z_2}!\,d_k^{z_3}!\, i^{d_i^{z_1}}\,j^{d_j^{z_2}}\,k^{d_k^{z_3}}}.\]
\end{enumerate}
\end{theo}

In practice, it is not necessary perform computations for all possible triples $(z_1,z_2,z_3)\in \mathcal{CS}_r\times\mathcal{CS}_s\times\mathcal{CS}_n$ to determine the number of isotopism classes in statement~\ref{iso_classes} of Theorem~\ref{th:isomisot}. The following lemma gives necessary and sufficient conditions for $\PLR((\Theta,\mathrm{Id}))$ to contain a non-empty partial Latin rectangle. This generalizes in a natural way a pair of similar results referred to Latin squares \cite[Lemma 3.6]{StonesVojtechovskyWanless2012} and partial Latin squares \cite[Lemma 2.2]{Falcon2013}.

\begin{lemm} \label{lemmLCM} Let $\Theta\in \mathfrak{I}_{r,s,n}$ be an isotopism of cycle structure $(z_1,z_2,z_3)\in \mathcal{CS}_r\times \mathcal{CS}_s \times \mathcal{CS}_n$. The set $\PLR((\Theta,\mathrm{Id}))$ contains at least one non-empty partial Latin rectangle if and only if there exists a triple $(i,j,k)\in [r]\times [s]\times [n]$ such that the following two conditions are satisfied:
\begin{enumerate}
\item $\mathrm{lcm}(i,j)=\mathrm{lcm}(i,k)=\mathrm{lcm}(j,k)= \mathrm{lcm}(i,j,k)$, and
\item $z_1$ has an $i$-cycle, $z_2$ has a $j$-cycle, and $z_3$ has a $k$-cycle.
\end{enumerate}
\end{lemm}

Parastrophisms preserve the number of isotopism and main classes of partial Latin rectangles of a given order. Thus, in practice, it is enough to focus on the case $r\leq s\leq n$ to determine the number of isotopism classes in $\PLR(r,s,n;m)$, whereas the number of main classes splits into three cases: (a) $r<s<n$; (b) $r=s<n$; and (c) $r=s=n$. In (a), the parastrophism group $S_{r,s,n}$ is only formed by the trivial permutation $\mathrm{Id}\in S_3$ and hence, the number of main classes coincides with that of isotopism classes. In order to deal with (b) and (c), and keeping in mind Theorem~\ref{theo:Conjugacy}, let us define the following two sets for each pair of permutations $\beta,\gamma\in S_r$

\[C_1(\beta,\gamma):=\{(\alpha,\beta',\gamma')\in \mathfrak{I}_{r,r,n}\colon \alpha\beta'\sim\beta \text{ and } \gamma'\sim\gamma\}.\]
\[C_2(\gamma):=\{(\alpha,\beta,\gamma')\in \mathfrak{I}_{r,r,r}\colon \alpha\beta\gamma'\sim\gamma\}.\]

The next result holds straightforwardly from Theorem \ref{theo:isomisot}, Lemma \ref{lmm_Conjugacy}, and \eqref{eq:permcyc}.

\begin{theo}\label{theo:MC} Let $r,n\geq 1$ and $m\leq r^2$. The following statements hold.
\begin{enumerate}
\item If $n\neq r$, then the number of main classes in $\mathcal{PLR}(r,r,n;m)$ is
\[\#\mathrm{MC}(r,r,n;m)=\frac {\#\mathrm{Isot}(r,r,n;m)}2 + \frac 1{2\,r!^2\,n!}\sum_{(\beta,\gamma)\in S_s\times S_n} |C_1(\beta)|\, \Delta^{(12)}_m(z_{\beta},z_{\gamma}).\]
\item The number of main classes in $\mathcal{PLR}(r,r,r;m)$ is
\begin{align*}
\#\mathrm{MC}(r,r,r;m)= {} & \frac {\#\mathrm{Isot}(r,r,r;m)}6 + \frac 1{2\,r!^3}\sum_{(\beta,\gamma)\in S_s\times S_n} |C_1(\beta)|\, \Delta^{(12)}_m(z_{\beta},z_{\gamma})\\
 & + \frac 1{3\,r!^3}\sum_{\gamma\in S_n} |C_2(\gamma)|\, \Delta^{(123)}_m(z_{\gamma}).
\end{align*}
\end{enumerate}
\end{theo}

\section{Computational results}\label{sec:comp_results}

\paragraph{\bf Inclusion-exclusion method}
For small graphs $G$, Tables~\ref{ta:polyP1} and~\ref{ta:polyP2} list the polynomial $P(G)=P(G;r,s,n)$ in Theorem~\ref{th:isosplit}.  These polynomials were computed using a C++ program, using \texttt{geng} (packaged with \texttt{nauty} \cite{McKay1981,nauty,McKayPiperno2014}) to generate a list of isolated-vertex-free non-isomorphic graphs (e.g.\ ``\texttt{geng -d1 3}'' generates $3$-vertex isomorphism class representatives with minimum degree $1$) and \texttt{bliss} \cite{JunttilaKaski2007} to compute their automorphism group size.  The notation $\overline{abc}$ is shorthand for the sum of the monic monomials with variables $r$, $s$, and $n$ and exponents $a$, $b$, and $c$.  For example, $\overline{210}=r^2s+r^2n+s^2r+s^2n+n^2r+n^2s$ and $2\,\overline{100}=2(r+s+n)$.

By Lemma~\ref{lm:maxdeg}, substituting the data in Tables~\ref{ta:polyP1} and~\ref{ta:polyP2} into the formula in Theorem~\ref{th:isosplit} gives a formula for $f_m(r,s,n)$ containing all terms of degree $\geq 3m-9$; unlisted graphs $G$ have $v-c(G) \geq 5$, and thus contribute to terms in the polynomial with degree at most $3m-10$. In this regard, the following result generalizes Theorem 4.7 in \cite{Falcon2018}, which only deals with the case $m\leq 6$.

%Original position for Tables {ta:polyP1} and {ta:polyP2}.

\begin{theo}\label{th:asmpt}
Let $m$ be a positive integer. Then, $
f_m(r,s,n)=
(rsn)^m
+ \binom{m}{2} (rsn)^{m-1}(2-\overline{100})
+ \binom{m}{3} (rsn)^{m-2}(14-12\,\overline{100}+6\,\overline{110}+2\,\overline{200})
+ \binom{m}{4} (rsn)^{m-3}(198-228\,\overline{100}+198\,\overline{110}-84\,\overline{111}+72\,\overline{200}-36\,\overline{210}-12\,\overline{211}+6\,\overline{221}-6\,\overline{300}+3\,\overline{311})
+ \binom{m}{5} (rsn)^{m-4}(-6360\,\overline{100}+7440\,\overline{110}-6080\,\overline{111}+2880\,\overline{200}-2520\,\overline{210}+820\,\overline{211}+480\,\overline{220}+360\,\overline{221}-180\,\overline{222}-480\,\overline{300}+240\,\overline{310}+160\,\overline{311}-80\,\overline{321}+24\,\overline{400}-20\,\overline{411})
+ \binom{m}{6} (rsn)^{m-5}(-13170\,\overline{211}+17340\,\overline{221}-15990\,\overline{222}+7580\,\overline{311}-7050\,\overline{321}+3300\,\overline{322}+1520\,\overline{331}+180\,\overline{332}-90\,\overline{333}-1740\,\overline{411}+870\,\overline{421}+90\,\overline{422}-45\,\overline{432}+130\,\overline{511}-15\,\overline{522})
+ \binom{m}{7} (rsn)^{m-6}(-10920\,\overline{322}+15540\,\overline{332}-15120\,\overline{333}+7350\,\overline{422}-7140\,\overline{432}+3570\,\overline{433}+1680\,\overline{442}-2100\,\overline{522}+1050\,\overline{532}+210\,\overline{622})
+ \binom{m}{8} (rsn)^{m-7}(-3360\,\overline{433}+5040\,\overline{443}-5040\,\overline{444}+2520\,\overline{533}-2520\,\overline{543}+1260\,\overline{544}+630\,\overline{553}-840\,\overline{633}+420\,\overline{643}+105\,\overline{733})
\ +$ some polynomial of degree $\leq 3m-10$.
\end{theo}
Theorem~\ref{th:asmpt} is exact for $m\leq 5$ since Tables~\ref{ta:polyP1} and~\ref{ta:polyP2} contain all graphs with no isolated vertices with up to $5$ vertices, and when $v \geq 6$, graphs make a zero contribution to Theorem~\ref{th:isosplit} since the binomial $\binom{m}{v}=0$.

\paragraph{\bf Chromatic polynomial method}
We use Theorem~\ref{th:PLRcomp} to compute exact formulas for $f_m(r,s,n)$ for all $m \leq 13$ which are available from \cite{FalconStonesMendeley}. They corroborate in particular the formulas shown in \cite{Falcon2018} for $m\leq 6$. The authors acknowledge the use of \texttt{GAP} \cite{GAP}, the \texttt{GAP} package \texttt{GRAPE} \cite{GRAPE} (which uses \texttt{nauty}), and the Tutte polynomial software \texttt{tutte\_bhkk} \cite{BjoklundHusfeldtKaskiKoivisto2008} (available from \url{github.com/thorehusfeldt/tutte_bhkk}) for these computations.

%Original position for Table {ta:blocks}

\paragraph{\bf Sade's method}
We implement Algorithm~\ref{al:Sade} in C++ using \texttt{nauty} for graph isomorphism and \text{GMP} \cite{GMP} for arbitrary precision arithmetic, which we use to compute $\#\PLR(r,s,n;m)$ for all $r,s,n \leq 7$, and for $r,s \leq 6$ when $n=8$ (for all $0 \leq m \leq rs$). Our computations for $r,s,n \leq 6$ corroborate Tables 2 through 5 in \cite{Falcon2018}. The remaining cases are listed here in Tables~\ref{ta:rs7a} through~\ref{ta:rs8a}.

%Original position for Table {ta:SadeExample}

\paragraph{\bf Algebraic geometry method}

We implement Theorem \ref{theoPLR} in {\tt Singular} \cite{Decker2015} and {\tt Minion} \cite{Gent2006} to determine the values $\Delta(z_1,z_2,z_3)$, for all $(z_1,z_2,z_3)\in \mathcal{CS}_r\times \mathcal{CS}_s \times \mathcal{CS}_n$ satisfying the conditions of Lemma \ref{lemmLCM}, when $r,s,n\leq 6$. Theorem~\ref{th:isomisot} has then be applied to obtain the corresponding numbers of isomorphism and isotopism classes of partial Latin rectangles, as listed in Tables \ref{tableIsom} through \ref{tableIsot2}.

The number of main classes of partial Latin rectangles in $\PLR(r,s,n)$ according to their weights is given in Tables \ref{tableMainClasses} and \ref{tableMainClasses1} when $2 \leq r\leq s\leq n \leq 6$. We include only the cases in which $r$, $s$, and $n$ are not pairwise distinct; otherwise, the number of main classes and isotopism classes coincide.

%Original position for Tables {tableIsom}, {tableIsot}, {tableIsota}, {tableIsot2}, {tableMainClasses}, {tableMainClasses1}

\paragraph{\bf Constructive enumeration}

It is also possible to enumerate constructively the number of isotopism and main classes in $\PLR(r,s,n;m)$.  We simply extend all representative weight-$(m-1)$ partial Latin rectangles by one entry in all possible ways, and throw away those that belong to the same class as an already discovered partial Latin rectangle.  To compare isotopism and main class equivalence, it is enough, for instance, to generate a graph similar to those proposed in \cite{McKayMeynertMyrvold2007}.

For fixed $m \geq 1$, provided $r \geq m$, $s \geq m$, and $n \geq m$, the number of isotopism classes and main classes in the set $\PLR(r,s,n;m)$ do not vary with $r$, $s$ and $n$, which amounts to adding empty rows, empty columns, or unused symbols.  To compute these numbers, we use the above constructive method, but allow the possibility of introducing new rows, columns, and/or symbols when extending weight-$(m-1)$ partial Latin rectangles.  We perform this enumeration for $m \leq 11$, and the results are given in Table~\ref{ta:unbounded}.  The results for main classes is consistent with those independently obtained in \cite{DietrichWanless2018, WanlessWebb2017}, and moreover, \cite{DietrichWanless2018} also computes the number of main classes for $m=12$.

Direct constructive enumeration of isomorphism classes is infeasible, since the numbers grow too quickly.  Moreover, isotopic partial Latin rectangles may have different-sized isomorphism classes\footnote{In the context of Latin squares, this led to \cite[Th.~2(i)]{McKayMeynertMyrvold2007} being false; which is acknowledged in a corrected version of \cite{McKayMeynertMyrvold2007} on McKay's website \url{http://users.cecs.anu.edu.au/~bdm/papers/ls_final_corr.pdf}.}, so we cannot easily derive the number of isomorphism classes within a isotopism class, which thwarts modifying the approach we use for enumerating isotopism classes to enumerating isomorphism classes.  Instead, using an algebraic geometry method like in Section~\ref{sec:alg_geo}, we enumerate isomorphism classes for $m \leq 6$ in \cite[Table 2]{FalconStones2015}.

%Original position for Table {ta:unbounded}

\section{Verification}\label{sec:verify}

The authors have made efforts to ensure the numbers and formulas presented here are as bug-free as possible; we document these efforts in this section. First, the various source codes used and their output are available from \cite{FalconStonesMendeley}.  Next, where feasible, computations have been independently performed, using different techniques and different software.  Where possible, we have also cross-checked the results of the enumeration methods.
\begin{itemize}
 \item The computation of $\#\PLR(r,s,n;m)$ for all $r,s,n \leq 7$ has been performed using both the algebraic geometry method (except for $(s,n) \in \{(6,7),(7,7)\}$) and Sade's method.
 \item The number of isotopism classes and main classes has been computed using both the algebraic geometry method (for $r,s,n \leq 6$) and constructive enumeration (for $r,s,n \leq 5$).
 \item For $m \leq 13$, the results of the computation of $\#\PLR(r,s,n;m)$ have been cross-checked against the computed polynomials $f_m(r,s,n)$.  Thus for $854$ quadruples $(r,s,n;m)$ the computations agreed exactly.
\end{itemize}

In addition to cross-checking computational results, we check the divisibility of the numbers computed using the following theorem.  More specifically, we check the exact formulas for $f_m(r,s,n)$, for $m \leq 13$, satisfy Theorem~\ref{th:div} whenever $k \in \{1,\ldots,10\}$ and $r,s,n \in \{k+1,\ldots,k+10\}$.

\begin{theo}\label{th:div}
For all $r,s,n,m \geq 1$ and $k \geq 0$, we have
\begin{align*}
\#\PLR(r,s,n;m) & \equiv \#\PLR(k,s,n;m) \pmod {r-k} \text{ when } r \geq k+1, \\
\#\PLR(r,s,n;m) & \equiv \#\PLR(r,k,n;m) \pmod {s-k} \text{ when } s \geq k+1, \text{and} \\
\#\PLR(r,s,n;m) & \equiv \#\PLR(r,s,k;m) \pmod {n-k} \text{ when } n \geq k+1.
\end{align*}
\end{theo}

\begin{proof}
Firstly, we prove the second claim.  We act on $\PLR(r,s,n;m)$ by permuting the columns using the group $G$ of isotopisms generated by $(\mathrm{Id},\beta,\mathrm{Id})$ for an $(m-k)$-cycle $\beta$.  By the Orbit-Stabilizer Theorem, orbits have size $m-k$ unless they contain partial Latin rectangles that admit a non-trivial autotopism in $G$.  This is only possible if the columns permuted by $\beta$ are empty.  Hence, the orbits of size less than $m-k$ together form $\PLR(r,k,n;m)$ by deleting the columns permuted by $\beta$.  The same argument works for rows and symbols, which gives the first and third claims.
\end{proof}

% \begin{corol}\mbox{}
% \begin{itemize}
%  \item If $m>k\min(s,n)$, then $r-k$ divides $\#\PLR(r,s,n;m)$.
%  \item If $m>k\min(r,n)$, then $s-k$ divides $\#\PLR(r,s,n;m)$.
%  \item If $m>k\min(r,s)$, then $n-k$ divides $\#\PLR(r,s,n;m)$.
% \end{itemize}
% \end{corol}
%
% \begin{proof}
% We apply Theorem~\ref{th:div}, and note that in the first case $\#\PLR(k,s,n;m)=0$, and so on for the other cases.
% \end{proof}

\section*{Acknowledgements} Stones was supported by her NSFC Research Fellowship for International Young Scientists (grant numbers: 11450110409, 11550110491), NSF China grant 61170301, and the Thousand Youth Talents Plan in Tianjin.  Stones would also like to acknowledge the use of \url{math.stackexchange.com} for discussing problems arising in this work.  Thanks to Zhuanhao Wu for parallelizing and improving the memory management in Stones's C++ code.  Thanks also to Daniel Kotlar for assistance in computing autotopism groups, which is leading to the papers \cite{FalconKotlarStones2,FalconKotlarStones}

\bibliographystyle{amsplain}
\bibliography{LatinB4}

\appendix
\section{Glossary of symbols}\label{sec:symbols}

\noindent \begin{tabular}{ll}
$S_n$ & The symmetric group on $n$ elements.\\
$\mathfrak{I}_{r,s,n}$ & The isotopism group $S_r\times S_s\times S_n$.\\
$S_{r,s,n}$ & The parastrophism group defined in Section~\ref{sec:intro}. \\
$\mathfrak{P}_{r,s,n}$ & The paratopism group $S_r\times S_s\times S_n\rtimes S_{r,s,n}$.\\
$\PLR(r,s,n)$ & The set of $r\times s$ partial Latin rectangles on $[n]\cup\{\cdot\}$.\\
$\PLR(r,s,n;m)$ & The subset of partial Latin rectangles in $\PLR(r,s,n)$ of weight $m$.\\
$\PLR((\Theta,\pi))$ & The set of partial Latin rectangles having $(\Theta,\pi)$ as autoparatopism.\\
$\PLR((\Theta,\pi);m)$ & The subset of partial Latin rectangles in $\PLR((\Theta,\pi))$ of weight $m$.\\
$\PLS(n;m)$ & The set of partial Latin squares of order $n$ and weight $m$.\\
$\mathrm{Isom}(n;m)$ & The set of isomorphism classes of $\PLS(n;m)$.\\
$\mathrm{Isot}(r,s,n;m)$ & The set of isotopism classes of $\PLR(r,s,n;m)$.\\
$\mathrm{MC}(r,s,n;m)$ & The set of main classes of $\PLR(r,s,n;m)$.\\
\end{tabular}

\section{Tables}\label{sec:tables}

%AAAA --- LINK TO A FULL VERSION ON ARXIV --- AAAA

% The length of the paper should not exceed 25 pages (in 11-point LaTeX format on US letter-size paper with 1-inch margins). Authors of longer papers are advised to submit a 25-page version to COMBINATORICA with a link to a full version on arXiv.

\def\blockone{
\tikz[scale=1.5]{
\coordinate (P1) at (0,0);
\fill (P1) circle(2pt);
}
}

\def\blocktwo{
\tikz[scale=1.5]{
\coordinate (P1) at (0,0);
\coordinate (P2) at (\smallgraphnodedistance,0);
\draw (P1) -- (P2);
\fill (P1) circle(2pt);
\fill (P2) circle(2pt);
}
}

\def\blockthree{
\tikz[scale=1.5]{
\coordinate (P1) at (0,0);
\coordinate (P2) at (\smallgraphnodedistance,0);
\coordinate (P3) at (2*\smallgraphnodedistance,0);
\draw (P1) -- (P2) -- (P3);
\draw (P1) to[bend left=60] (P3);
\fill (P1) circle(2pt);
\fill (P2) circle(2pt);
\fill (P3) circle(2pt);
}
}

\def\blockfour{
\tikz[scale=1.5]{
\coordinate (P1) at (0,0);
\coordinate (P2) at (\smallgraphnodedistance,0);
\coordinate (P3) at (0,-\smallgraphnodedistance);
\draw (P2) -- (P1) -- (P3);
\fill (P1) circle(2pt);
\fill (P2) circle(2pt);
\fill (P3) circle(2pt);
}
}

\def\blockfive{
\tikz[scale=1.5]{
\coordinate (P1) at (0,0);
\coordinate (P2) at (\smallgraphnodedistance,0);
\coordinate (P3) at (2*\smallgraphnodedistance,0);
\coordinate (P4) at (3*\smallgraphnodedistance,0);
\draw (P1) -- (P2) -- (P3) -- (P4);
\draw (P1) to[bend left=60] (P3);
\draw (P2) to[bend left=60] (P4);
\draw (P1) to[bend left=60] (P4);
\fill (P1) circle(2pt);
\fill (P2) circle(2pt);
\fill (P3) circle(2pt);
\fill (P4) circle(2pt);
}
}

\def\blocksix{
\tikz[scale=1.5]{
\coordinate (P1) at (0,0);
\coordinate (P2) at (\smallgraphnodedistance,0);
\coordinate (P3) at (2*\smallgraphnodedistance,0);
\coordinate (P4) at (0,-\smallgraphnodedistance);
\draw (P1) -- (P2) -- (P3);
\draw (P1) to[bend left=60] (P3);
\draw (P1) -- (P4);
\fill (P1) circle(2pt);
\fill (P2) circle(2pt);
\fill (P3) circle(2pt);
\fill (P4) circle(2pt);
}
}

\def\blockseven{
\tikz[scale=1.5]{
\coordinate (P1) at (0,0);
\coordinate (P2) at (\smallgraphnodedistance,0);
\coordinate (P3) at (0,-\smallgraphnodedistance);
\coordinate (P4) at (2*\smallgraphnodedistance,-\smallgraphnodedistance);
\draw (P1) -- (P2);
\draw (P1) -- (P3) -- (P4);
\fill (P1) circle(2pt);
\fill (P2) circle(2pt);
\fill (P3) circle(2pt);
\fill (P4) circle(2pt);
}
}

\def\blockeight{
\tikz[scale=1.5]{
\coordinate (P1) at (0,0);
\coordinate (P2) at (\smallgraphnodedistance,0);
\coordinate (P3) at (0,-\smallgraphnodedistance);
\coordinate (P4) at (\smallgraphnodedistance,-\smallgraphnodedistance);
\draw (P1) -- (P2);
\draw (P1) -- (P3) -- (P4) -- (P2);
\fill (P1) circle(2pt);
\fill (P2) circle(2pt);
\fill (P3) circle(2pt);
\fill (P4) circle(2pt);
}
}

\def\blocknine{
\tikz[scale=1.5]{
\coordinate (P1) at (0,0);
\coordinate (P2) at (\smallgraphnodedistance,0);
\coordinate (P3) at (2*\smallgraphnodedistance,0);
\coordinate (P4) at (3*\smallgraphnodedistance,0);
\coordinate (P5) at (4*\smallgraphnodedistance,0);
\draw (P1) -- (P2) -- (P3) -- (P4) -- (P5);
\draw (P1) to[bend left=60] (P3);
\draw (P2) to[bend left=60] (P4);
\draw (P3) to[bend left=60] (P5);
\draw (P1) to[bend left=60] (P4);
\draw (P2) to[bend left=60] (P5);
\draw (P1) to[bend left=60] (P5);
\fill (P1) circle(2pt);
\fill (P2) circle(2pt);
\fill (P3) circle(2pt);
\fill (P4) circle(2pt);
\fill (P5) circle(2pt);
}
}

\def\blockten{
\tikz[scale=1.5]{
\coordinate (P1) at (0,0);
\coordinate (P2) at (\smallgraphnodedistance,0);
\coordinate (P3) at (2*\smallgraphnodedistance,0);
\coordinate (P4) at (3*\smallgraphnodedistance,0);
\coordinate (P5) at (0,-\smallgraphnodedistance);
\draw (P5) -- (P1) -- (P2) -- (P3) -- (P4);
\draw (P1) to[bend left=60] (P3);
\draw (P2) to[bend left=60] (P4);
\draw (P1) to[bend left=60] (P4);
\fill (P1) circle(2pt);
\fill (P2) circle(2pt);
\fill (P3) circle(2pt);
\fill (P4) circle(2pt);
\fill (P5) circle(2pt);
}
}

\def\blockeleven{
\tikz[scale=1.5]{
\coordinate (P1) at (0,0);
\coordinate (P2) at (\smallgraphnodedistance,0);
\coordinate (P3) at (2*\smallgraphnodedistance,0);
\coordinate (P4) at (3*\smallgraphnodedistance,-\smallgraphnodedistance);
\coordinate (P5) at (0,-\smallgraphnodedistance);
\draw (P5) -- (P1) -- (P2) -- (P3);
\draw (P1) to[bend left=60] (P3);
\draw (P5) -- (P4);
\fill (P1) circle(2pt);
\fill (P2) circle(2pt);
\fill (P3) circle(2pt);
\fill (P4) circle(2pt);
\fill (P5) circle(2pt);
}
}

\def\blocktwelve{
\tikz[scale=1.5]{
\coordinate (P1) at (0,0);
\coordinate (P2) at (\smallgraphnodedistance,0);
\coordinate (P3) at (2*\smallgraphnodedistance,0);
\coordinate (P4) at (\smallgraphnodedistance,-\smallgraphnodedistance);
\coordinate (P5) at (0,-\smallgraphnodedistance);
\draw (P5) -- (P1) -- (P2) -- (P3);
\draw (P1) to[bend left=60] (P3);
\draw (P5) -- (P4) -- (P2);
\fill (P1) circle(2pt);
\fill (P2) circle(2pt);
\fill (P3) circle(2pt);
\fill (P4) circle(2pt);
\fill (P5) circle(2pt);
}
}

\def\blockthirteen{
\tikz[scale=1.5]{
\coordinate (P1) at (0,0);
\coordinate (P2) at (\smallgraphnodedistance,0);
\coordinate (P3) at (2*\smallgraphnodedistance,0);
\coordinate (P4) at (0,-\smallgraphnodedistance);
\coordinate (P5) at (0,-2*\smallgraphnodedistance);
\draw (P1) -- (P2) -- (P3);
\draw (P1) to[bend left=60] (P3);
\draw (P1) -- (P4) -- (P5);
\draw (P1) to[bend right=60] (P5);
\fill (P1) circle(2pt);
\fill (P2) circle(2pt);
\fill (P3) circle(2pt);
\fill (P4) circle(2pt);
\fill (P5) circle(2pt);
}
}

\def\blockfourteen{
\tikz[scale=1.5]{
\coordinate (P1) at (0,0);
\coordinate (P2) at (\smallgraphnodedistance,0);
\coordinate (P3) at (2*\smallgraphnodedistance,0);
\coordinate (P4) at (0,-\smallgraphnodedistance);
\coordinate (P5) at (\smallgraphnodedistance,-2*\smallgraphnodedistance);
\draw (P1) -- (P2) -- (P3);
\draw (P1) to[bend left=60] (P3);
\draw (P1) -- (P4);
\draw (P2) -- (P5);
\fill (P1) circle(2pt);
\fill (P2) circle(2pt);
\fill (P3) circle(2pt);
\fill (P4) circle(2pt);
\fill (P5) circle(2pt);
}
}

\def\blockfifteen{
\tikz[scale=1.5]{
\coordinate (P1) at (0,0);
\coordinate (P2) at (\smallgraphnodedistance,0);
\coordinate (P3) at (0,-\smallgraphnodedistance);
\coordinate (P4) at (2*\smallgraphnodedistance,-\smallgraphnodedistance);
\coordinate (P5) at (0,-2*\smallgraphnodedistance);
\draw (P1) -- (P2);
\draw (P1) -- (P3) -- (P5);
\draw (P1) to[bend right=60] (P5);
\draw (P3) -- (P4);
\fill (P1) circle(2pt);
\fill (P2) circle(2pt);
\fill (P3) circle(2pt);
\fill (P4) circle(2pt);
\fill (P5) circle(2pt);
}
}

\def\blocksixteen{
\tikz[scale=1.5]{
\coordinate (P1) at (0,0);
\coordinate (P2) at (\smallgraphnodedistance,0);
\coordinate (P3) at (0,-\smallgraphnodedistance);
\coordinate (P4) at (2*\smallgraphnodedistance,-\smallgraphnodedistance);
\coordinate (P5) at (\smallgraphnodedistance,-2*\smallgraphnodedistance);
\draw (P1) -- (P2) -- (P5);
\draw (P1) -- (P3) -- (P4);
\fill (P1) circle(2pt);
\fill (P2) circle(2pt);
\fill (P3) circle(2pt);
\fill (P4) circle(2pt);
\fill (P5) circle(2pt);
}
}

\begin{table}[htp]
\centering
\resizebox{0.8\textwidth}{!}{
$\begin{array}{ccll}
\hline
\text{block } K & \text{induced subgraph} & |\mathrm{Aut}(G_K)| & \Pi(K;n) \\
\hline
\rule{0pt}{3ex}
\begin{array}{|c|} \hline 1 \\ \hline \end{array} \rule[-1.7ex]{0pt}{0ex} & \blockone & 1 & n \\
\hline
\rule{0pt}{3ex}
\begin{array}{|cc|} \hline 1 & 1 \\ \hline \end{array} \rule[-1.7ex]{0pt}{0ex} & \blocktwo & 2 & n(n-1) \\
\hline
\rule{0pt}{3ex}
\begin{array}{|ccc|} \hline 1 & 1 & 1 \\ \hline \end{array} & \blockthree & 6 & n(n-1)(n-2) \\[0.5em]
\begin{array}{|cc|} \hline 1 & 1 \\ 1 & 0 \\ \hline \end{array} \rule[-3.2ex]{0pt}{0ex} & \raisebox{-0.5em}{\blockfour} & 1 & n(n-1)^2 \\
\hline
\begin{array}{|cccc|} \hline 1 & 1 & 1 & 1 \\ \hline \end{array} & \blockfive & 24 & n(n-1)(n-2)(n-3) \\[0.5em]
\begin{array}{|ccc|} \hline 1 & 1 & 1 \\ 1 & 0 & 0 \\ \hline \end{array} & \raisebox{-0.5em}{\blocksix} & 2 & n(n-1)^2(n-2) \\[1em]
\begin{array}{|ccc|} \hline 1 & 1 & 0 \\ 1 & 0 & 1 \\ \hline \end{array} & \raisebox{-0.5em}{\blockseven} & 2 & n(n-1)^3 \\[1em]
\begin{array}{|cc|} \hline 1 & 1 \\ 1 & 1 \\ \hline \end{array} \rule[-3.2ex]{0pt}{0ex} & \raisebox{-0.5em}{\blockeight} & 4 & n(n-1)(n^2-3n+3) \\
\hline
\begin{array}{|ccccc|} \hline 1 & 1 & 1 & 1 & 1 \\ \hline \end{array} & \blocknine & 120 & n(n-1)(n-2)(n-3)(n-4) \\[0.5em]
\begin{array}{|cccc|} \hline 1 & 1 & 1 & 1 \\ 1 & 0 & 0 & 0 \\ \hline \end{array} & \raisebox{-0.5em}{\blockten} & 6 & n(n-1)^2(n-2)(n-3) \\[1em]
\begin{array}{|cccc|} \hline 1 & 1 & 1 & 0 \\ 1 & 0 & 0 & 1 \\ \hline \end{array} & \raisebox{-0.5em}{\blockeleven} & 2 & n(n-1)^3(n-2) \\[1em]
\begin{array}{|ccc|} \hline 1 & 1 & 1 \\ 1 & 1 & 0 \\ \hline \end{array} & \raisebox{-0.5em}{\blocktwelve} & 2 & n(n-1)(n-2)(n^2-3n+3) \\[1em]
\begin{array}{|ccc|} \hline 1 & 1 & 1 \\ 1 & 0 & 0 \\ 1 & 0 & 0 \\ \hline \end{array} & \raisebox{-1em}{\blockthirteen} & 4 & n(n-1)^2(n-2)^2 \\[1.7em]
\begin{array}{|ccc|} \hline 1 & 1 & 1 \\ 1 & 0 & 0 \\ 0 & 1 & 0 \\ \hline \end{array} & \raisebox{-1em}{\blockfourteen} & 2 & n(n-1)^3(n-2) \\[1.7em]
\begin{array}{|ccc|} \hline 1 & 1 & 0 \\ 1 & 0 & 1 \\ 1 & 0 & 0 \\ \hline \end{array} & \raisebox{-1em}{\blockfifteen} & 2 & n(n-1)^3(n-2) \\[1.7em]
\begin{array}{|ccc|} \hline 1 & 1 & 0 \\ 1 & 0 & 1 \\ 0 & 1 & 0 \\ \hline \end{array} \rule[-4.8ex]{0pt}{0ex} & \raisebox{-1em}{\blocksixteen} & 1 & n(n-1)^4 \\
\hline
\end{array}$
}
\caption{The blocks $K$ with $\leq 5$ ones, along with $|\mathrm{Aut}(G_K)|$ and its chromatic polynomial $\Pi(K;n)$.}\label{ta:blocks}
\end{table}

\def\Ktwo{
\tikz{
\coordinate (P1) at (0,0);
\coordinate (P2) at (\smallgraphnodedistance,0);
\draw[thick] (P1) -- (P2);
\fill (P1) circle(2pt);
\fill (P2) circle(2pt);
}
}

\def\Ptwo{
\tikz{
\coordinate (P1) at (0,0);
\coordinate (P2) at (\smallgraphnodedistance,0);
\coordinate (P3) at (0.5*\smallgraphnodedistance,\smallgraphnodedistance);
\draw[thick] (P1) -- (P3) -- (P2);
\fill (P1) circle(2pt);
\fill (P2) circle(2pt);
\fill (P3) circle(2pt);
}
}

\def\Kthree{
\tikz{
\coordinate (P1) at (0,0);
\coordinate (P2) at (\smallgraphnodedistance,0);
\coordinate (P3) at (0.5*\smallgraphnodedistance,\smallgraphnodedistance);
\draw[thick] (P1) -- (P2) -- (P3) -- (P1);
\fill (P1) circle(2pt);
\fill (P2) circle(2pt);
\fill (P3) circle(2pt);
}
}

\def\TwoKtwo{
\tikz{
\coordinate (P1) at (0,0);
\coordinate (P2) at (\smallgraphnodedistance,0);
\coordinate (P3) at (0.5*\smallgraphnodedistance,\smallgraphnodedistance);
\coordinate (P4) at (1.5*\smallgraphnodedistance,\smallgraphnodedistance);
\draw[thick] (P1) -- (P3);
\draw[thick] (P2) -- (P4);
\fill (P1) circle(2pt);
\fill (P2) circle(2pt);
\fill (P3) circle(2pt);
\fill (P4) circle(2pt);
}
}

\def\Pthree{
\tikz{
\coordinate (P1) at (0,0);
\coordinate (P2) at (\smallgraphnodedistance,0);
\coordinate (P3) at (0.5*\smallgraphnodedistance,\smallgraphnodedistance);
\coordinate (P4) at (1.5*\smallgraphnodedistance,\smallgraphnodedistance);
\draw[thick] (P1) -- (P3) -- (P2) -- (P4);
\fill (P1) circle(2pt);
\fill (P2) circle(2pt);
\fill (P3) circle(2pt);
\fill (P4) circle(2pt);
}
}

\def\ThreeStar{
\tikz{
\coordinate (P1) at (0,0);
\coordinate (P2) at (\smallgraphnodedistance,0);
\coordinate (P3) at (0.5*\smallgraphnodedistance,\smallgraphnodedistance);
\coordinate (P4) at (1.5*\smallgraphnodedistance,\smallgraphnodedistance);
\draw[thick] (P1) -- (P3);
\draw[thick] (P2) -- (P3);
\draw[thick] (P4) -- (P3);
\fill (P1) circle(2pt);
\fill (P2) circle(2pt);
\fill (P3) circle(2pt);
\fill (P4) circle(2pt);
}
}

\def\KthreeLeaf{
\tikz{
\coordinate (P1) at (0,0);
\coordinate (P2) at (\smallgraphnodedistance,0);
\coordinate (P3) at (0.5*\smallgraphnodedistance,\smallgraphnodedistance);
\coordinate (P4) at (1.5*\smallgraphnodedistance,\smallgraphnodedistance);
\draw[thick] (P1) -- (P2) -- (P3) -- (P1);
\draw[thick] (P4) -- (P3);
\fill (P1) circle(2pt);
\fill (P2) circle(2pt);
\fill (P3) circle(2pt);
\fill (P4) circle(2pt);
}
}

\def\Cfour{
\tikz{
\coordinate (P1) at (0,0);
\coordinate (P2) at (\smallgraphnodedistance,0);
\coordinate (P3) at (0.5*\smallgraphnodedistance,\smallgraphnodedistance);
\coordinate (P4) at (1.5*\smallgraphnodedistance,\smallgraphnodedistance);
\draw[thick] (P1) -- (P3) -- (P4) -- (P2) -- (P1);
\fill (P1) circle(2pt);
\fill (P2) circle(2pt);
\fill (P3) circle(2pt);
\fill (P4) circle(2pt);
}
}

\def\KfourMinusEdge{
\tikz{
\coordinate (P1) at (0,0);
\coordinate (P2) at (\smallgraphnodedistance,0);
\coordinate (P3) at (0.5*\smallgraphnodedistance,\smallgraphnodedistance);
\coordinate (P4) at (1.5*\smallgraphnodedistance,\smallgraphnodedistance);
\draw[thick] (P1) -- (P2);
\draw[thick] (P1) -- (P3);
\draw[thick] (P1) -- (P4);
\draw[thick] (P2) -- (P4);
\draw[thick] (P3) -- (P4);
\fill (P1) circle(2pt);
\fill (P2) circle(2pt);
\fill (P3) circle(2pt);
\fill (P4) circle(2pt);
}
}

\def\Kfour{
\tikz{
\coordinate (P1) at (0,0);
\coordinate (P2) at (\smallgraphnodedistance,0);
\coordinate (P3) at (0.5*\smallgraphnodedistance,\smallgraphnodedistance);
\coordinate (P4) at (1.5*\smallgraphnodedistance,\smallgraphnodedistance);
\draw[thick] (P1) -- (P2) -- (P3) -- (P4) -- (P1);
\draw[thick] (P1) -- (P3);
\draw[thick] (P2) -- (P4);
\fill (P1) circle(2pt);
\fill (P2) circle(2pt);
\fill (P3) circle(2pt);
\fill (P4) circle(2pt);
}
}

\def\PtwoKtwo{
\tikz{
\coordinate (P1) at (0,0);
\coordinate (P2) at (\smallgraphnodedistance,0);
\coordinate (P3) at (0.5*\smallgraphnodedistance,\smallgraphnodedistance);
\coordinate (P4) at (1.5*\smallgraphnodedistance,\smallgraphnodedistance);
\coordinate (P5) at (2*\smallgraphnodedistance,0);
\draw[thick] (P1) -- (P3) -- (P2);
\draw[thick] (P4) -- (P5);
\fill (P1) circle(2pt);
\fill (P2) circle(2pt);
\fill (P3) circle(2pt);
\fill (P4) circle(2pt);
\fill (P5) circle(2pt);
}
}

\def\KthreeKtwo{
\tikz{
\coordinate (P1) at (0,0);
\coordinate (P2) at (\smallgraphnodedistance,0);
\coordinate (P3) at (0.5*\smallgraphnodedistance,\smallgraphnodedistance);
\coordinate (P4) at (1.5*\smallgraphnodedistance,\smallgraphnodedistance);
\coordinate (P5) at (2*\smallgraphnodedistance,0);
\draw[thick] (P1) -- (P2) -- (P3) -- (P1);
\draw[thick] (P4) -- (P5);
\fill (P1) circle(2pt);
\fill (P2) circle(2pt);
\fill (P3) circle(2pt);
\fill (P4) circle(2pt);
\fill (P5) circle(2pt);
}
}

\def\ThreeKtwo{
\tikz{
\coordinate (P1) at (0,0);
\coordinate (P2) at (\smallgraphnodedistance,0);
\coordinate (P3) at (0.5*\smallgraphnodedistance,\smallgraphnodedistance);
\coordinate (P4) at (1.5*\smallgraphnodedistance,\smallgraphnodedistance);
\coordinate (P5) at (2*\smallgraphnodedistance,0);
\coordinate (P6) at (2.5*\smallgraphnodedistance,\smallgraphnodedistance);
\draw[thick] (P1) -- (P3);
\draw[thick] (P2) -- (P4);
\draw[thick] (P5) -- (P6);
\fill (P1) circle(2pt);
\fill (P2) circle(2pt);
\fill (P3) circle(2pt);
\fill (P4) circle(2pt);
\fill (P5) circle(2pt);
\fill (P6) circle(2pt);
}
}

\def\FiveNodeTreeOne{
\tikz{
\coordinate (P1) at (0,0);
\coordinate (P2) at (\smallgraphnodedistance,0);
\coordinate (P3) at (0.5*\smallgraphnodedistance,\smallgraphnodedistance);
\coordinate (P4) at (1.5*\smallgraphnodedistance,\smallgraphnodedistance);
\coordinate (P5) at (2*\smallgraphnodedistance,0);
\draw[thick] (P2) -- (P1);
\draw[thick] (P2) -- (P3);
\draw[thick] (P2) -- (P4);
\draw[thick] (P2) -- (P5);
\fill (P1) circle(2pt);
\fill (P2) circle(2pt);
\fill (P3) circle(2pt);
\fill (P4) circle(2pt);
\fill (P5) circle(2pt);
}
}

\def\FiveNodeTreeTwo{
\tikz{
\coordinate (P1) at (0,0);
\coordinate (P2) at (\smallgraphnodedistance,0);
\coordinate (P3) at (0.5*\smallgraphnodedistance,\smallgraphnodedistance);
\coordinate (P4) at (1.5*\smallgraphnodedistance,\smallgraphnodedistance);
\coordinate (P5) at (2*\smallgraphnodedistance,0);
\draw[thick] (P3) -- (P1);
\draw[thick] (P2) -- (P3);
\draw[thick] (P2) -- (P4);
\draw[thick] (P2) -- (P5);
\fill (P1) circle(2pt);
\fill (P2) circle(2pt);
\fill (P3) circle(2pt);
\fill (P4) circle(2pt);
\fill (P5) circle(2pt);
}
}

\def\FiveNodePath{
\tikz{
\coordinate (P1) at (0,0);
\coordinate (P2) at (\smallgraphnodedistance,0);
\coordinate (P3) at (0.5*\smallgraphnodedistance,\smallgraphnodedistance);
\coordinate (P4) at (1.5*\smallgraphnodedistance,\smallgraphnodedistance);
\coordinate (P5) at (2*\smallgraphnodedistance,0);
\draw[thick] (P1) -- (P3);
\draw[thick] (P2) -- (P3);
\draw[thick] (P2) -- (P4);
\draw[thick] (P4) -- (P5);
\fill (P1) circle(2pt);
\fill (P2) circle(2pt);
\fill (P3) circle(2pt);
\fill (P4) circle(2pt);
\fill (P5) circle(2pt);
}
}

\def\TriangleCherry{
\tikz{
\coordinate (P1) at (0,0);
\coordinate (P2) at (\smallgraphnodedistance,0);
\coordinate (P3) at (0.5*\smallgraphnodedistance,\smallgraphnodedistance);
\coordinate (P4) at (1.5*\smallgraphnodedistance,\smallgraphnodedistance);
\coordinate (P5) at (2*\smallgraphnodedistance,0);
\draw[thick] (P2) -- (P1);
\draw[thick] (P2) -- (P3);
\draw[thick] (P1) -- (P3);
\draw[thick] (P2) -- (P4);
\draw[thick] (P2) -- (P5);
\fill (P1) circle(2pt);
\fill (P2) circle(2pt);
\fill (P3) circle(2pt);
\fill (P4) circle(2pt);
\fill (P5) circle(2pt);
}
}

\def\FourCycPendant{
\tikz{
\coordinate (P1) at (0,0);
\coordinate (P2) at (\smallgraphnodedistance,0);
\coordinate (P3) at (0.5*\smallgraphnodedistance,\smallgraphnodedistance);
\coordinate (P4) at (1.5*\smallgraphnodedistance,\smallgraphnodedistance);
\coordinate (P5) at (2*\smallgraphnodedistance,0);
\draw[thick] (P1) -- (P2);
\draw[thick] (P2) -- (P4);
\draw[thick] (P4) -- (P3);
\draw[thick] (P3) -- (P1);
\draw[thick] (P2) -- (P5);
\fill (P1) circle(2pt);
\fill (P2) circle(2pt);
\fill (P3) circle(2pt);
\fill (P4) circle(2pt);
\fill (P5) circle(2pt);
}
}

\def\ThreeCycTwoPendant{
\tikz{
\coordinate (P1) at (0,0);
\coordinate (P2) at (\smallgraphnodedistance,0);
\coordinate (P3) at (0.5*\smallgraphnodedistance,\smallgraphnodedistance);
\coordinate (P4) at (1.5*\smallgraphnodedistance,\smallgraphnodedistance);
\coordinate (P5) at (2*\smallgraphnodedistance,0);
\draw[thick] (P1) -- (P2);
\draw[thick] (P2) -- (P3);
\draw[thick] (P3) -- (P1);
\draw[thick] (P2) -- (P5);
\draw[thick] (P3) -- (P4);
\fill (P1) circle(2pt);
\fill (P2) circle(2pt);
\fill (P3) circle(2pt);
\fill (P4) circle(2pt);
\fill (P5) circle(2pt);
}
}

\def\KfourMinusEdgePendantOne{
\tikz{
\coordinate (P1) at (0,0);
\coordinate (P2) at (\smallgraphnodedistance,0);
\coordinate (P3) at (0.5*\smallgraphnodedistance,\smallgraphnodedistance);
\coordinate (P4) at (1.5*\smallgraphnodedistance,\smallgraphnodedistance);
\coordinate (P5) at (2*\smallgraphnodedistance,0);
\draw[thick] (P1) -- (P2);
\draw[thick] (P1) -- (P3);
\draw[thick] (P2) -- (P3);
\draw[thick] (P2) -- (P4);
\draw[thick] (P3) -- (P4);
\draw[thick] (P2) -- (P5);
\fill (P1) circle(2pt);
\fill (P2) circle(2pt);
\fill (P3) circle(2pt);
\fill (P4) circle(2pt);
\fill (P5) circle(2pt);
}
}

\def\KTwoThree{
\tikz{
\coordinate (P1) at (0,0);
\coordinate (P2) at (\smallgraphnodedistance,0);
\coordinate (P3) at (0.5*\smallgraphnodedistance,\smallgraphnodedistance);
\coordinate (P4) at (1.5*\smallgraphnodedistance,\smallgraphnodedistance);
\coordinate (P5) at (2*\smallgraphnodedistance,0);
\draw[thick] (P3) -- (P1);
\draw[thick] (P3) -- (P2);
\draw[thick] (P3) -- (P5);
\draw[thick] (P4) -- (P1);
\draw[thick] (P4) -- (P2);
\draw[thick] (P4) -- (P5);
\fill (P1) circle(2pt);
\fill (P2) circle(2pt);
\fill (P3) circle(2pt);
\fill (P4) circle(2pt);
\fill (P5) circle(2pt);
}
}

\def\KTwoThreePlusEdgeOne{
\tikz{
\coordinate (P1) at (0,0);
\coordinate (P2) at (\smallgraphnodedistance,0);
\coordinate (P3) at (0.5*\smallgraphnodedistance,\smallgraphnodedistance);
\coordinate (P4) at (1.5*\smallgraphnodedistance,\smallgraphnodedistance);
\coordinate (P5) at (2*\smallgraphnodedistance,0);
\draw[thick] (P3) -- (P1);
\draw[thick] (P3) -- (P2);
\draw[thick] (P3) -- (P5);
\draw[thick] (P4) -- (P1);
\draw[thick] (P4) -- (P2);
\draw[thick] (P4) -- (P5);
\draw[thick] (P3) -- (P4);
\fill (P1) circle(2pt);
\fill (P2) circle(2pt);
\fill (P3) circle(2pt);
\fill (P4) circle(2pt);
\fill (P5) circle(2pt);
}
}

\def\TriangleTwoPath{
\tikz{
\coordinate (P1) at (0,0);
\coordinate (P2) at (\smallgraphnodedistance,0);
\coordinate (P3) at (0.5*\smallgraphnodedistance,\smallgraphnodedistance);
\coordinate (P4) at (1.5*\smallgraphnodedistance,\smallgraphnodedistance);
\coordinate (P5) at (2*\smallgraphnodedistance,0);
\draw[thick] (P1) -- (P3);
\draw[thick] (P2) -- (P3);
\draw[thick] (P1) -- (P2);
\draw[thick] (P2) -- (P4);
\draw[thick] (P4) -- (P5);
\fill (P1) circle(2pt);
\fill (P2) circle(2pt);
\fill (P3) circle(2pt);
\fill (P4) circle(2pt);
\fill (P5) circle(2pt);
}
}

\def\Bowtie{
\tikz{
\coordinate (P1) at (0,0);
\coordinate (P2) at (\smallgraphnodedistance,0);
\coordinate (P3) at (0.5*\smallgraphnodedistance,\smallgraphnodedistance);
\coordinate (P4) at (1.5*\smallgraphnodedistance,\smallgraphnodedistance);
\coordinate (P5) at (2*\smallgraphnodedistance,0);
\draw[thick] (P1) -- (P3);
\draw[thick] (P2) -- (P3);
\draw[thick] (P1) -- (P2);
\draw[thick] (P2) -- (P4);
\draw[thick] (P4) -- (P5);
\draw[thick] (P2) -- (P5);
\fill (P1) circle(2pt);
\fill (P2) circle(2pt);
\fill (P3) circle(2pt);
\fill (P4) circle(2pt);
\fill (P5) circle(2pt);
}
}

\def\Cfive{
\tikz{
\coordinate (P1) at (0,0);
\coordinate (P2) at (\smallgraphnodedistance,0);
\coordinate (P3) at (0.5*\smallgraphnodedistance,\smallgraphnodedistance);
\coordinate (P4) at (1.5*\smallgraphnodedistance,\smallgraphnodedistance);
\coordinate (P5) at (2*\smallgraphnodedistance,0);
\draw[thick] (P1) -- (P2);
\draw[thick] (P2) -- (P5);
\draw[thick] (P5) -- (P4);
\draw[thick] (P4) -- (P3);
\draw[thick] (P3) -- (P1);
\fill (P1) circle(2pt);
\fill (P2) circle(2pt);
\fill (P3) circle(2pt);
\fill (P4) circle(2pt);
\fill (P5) circle(2pt);
}
}

\def\House{
\tikz{
\coordinate (P1) at (0,0);
\coordinate (P2) at (\smallgraphnodedistance,0);
\coordinate (P3) at (0.5*\smallgraphnodedistance,\smallgraphnodedistance);
\coordinate (P4) at (1.5*\smallgraphnodedistance,\smallgraphnodedistance);
\coordinate (P5) at (2*\smallgraphnodedistance,0);
\draw[thick] (P1) -- (P3);
\draw[thick] (P3) -- (P4);
\draw[thick] (P1) -- (P2);
\draw[thick] (P2) -- (P4);
\draw[thick] (P4) -- (P5);
\draw[thick] (P2) -- (P5);
\fill (P1) circle(2pt);
\fill (P2) circle(2pt);
\fill (P3) circle(2pt);
\fill (P4) circle(2pt);
\fill (P5) circle(2pt);
}
}

\def\House{
\tikz{
\coordinate (P1) at (0,0);
\coordinate (P2) at (\smallgraphnodedistance,0);
\coordinate (P3) at (0.5*\smallgraphnodedistance,\smallgraphnodedistance);
\coordinate (P4) at (1.5*\smallgraphnodedistance,\smallgraphnodedistance);
\coordinate (P5) at (2*\smallgraphnodedistance,0);
\draw[thick] (P1) -- (P3);
\draw[thick] (P3) -- (P4);
\draw[thick] (P1) -- (P2);
\draw[thick] (P2) -- (P4);
\draw[thick] (P4) -- (P5);
\draw[thick] (P2) -- (P5);
\fill (P1) circle(2pt);
\fill (P2) circle(2pt);
\fill (P3) circle(2pt);
\fill (P4) circle(2pt);
\fill (P5) circle(2pt);
}
}

\def\KTwoThreePlueEdgeTwo{
\tikz{
\coordinate (P1) at (0,0);
\coordinate (P2) at (\smallgraphnodedistance,0);
\coordinate (P3) at (0.5*\smallgraphnodedistance,\smallgraphnodedistance);
\coordinate (P4) at (1.5*\smallgraphnodedistance,\smallgraphnodedistance);
\coordinate (P5) at (2*\smallgraphnodedistance,0);
\draw[thick] (P1) -- (P2) -- (P5);
\draw[thick] (P1) -- (P3) -- (P5);
\draw[thick] (P1) -- (P4) -- (P5);
\draw[thick] (P2) -- (P4);
\fill (P1) circle(2pt);
\fill (P2) circle(2pt);
\fill (P3) circle(2pt);
\fill (P4) circle(2pt);
\fill (P5) circle(2pt);
}
}

\def\KfourMinusEdgePendantTwo{
\tikz{
\coordinate (P1) at (0,0);
\coordinate (P2) at (\smallgraphnodedistance,0);
\coordinate (P3) at (0.5*\smallgraphnodedistance,\smallgraphnodedistance);
\coordinate (P4) at (1.5*\smallgraphnodedistance,\smallgraphnodedistance);
\coordinate (P5) at (2*\smallgraphnodedistance,0);
\draw[thick] (P1) -- (P2);
\draw[thick] (P1) -- (P3);
\draw[thick] (P1) -- (P4);
\draw[thick] (P2) -- (P4);
\draw[thick] (P3) -- (P4);
\draw[thick] (P2) -- (P5);
\fill (P1) circle(2pt);
\fill (P2) circle(2pt);
\fill (P3) circle(2pt);
\fill (P4) circle(2pt);
\fill (P5) circle(2pt);
}
}

\def\KfourPendant{
\tikz{
\coordinate (P1) at (0,0);
\coordinate (P2) at (\smallgraphnodedistance,0);
\coordinate (P3) at (0.5*\smallgraphnodedistance,\smallgraphnodedistance);
\coordinate (P4) at (1.5*\smallgraphnodedistance,\smallgraphnodedistance);
\coordinate (P5) at (2*\smallgraphnodedistance,0);
\draw[thick] (P1) -- (P2);
\draw[thick] (P1) -- (P3);
\draw[thick] (P1) -- (P4);
\draw[thick] (P2) -- (P4);
\draw[thick] (P3) -- (P4);
\draw[thick] (P2) -- (P5);
\draw[thick] (P2) -- (P3);
\fill (P1) circle(2pt);
\fill (P2) circle(2pt);
\fill (P3) circle(2pt);
\fill (P4) circle(2pt);
\fill (P5) circle(2pt);
}
}

\def\KfourTwoPath{
\tikz{
\coordinate (P1) at (0,0);
\coordinate (P2) at (\smallgraphnodedistance,0);
\coordinate (P3) at (0.5*\smallgraphnodedistance,\smallgraphnodedistance);
\coordinate (P4) at (1.5*\smallgraphnodedistance,\smallgraphnodedistance);
\coordinate (P5) at (2*\smallgraphnodedistance,0);
\draw[thick] (P1) -- (P2);
\draw[thick] (P1) -- (P3);
\draw[thick] (P1) -- (P4);
\draw[thick] (P2) -- (P4);
\draw[thick] (P3) -- (P4);
\draw[thick] (P2) -- (P5);
\draw[thick] (P2) -- (P3);
\draw[thick] (P5) -- (P4);
\fill (P1) circle(2pt);
\fill (P2) circle(2pt);
\fill (P3) circle(2pt);
\fill (P4) circle(2pt);
\fill (P5) circle(2pt);
}
}

\def\KTwoThreePlusEdgeTwo{
\tikz{
\coordinate (P1) at (0,0);
\coordinate (P2) at (\smallgraphnodedistance,0);
\coordinate (P3) at (0.5*\smallgraphnodedistance,\smallgraphnodedistance);
\coordinate (P4) at (1.5*\smallgraphnodedistance,\smallgraphnodedistance);
\coordinate (P5) at (2*\smallgraphnodedistance,0);
\draw[thick] (P3) -- (P1);
\draw[thick] (P3) -- (P2);
\draw[thick] (P3) -- (P5);
\draw[thick] (P4) -- (P1);
\draw[thick] (P4) -- (P2);
\draw[thick] (P4) -- (P5);
\draw[thick] (P1) -- (P2);
\fill (P1) circle(2pt);
\fill (P2) circle(2pt);
\fill (P3) circle(2pt);
\fill (P4) circle(2pt);
\fill (P5) circle(2pt);
}
}

\def\KTwoThreePlusTwoEdges{
\tikz{
\coordinate (P1) at (0,0);
\coordinate (P2) at (\smallgraphnodedistance,0);
\coordinate (P3) at (0.5*\smallgraphnodedistance,\smallgraphnodedistance);
\coordinate (P4) at (1.5*\smallgraphnodedistance,\smallgraphnodedistance);
\coordinate (P5) at (2*\smallgraphnodedistance,0);
\draw[thick] (P3) -- (P1);
\draw[thick] (P3) -- (P2);
\draw[thick] (P3) -- (P5);
\draw[thick] (P4) -- (P1);
\draw[thick] (P4) -- (P2);
\draw[thick] (P4) -- (P5);
\draw[thick] (P1) -- (P2);
\draw[thick] (P2) -- (P5);
\fill (P1) circle(2pt);
\fill (P2) circle(2pt);
\fill (P3) circle(2pt);
\fill (P4) circle(2pt);
\fill (P5) circle(2pt);
}
}

\def\KFiveMinusEdge{
\tikz{
\coordinate (P1) at (0,0);
\coordinate (P2) at (\smallgraphnodedistance,0);
\coordinate (P3) at (0.5*\smallgraphnodedistance,\smallgraphnodedistance);
\coordinate (P4) at (1.5*\smallgraphnodedistance,\smallgraphnodedistance);
\coordinate (P5) at (2*\smallgraphnodedistance,0);
\draw[thick] (P3) -- (P1);
\draw[thick] (P3) -- (P2);
\draw[thick] (P3) -- (P5);
\draw[thick] (P4) -- (P1);
\draw[thick] (P4) -- (P2);
\draw[thick] (P4) -- (P5);
\draw[thick] (P1) -- (P2);
\draw[thick] (P2) -- (P5);
\draw[thick] (P3) -- (P4);
\fill (P1) circle(2pt);
\fill (P2) circle(2pt);
\fill (P3) circle(2pt);
\fill (P4) circle(2pt);
\fill (P5) circle(2pt);
}
}

\def\KTwoUnionThreeStar{
\tikz{
\coordinate (P1) at (0,0);
\coordinate (P2) at (\smallgraphnodedistance,0);
\coordinate (P3) at (0.5*\smallgraphnodedistance,\smallgraphnodedistance);
\coordinate (P4) at (1.5*\smallgraphnodedistance,\smallgraphnodedistance);
\coordinate (P5) at (2*\smallgraphnodedistance,0);
\coordinate (P6) at (2.5*\smallgraphnodedistance,\smallgraphnodedistance);
\draw[thick] (P3) -- (P2);
\draw[thick] (P3) -- (P1);
\draw[thick] (P3) -- (P4);
\draw[thick] (P5) -- (P6);
\fill (P1) circle(2pt);
\fill (P2) circle(2pt);
\fill (P3) circle(2pt);
\fill (P4) circle(2pt);
\fill (P5) circle(2pt);
\fill (P6) circle(2pt);
}
}

\def\TwoTwoPaths{
\tikz{
\coordinate (P1) at (0,0);
\coordinate (P2) at (\smallgraphnodedistance,0);
\coordinate (P3) at (0.5*\smallgraphnodedistance,\smallgraphnodedistance);
\coordinate (P4) at (1.5*\smallgraphnodedistance,\smallgraphnodedistance);
\coordinate (P5) at (2*\smallgraphnodedistance,0);
\coordinate (P6) at (2.5*\smallgraphnodedistance,\smallgraphnodedistance);
\draw[thick] (P1) -- (P3) -- (P2);
\draw[thick] (P4) -- (P5) -- (P6);
\fill (P1) circle(2pt);
\fill (P2) circle(2pt);
\fill (P3) circle(2pt);
\fill (P4) circle(2pt);
\fill (P5) circle(2pt);
\fill (P6) circle(2pt);
}
}

\def\KTwoUnionThreePath{
\tikz{
\coordinate (P1) at (0,0);
\coordinate (P2) at (\smallgraphnodedistance,0);
\coordinate (P3) at (0.5*\smallgraphnodedistance,\smallgraphnodedistance);
\coordinate (P4) at (1.5*\smallgraphnodedistance,\smallgraphnodedistance);
\coordinate (P5) at (2*\smallgraphnodedistance,0);
\coordinate (P6) at (2.5*\smallgraphnodedistance,\smallgraphnodedistance);
\draw[thick] (P1) -- (P3) -- (P2) -- (P4);
\draw[thick] (P5) -- (P6);
\fill (P1) circle(2pt);
\fill (P2) circle(2pt);
\fill (P3) circle(2pt);
\fill (P4) circle(2pt);
\fill (P5) circle(2pt);
\fill (P6) circle(2pt);
}
}

\def\KTwoUnionTrianglePendant{
\tikz{
\coordinate (P1) at (0,0);
\coordinate (P2) at (\smallgraphnodedistance,0);
\coordinate (P3) at (0.5*\smallgraphnodedistance,\smallgraphnodedistance);
\coordinate (P4) at (1.5*\smallgraphnodedistance,\smallgraphnodedistance);
\coordinate (P5) at (2*\smallgraphnodedistance,0);
\coordinate (P6) at (2.5*\smallgraphnodedistance,\smallgraphnodedistance);
\draw[thick] (P1) -- (P3) -- (P2) -- (P4);
\draw[thick] (P1) -- (P2);
\draw[thick] (P5) -- (P6);
\fill (P1) circle(2pt);
\fill (P2) circle(2pt);
\fill (P3) circle(2pt);
\fill (P4) circle(2pt);
\fill (P5) circle(2pt);
\fill (P6) circle(2pt);
}
}

\def\KTwoUnionCfour{
\tikz{
\coordinate (P1) at (0,0);
\coordinate (P2) at (\smallgraphnodedistance,0);
\coordinate (P3) at (0.5*\smallgraphnodedistance,\smallgraphnodedistance);
\coordinate (P4) at (1.5*\smallgraphnodedistance,\smallgraphnodedistance);
\coordinate (P5) at (2*\smallgraphnodedistance,0);
\coordinate (P6) at (2.5*\smallgraphnodedistance,\smallgraphnodedistance);
\draw[thick] (P1) -- (P2) -- (P4) -- (P3) -- (P1);
\draw[thick] (P5) -- (P6);
\fill (P1) circle(2pt);
\fill (P2) circle(2pt);
\fill (P3) circle(2pt);
\fill (P4) circle(2pt);
\fill (P5) circle(2pt);
\fill (P6) circle(2pt);
}
}

\def\TriangleUnionTwoPath{
\tikz{
\coordinate (P1) at (0,0);
\coordinate (P2) at (\smallgraphnodedistance,0);
\coordinate (P3) at (0.5*\smallgraphnodedistance,\smallgraphnodedistance);
\coordinate (P4) at (1.5*\smallgraphnodedistance,\smallgraphnodedistance);
\coordinate (P5) at (2*\smallgraphnodedistance,0);
\coordinate (P6) at (2.5*\smallgraphnodedistance,\smallgraphnodedistance);
\draw[thick] (P1) -- (P2) -- (P3) -- (P1);
\draw[thick] (P4) -- (P5) -- (P6);
\fill (P1) circle(2pt);
\fill (P2) circle(2pt);
\fill (P3) circle(2pt);
\fill (P4) circle(2pt);
\fill (P5) circle(2pt);
\fill (P6) circle(2pt);
}
}

\def\KfourMinusEdgeUnionKtwo{
\tikz{
\coordinate (P1) at (0,0);
\coordinate (P2) at (\smallgraphnodedistance,0);
\coordinate (P3) at (0.5*\smallgraphnodedistance,\smallgraphnodedistance);
\coordinate (P4) at (1.5*\smallgraphnodedistance,\smallgraphnodedistance);
\coordinate (P5) at (2*\smallgraphnodedistance,0);
\coordinate (P6) at (2.5*\smallgraphnodedistance,\smallgraphnodedistance);
\draw[thick] (P1) -- (P2) -- (P4) -- (P3) -- (P1) -- (P4);
\draw[thick] (P5) -- (P6);
\fill (P1) circle(2pt);
\fill (P2) circle(2pt);
\fill (P3) circle(2pt);
\fill (P4) circle(2pt);
\fill (P5) circle(2pt);
\fill (P6) circle(2pt);
}
}

\def\TwoTriangles{
\tikz{
\coordinate (P1) at (0,0);
\coordinate (P2) at (\smallgraphnodedistance,0);
\coordinate (P3) at (0.5*\smallgraphnodedistance,\smallgraphnodedistance);
\coordinate (P4) at (1.5*\smallgraphnodedistance,\smallgraphnodedistance);
\coordinate (P5) at (2*\smallgraphnodedistance,0);
\coordinate (P6) at (2.5*\smallgraphnodedistance,\smallgraphnodedistance);
\draw[thick] (P1) -- (P2) -- (P3) -- (P1);
\draw[thick] (P4) -- (P5) -- (P6) -- (P4);
\fill (P1) circle(2pt);
\fill (P2) circle(2pt);
\fill (P3) circle(2pt);
\fill (P4) circle(2pt);
\fill (P5) circle(2pt);
\fill (P6) circle(2pt);
}
}

\def\KfourUnionKtwo{
\tikz{
\coordinate (P1) at (0,0);
\coordinate (P2) at (\smallgraphnodedistance,0);
\coordinate (P3) at (0.5*\smallgraphnodedistance,\smallgraphnodedistance);
\coordinate (P4) at (1.5*\smallgraphnodedistance,\smallgraphnodedistance);
\coordinate (P5) at (2*\smallgraphnodedistance,0);
\coordinate (P6) at (2.5*\smallgraphnodedistance,\smallgraphnodedistance);
\draw[thick] (P1) -- (P2) -- (P4) -- (P3) -- (P1);
\draw[thick] (P2) -- (P3);
\draw[thick] (P1) -- (P4);
\draw[thick] (P5) -- (P6);
\fill (P1) circle(2pt);
\fill (P2) circle(2pt);
\fill (P3) circle(2pt);
\fill (P4) circle(2pt);
\fill (P5) circle(2pt);
\fill (P6) circle(2pt);
}
}

\def\TwoPathUnionTwoKtwo{
\tikz{
\coordinate (P1) at (0,0);
\coordinate (P2) at (\smallgraphnodedistance,0);
\coordinate (P3) at (0.5*\smallgraphnodedistance,\smallgraphnodedistance);
\coordinate (P4) at (1.5*\smallgraphnodedistance,\smallgraphnodedistance);
\coordinate (P5) at (2*\smallgraphnodedistance,0);
\coordinate (P6) at (2.5*\smallgraphnodedistance,\smallgraphnodedistance);
\coordinate (P7) at (3*\smallgraphnodedistance,0);
\draw[thick] (P1) -- (P3) -- (P2);
\draw[thick] (P4) -- (P5);
\draw[thick] (P6) -- (P7);
\fill (P1) circle(2pt);
\fill (P2) circle(2pt);
\fill (P3) circle(2pt);
\fill (P4) circle(2pt);
\fill (P5) circle(2pt);
\fill (P6) circle(2pt);
\fill (P7) circle(2pt);
}
}

\def\KthreeUnionTwoKtwo{
\tikz{
\coordinate (P1) at (0,0);
\coordinate (P2) at (\smallgraphnodedistance,0);
\coordinate (P3) at (0.5*\smallgraphnodedistance,\smallgraphnodedistance);
\coordinate (P4) at (1.5*\smallgraphnodedistance,\smallgraphnodedistance);
\coordinate (P5) at (2*\smallgraphnodedistance,0);
\coordinate (P6) at (2.5*\smallgraphnodedistance,\smallgraphnodedistance);
\coordinate (P7) at (3*\smallgraphnodedistance,0);
\draw[thick] (P1) -- (P3) -- (P2) -- (P1);
\draw[thick] (P4) -- (P5);
\draw[thick] (P6) -- (P7);
\fill (P1) circle(2pt);
\fill (P2) circle(2pt);
\fill (P3) circle(2pt);
\fill (P4) circle(2pt);
\fill (P5) circle(2pt);
\fill (P6) circle(2pt);
\fill (P7) circle(2pt);
}
}

\def\FourKtwo{
\tikz{
\coordinate (P1) at (0,0);
\coordinate (P2) at (\smallgraphnodedistance,0);
\coordinate (P3) at (0.5*\smallgraphnodedistance,\smallgraphnodedistance);
\coordinate (P4) at (1.5*\smallgraphnodedistance,\smallgraphnodedistance);
\coordinate (P5) at (2*\smallgraphnodedistance,0);
\coordinate (P6) at (2.5*\smallgraphnodedistance,\smallgraphnodedistance);
\coordinate (P7) at (3*\smallgraphnodedistance,0);
\coordinate (P8) at (3.5*\smallgraphnodedistance,\smallgraphnodedistance);
\draw[thick] (P1) -- (P3);
\draw[thick] (P2) -- (P4);
\draw[thick] (P5) -- (P6);
\draw[thick] (P7) -- (P8);
\fill (P1) circle(2pt);
\fill (P2) circle(2pt);
\fill (P3) circle(2pt);
\fill (P4) circle(2pt);
\fill (P5) circle(2pt);
\fill (P6) circle(2pt);
\fill (P7) circle(2pt);
\fill (P8) circle(2pt);
}
}

\begin{table}[htp]
\centering\small
% \resizebox{\textwidth}{!}{
$
\begin{array}{ccccccc}
\hline
G & v & e & c(G) & |\mathrm{Aut}(G)| & P(G)=P(G;r,s,n) \\
\hline
\Ktwo & 2 & 1 & 1 & 2 & \overline{100}-2 \\
\hline
\Ptwo & 3 & 2 & 1 & 2 & P(\Ktwo)^2 \\
\Kthree & 3 & 3 & 1 & 6 & \overline{200}-2 \\
\TwoKtwo & 4 & 2 & 2 & 8 & \overline{111}\,P(\Ktwo)^2 \\
\hline
\ThreeStar & 4 & 3 & 1 & 6 & P(\Ktwo)^3 \\
\Pthree & 4 & 3 & 1 & 2 & P(\Ktwo)^3 \\
% \KthreeLeaf & 4 & 1 & 4 & 2 & P(\Kthree)P(\Ktwo) \\ %becky: typo in submitted version
\KthreeLeaf & 4 & 4 & 1 & 2 & P(\Kthree)P(\Ktwo) \\ %becky: corrected version
\Cfour & 4 & 4 & 1 & 8 & \overline{300}+6\,\overline{110}-12\,\overline{100}+16 \\
\KfourMinusEdge & 4 & 5 & 1 & 4 & \overline{300}+2\,\overline{110}-4\,\overline{100}+4 \\
\Kfour & 4 & 6 & 1 & 24 & \overline{300}-2 \\
\PtwoKtwo & 5 & 3 & 2 & 4 & \overline{111}\,P(\Ktwo)^3 \\
\KthreeKtwo & 5 & 4 & 2 & 12 & \overline{111}\,P(\Kthree)P(\Ktwo) \\
\ThreeKtwo & 6 & 3 & 3 & 48 & \overline{222}\,P(\Ktwo)^3 \\
\hline
\end{array}
$
% }
\caption{The polynomial $P(G)$ for $v$-vertex graphs $G$ with $v-c(G) \leq 3$.}\label{ta:polyP1}
\end{table}

\begin{table}[htp]
\centering\small
% \resizebox{0.8\textwidth}{!}{
$
\begin{array}{ccccccc}
\hline
G & v & e & c(G) & |\mathrm{Aut}(G)| & P(G)=P(G;r,s,n) \\
\hline
\FiveNodeTreeOne & 5 & 4 & 1 & 24 & P(\Ktwo)^4 \\
\FiveNodeTreeTwo & 5 & 4 & 1 & 2 & P(\Ktwo)^4 \\
\TriangleCherry & 5 & 5 & 1 & 4 & P(\Kthree)P(\Ktwo)^2 \\
\FourCycPendant & 5 & 5 & 1 & 2 & P(\Cfour)P(\Ktwo) \\
\ThreeCycTwoPendant & 5 & 5 & 1 & 2 & P(\Kthree)P(\Ktwo)^2 \\
\KfourMinusEdgePendantOne & 5 & 6 & 1 & 2 & P(\KfourMinusEdge)P(\Ktwo) \\
\KTwoThree & 5 & 6 & 1 & 12 & P(\KfourMinusEdge)P(\Ktwo) \\
\KTwoThreePlusEdgeOne & 5 & 7 & 1 & 12 & \overline{400}+3\,\overline{210}+6\,\overline{111}+3\,\overline{211}-6\,\overline{200}-12\,\overline{110}+12\,\overline{100}-8 \\
\FiveNodePath & 5 & 4 & 1 & 2 & P(\Ktwo)^4 \\
\TriangleTwoPath & 5 & 5 & 1 & 2 & P(\Kthree) P(\Ktwo)^2 \\
\Bowtie & 5 & 6 & 1 & 8 & P(\Kthree)^2 \\
\Cfive & 5 & 5 & 1 & 10 & \overline{400}+10\,\overline{210}-20\,\overline{200}-30\,\overline{110}+40\,\overline{100}-32 \\
\House & 5 & 6 & 1 & 2 & \overline{400}+4\,\overline{210}+4\,\overline{211}-8\,\overline{200}-6\,\overline{110}+4\,\overline{110}+4 \\
\KTwoThreePlueEdgeTwo & 5 & 7 & 1 & 2 & \overline{400}+2\,\overline{210}-4\,\overline{200}-2\,\overline{110}+4 \\
\KfourMinusEdgePendantTwo & 5 & 6 & 1 & 2 & P(\KfourMinusEdge)P(\Ktwo) \\
\KfourPendant & 5 & 7 & 1 & 6 & P(\Kfour)P(\Ktwo) \\
\KfourTwoPath & 5 & 8 & 1 & 4 & \overline{400}+\overline{210}-2\,\overline{200}-2\,\overline{100}+4 \\
\KTwoThreePlusEdgeTwo & 5 & 7 & 1 & 4 & \overline{400}+\overline{210}-2\,\overline{200}+4\,\overline{110}-10\,\overline{100}+16 \\
\KTwoThreePlusTwoEdges & 5 & 8 & 1 & 8 & \overline{400}+4\,\overline{110}-8\,\overline{100}+10 \\
\KFiveMinusEdge & 5 & 9 & 1 & 12 & \overline{400}+2\overline{110}-4\,\overline{100}+4 \\
K_5 & 5 & 10 & 1 & 120 & \overline{400}-2 \\
\KTwoUnionThreeStar & 6 & 4 & 2 & 12 & \overline{111}\, P(\Ktwo)^4 \\
\TwoTwoPaths & 6 & 4 & 2 & 8 & \overline{111}\, P(\Ktwo)^4 \\
\KTwoUnionThreePath & 6 & 4 & 2 & 4 & \overline{111}\, P(\Ktwo)^4 \\
\KTwoUnionTrianglePendant & 6 & 5 & 2 & 4 & \overline{111}\, P(\Kthree) P(\Ktwo)^2 \\
\KTwoUnionCfour & 6 & 5 & 2 & 16 & \overline{111}\, P(\Cfour) P(\Ktwo)^2 \\
\TriangleUnionTwoPath & 6 & 5 & 2 & 12 & \overline{111}\, P(\Kthree) P(\Ktwo)^2 \\
\KfourMinusEdgeUnionKtwo & 6 & 6 & 2 & 8 & \overline{111}\, P(\KfourMinusEdge) P(\Ktwo) \\
\TwoTriangles & 6 & 6 & 2 & 72 & \overline{111}\, P(\Kthree)^2 \\
\KfourUnionKtwo & 6 & 7 & 2 & 48 & \overline{111}\, P(\Kfour) P(\Ktwo) \\
\TwoPathUnionTwoKtwo & 7 & 4 & 3 & 16 & \overline{222}\, P(\Ptwo) P(\Ktwo)^2 \\
\KthreeUnionTwoKtwo & 7 & 5 & 3 & 48 & \overline{222}\, P(\Kthree) P(\Ktwo)^2 \\
\FourKtwo & 8 & 3 & 4 & 384 & \overline{333}\, P(\Ktwo)^4 \\
\hline
\end{array}
$
% }
\caption{The polynomial $P(G)$ for $v$-vertex graphs $G$ with $v-c(G)=4$.}\label{ta:polyP2}
\end{table}

\begin{table}[h]
\centering
\resizebox{\textwidth}{!}{
\begin{tabular}{rrrrrrrrrrrrrrrr} \hline
\ & \multicolumn{5}{l}{$\#\PLR(r,s,7;m)$}\\ \cline{2-16}
\ & \multicolumn{5}{l}{$r.s.7$}\\ \cline{2-16}
$m$ & 1.1.7 & 1.2.7 & 1.3.7 & 1.4.7 & 1.5.7  & 1.6.7 & 2.2.7 & 2.3.7 & 2.4.7 & 2.5.7 & 2.6.7 & 3.3.7 & 3.4.7 & 3.5.7 & 3.6.7\\ \hline
0 & 1 & 1 & 1 & 1 & 1 & 1 & 1 & 1 & 1 & 1 & 1 & 1 & 1 & 1 & 1\\
1 & 1 & 14 & 21 & 28 & 35 & 42 & 28 & 42 & 56 & 70 & 84 & 63 & 84 & 105 & 126\\
2 & & 42 & 126 & 252 & 420 & 630 & 266 & 672 & 1260 & 2030 & 2982 & 1638 & 3024 & 4830 & 7056 \\
3 & &  & 210 & 840 & 2100 & 4200 & 1008 & 5208 & 14784 & 31920 & 58800 & 22974 & 61488 & 128730 & 232680\\
4 & &  &  & 840 & 4200 & 12600 & 1302 & 20538 & 98364 & 299460 & 712530 & 190890 & 783972 & 2216340 & 5048190\\
5 & &  &  &  & 2520 & 15120 &  & 39060 & 378000 & 1739640 & 5549040 & 971838 & 6583248 & 26030340 & 76284180\\
6 & &  &  &  &  & 5040 &  & 28140 & 815640 & 6291600 & 28239960 & 3026772 & 37230984 & 214773720 & 829360980\\
7 & &  &  &  &  &  &  &  & 900480 & 13876800 & 93703680 & 5560380 & 142536240 & 1263691800 & 6610206960\\
8 & &  &  &  &  &  &  & & 390600 & 17711400 & 198840600 & 5477220 & 365911560 & 5328892800 & 39009259800\\
9 & &  &  &  &  &  &  & &  & 11718000 & 259408800 & 2212980 & 613495680 & 16053853200 & 171041026800\\
10 & &  &  &  &  &  &  & &  & 3059280 &  194125680 &  & 637509600 & 34161276240 & 556100475840\\
11 & &  &  &  &  &  &  & &  &  & 73422720 &  & 369109440 & 50271606000 & 1331142603840\\
12 & &  &  &  &  &  &  & &  &  & 10679760 &  & 90296640 & 49395578400 & 2316314150640\\
13 & &  &  &  &  &  &  & &  &  &  &  &  & 30542853600 & 2873671668000\\
14 & &  &  &  &  &  &  & &  &  &  &  &  & 10629360000 & 2470492936800\\
15 & &  &  &  &  &  &  & &  &  &  &  &  & 1573165440 & 1411231731840\\
16 & &  &  &  &  &  &  & &  &  &  &  &  & & 501894973440\\
17 & &  &  &  &  &  &  & &  &  &  &  &  & & 99105431040\\
18 & &  &  &  &  &  &  & &  &  &  &  &  & & 8211571200\\
\hline
Total & 8 & 57 & 358 & 1961 & 9276 & 37633 & 2605 & 93661 & 2599185 & 54730201 & 864744637 & 17464756  & 2263521961 & 199463431546 & 11785736969413\\ \hline
\end{tabular}
}
\caption{The values of $\#\PLR(r,s,7;m)$; continued in Tables~\ref{ta:rs7b} and~\ref{ta:r77}.}\label{ta:rs7a}
\end{table}

\begin{table}[h]
{\scriptsize
\resizebox{\textwidth}{!}{
\begin{tabular}{rrrrrrrr} \hline
\ & \multicolumn{5}{l}{$\#\PLR(r,s,7;m)$}\\ \cline{2-7}
\ & \multicolumn{5}{l}{$r.s.7$}\\ \cline{2-7}
$m$ & 4.4.7 & 4.5.7 & 4.6.7 & 5.5.7 & 5.6.7 & 6.6.7\\ \hline
0 & 1 & 1 & 1 & 1 & 1 & 1\\
1 & 112 & 140 & 168 & 175 & 210 &  252\\
2 & 5544 & 8820 & 12852 & 14000 & 20370 & 29610\\
3 & 160608 & 331800 & 594720 & 680400 & 1213800 & 2158800 \\
4 & 3040464 & 8342040 & 18654300 & 22520400 & 49851900 & 109648350\\
5 & 39789792 & 148690080 & 421288560 & 539486640 & 1501095960 & 4129786080\\
6 & 371511504 & 1945492080 & 7103917800 & 9705007200 & 34417437600 & 119886474960\\
7 & 2518935552 & 19094265120 & 91553898240 & 134286297600 & 616139899200 & 2752801934400\\
8 & 12508115256 & 142468484760 & 915820562160 & 1452407800200 & 8762762710800 & 50916808769400\\
9 & 45551970240 & 814365132000 & 7182549494400 & 12413692800600 & 100228554703200 & 768744893767200\\
10 & 121055555040 & 3578117047680 & 44440902031680 & 84446936458080 & 930070756954080 & 9567352024458480\\
11 & 231977692800 & 12080294553600 & 217628184896640 & 459215324652000 & 7044643080720000 & 98867148338165760\\
12 & 313967041920 & 31220730777600 & 844055906319360 & 2000199369924000 & 43729884582552000 & 852898268432422800\\
13 & 290077079040 & 61311770150400 & 2588575554835200 & 6978471536484000 & 223002930233664000 & 6166083869012592000\\
14 & 172656368640 & 90439590528000 & 6253755470524800 & 19466170012296000 & 935144001957312000 & 37457508269996136000\\
15 & 59253304320 & 98519956738560 & 11829864008309760 & 43255879780478400 & 3223730533876070400 & 191508486724243180800\\
16 & 8859553920 & 77323490294400 & 17371923533959680 & 76143045893544000 & 9122624741349504000 & 824650379018257377600\\
17 &  & 42126214233600 & 19575543305041920 & 105358455643896000 & 21136648322563200000 & 2990529904515892704000\\
18 &  & 14995766822400 & 16668621405273600 & 113411691586368000 & 39946627806672384000 & 9125647379336687472000\\
19 &  & 3114811929600 & 10507596032102400 & 93673159102656000 & 61276248720916992000 & 23396431711383803520000\\
20 &  & 284634362880 & 4768455577697280 & 58277161295539200 & 75809634470217446400 & 50284894217671092470400\\
21 &  &  & 1496705022167040 & 26615217299328000 & 75050036025947136000 & 90334790420061996748800\\
22 &  &  & 305223851842560 & 8591796855936000 & 58874808204632448000 & 135149758456395303936000\\
23 &  &  & 36075091046400 & 1844984711808000 & 36156622400801280000 & 167647764880657152000000\\
24 &  &  & 1862525145600 & 235436435136000 &  17119090026206784000 & 171521836534811629440000\\
25 &  &  &  & 13481774369280 & 6126452955671086080 & 143844493310595330785280\\
26 &  &  &  &  & 1613475264781900800 & 98168168535490134466560\\
27 &  &  &  &  & 300888959183769600 & 54057999485833839820800\\
28 &  &  &  &  & 37371505393152000 & 23779699801418663424000\\
29 &  &  &  &  & 2759601374208000 & 8256726182294360064000\\
30 &  &  &  &  & 91288879718400 & 2230046357199562137600\\
31 &  &  &  &  &  & 459939642510304051200\\
32 &  &  &  &  &  & 70680227381503488000\\
33 &  &  &  &  &  & 7813153251735552000\\
34 &  &  &  &  &  & 587441307350016000\\
35 &  &  &  &  &  & 27048481121894400\\
36 &  &  &  &  &  & 583662346444800\\ \hline
Total & 1258840124753 & 435973408185561 & 92518523839617121 & 556422824213480176 & 407007072002505214801 & 982388579887448747338333\\ \hline
\end{tabular}
}}
\caption{The values of $\#\PLR(r,s,7;m)$; continuing from Table~\ref{ta:rs7a} and continued in Table~\ref{ta:r77}.}\label{ta:rs7b}
\end{table}

\begin{table}
\centering
\resizebox{\textwidth}{!}{
\begin{tabular}{rrrrrrrr} \hline
\ & \multicolumn{5}{l}{$\#\PLR(r,7,7;m)$}\\ \cline{2-8}
\ & \multicolumn{5}{l}{$r.7.7$}\\ \cline{2-8}
$m$ & 1.7.7 & 2.7.7 & 3.7.7 & 4.7.7 & 5.7.7 & 6.7.7 & 7.7.7\\ \hline
\hline
0 & 1 & 1 & 1 & 1 & 1 & 1 & 1\\
1 & 49 & 98 & 147 & 196 & 245 & 294 & 343 \\
2 & 882 & 4116 & 9702 & 17640 & 27930 & 40572  & 55566 \\
3 & 7350 & 97608 & 381318 & 969024 & 1971270 & 3498600 & 5661558 \\
4 & 29400 & 1450302 & 9983358 & 36434244 & 96693660 & 211737330 & 407626002 \\
5 & 52920 & 14173740 & 184571730 & 996695280 & 3508057980 & 9577064700 & 22091837670 \\
6 & 35280 & 93118620 & 2493017100 & 20589037560 & 97824178200 & 336641627700 & 937499611860 \\
7 & 5040 & 413327040 & 25114127220 & 329058167760 & 2151220104600 & 9441643402800 & 31995541817820 \\
8 & & 1229208120 & 191003176980 & 4136301605520 & 37983532771800 & 215279839870200 & 895147467758460 \\
9 & & 2396605680 & 1103575119780 & 41356003473120 & 545522619369000 & 4045906316281200 & 20823534145010940 \\
10 & & 2949266880 & 4851540242640 & 331382137961280 & 6433667771868960 & 63326098060263360 & 407161408673448240 \\
11 & & 2154479040 & 16187551364880 & 2138171372830080 & 62740410283404000 & 832012118607983040 & 6747928605026748720 \\
12 & & 845696880 & 40729136096880 & 11136703296096000 & 508448168895240000 & 9231573671794519920 & 95414556472688784240 \\
13 & & 149516640 & 76460194354320 & 46850306414526720 & 3435977823932808000 & 86898548210325012000 & 1157011091919371520720 \\
14 & & 9344160 & 105451973716320 & 158998861707477120 & 19404150304485744000 & 696350429909011332000 & 12080739048610887859680 \\
15 & &  & 104574912049440 & 434071037204501760 & 91672146486194601600 & 4762047565503866736000 & 108953982522887641120800 \\
16 & &  & 72399498706080 & 948939730997852160 & 362335335766117560000 & 27837450938084937912000 & 850783267970671119386400 \\
17 & &  & 33593118763680 & 1650710438532288000 & 1197163519923384216000 & 139234645196548772976000 & 5762016601975755442288800 \\
18 & &  & 9863841496320 & 2265820889356362240 & 3300872875576140816000 & 596053661819256139968000 & 33886035458720657756006400 \\
19 & &  & 1690904920320 & 2427993422686218240 & 7575781755486572208000 & 2183386963441728494016000 & 173162378870925255394329600 \\
20 & &  & 151342732800 & 2003724888642247680 & 14423117979384567129600 & 6838094015249132148105600 & 769084553432477123576582400 \\
21 & &  & 5411750400  & 1251798397787105280 & 22679370383067115200000 & 18286225276118547108614400 & 2968275252277138334102611200 \\
22 & &  & & 579227971312972800 & 29296238052876891264000 & 41677467136757620221715200  & 9949401025065781152960038400 \\
23 & &  & & 193034752263198720 & 30887319756977889408000 & 80769545715156012215424000 & 28936380037160064863828851200 \\
24 & &  & & 44693646858846720 & 26372495140248144384000 & 132713702198233518483072000 & 72925724928137854106413900800 \\
25 & &  & & 6860934701107200 & 18066619214217787207680 & 184255530135693426746542080 & 158995953693483311073403284480 \\
26 & &  & & 655091210188800 & 9820001352274125465600 & 215290252940027718278891520 & 299284107563352671166958064640 \\
27 & &  & & 34832706048000 & 4178611218476036966400 & 210725753874837580720373760 & 485219498022119678183816647680 \\
28 & &  & & 782137036800 & 1369620935962581657600 & 171867215661546141628416000 & 675695410918704050010696483840 \\
29 & &  & &  & 339042043904814028800 & 116095014971808867619430400 & 805661431870384140528231336960 \\
30 & &  & &  & 61869133685050675200 & 64501823805957475751116800 & 819586476449488450769559091200 \\
31 & &  & &  & 8073669012853248000 & 29244277165012229350195200 & 708500236904008865987686041600 \\
32 & &  & &  & 723625308177408000 & 10722854433456169179033600 & 518142986889857917315006003200 \\
33 & &  & &  & 41905003262976000 & 3147095680488347470848000 & 318986018230847365565350041600 \\
34 & &  & &  & 1401648095232000 & 730644033010729291776000 & 164410170607527803740951142400 \\
35 & &  & &  & 20449013760000 & 132359825819298115584000 & 70519394490467460516228096000 \\
36 & &  & &  &  & 18410847956245011456000 & 25006983924959260345820160000 \\
37 & &  & &  &  & 1928283614905632768000 & 7279610083869038651882496000 \\
38 & &  & &  &  & 148277402084431872000 & 1726520284027400861325312000 \\
39 & &  & &  &  & 8074157388079104000 & 331008922550326911141888000 \\
40 & &  & &  &  & 293524209893376000 & 50895029497370545118822400 \\
41 & &  & &  &  & 6359357620224000 & 6228531360821639220019200 \\
42 & &  & &  &  & 61479419904000 & 602622105599612348006400 \\
43 & &  & &  &  & & 45850824283578118963200 \\
44 & &  & &  &  & & 2734157863261981900800 \\
45 & &  & &  &  & & 127631489644560384000 \\
46 & &  & &  &  & & 4668091942993920000 \\
47 & &  & &  &  & & 134218380312576000 \\
48 & &  & &  &  & & 3012491575296000 \\
49 & &  & &  &  & & 61479419904000 \\
\hline
Total & 130922 & 10256288925 & 467281806581416 & 12027068084311265945 & 170054389801868987652126 & 1289970420801370588662084277 & 5175166233060627523665748739420 \\
\hline
\end{tabular}
}
\caption{The values of $\#\PLR(r,7,7;m)$.}\label{ta:r77}
\end{table}

\begin{table}[h]
\centering
\resizebox{\textwidth}{!}{
\begin{tabular}{rrrrrrrrrrrrrrrr} \hline
\ & \multicolumn{5}{l}{$\#\PLR(r,s,8;m)$}\\ \cline{2-16}
\ & \multicolumn{5}{l}{$r.s.8$}\\ \cline{2-16}
$m$ & 1.1.8 & 1.2.8 & 1.3.8 & 1.4.8 & 1.5.8  & 1.6.8 & 2.2.8 & 2.3.8 & 2.4.8 & 2.5.8 & 2.6.8 & 3.3.8 & 3.4.8 & 3.5.8 & 3.6.8\\ \hline
0 & 1 & 1 & 1 & 1 & 1 & 1 & 1 & 1 & 1 & 1 & 1 & 1 & 1 & 1 & 1\\
1 & 8 & 16 & 24 & 32 & 40 & 48 & 32 & 48 & 64 & 80 & 96 & 72 & 96 & 120 & 144\\
2 & & 56 & 168 & 336 & 560 & 840 & 352 & 888 & 1664 & 2680 & 3936 & 2160 & 3984 & 6360 & 9288\\
3 & & & 336 & 1344 & 3360 & 6720 & 1568 & 8064 & 22848 & 49280 & 90720 & 35328 & 94272 & 197040 & 355776\\
4 & & & & 1680 & 8400 & 25200 & 2408 & 37800 & 180432 & 548240 & 1302840 & 346248 & 1413216 & 3981600 & 9049320\\
5 & & & & & 6720 & 40320 &  & 86688 & 835968 & 3837120 & 12216960 & 2104704 & 14107968 & 55460160 & 161935200\\
6 & & & & & & 20160 & & 76272 & 2212224 & 17025120 & 76258560 & 7925232 & 95977056 & 548866080 & 2107810320\\
7 & & & & & & & & & 3050880 & 47040000 & 317197440 & 17823456 & 447552000 & 3921099840 & 20355148800\\
8 & & & & & & & & & 1681680 & 77053200 & 866199600 & 21748608 & 1417799376 & 20342463120 & 147462199920\\
9 & & & & & & & & & & 67267200 & 1501920000 & 10997952 & 2973054336 & 76474164480 & 805219282560\\
10 & & & & & & & & & & 23782080 & 1555303680 & & 3916301760 & 206198475840 & 3312619813440\\
11 & & & & & & & & & & & 856154880 & & 2911668480 & 390865104000 & 10212300322560\\
12 & & & & & & & & & & & 189564480 & & 925505280 & 503426851200 & 23343747174720\\
13 & & & & & & & & & & & & & & 415516953600 & 38901656989440\\
14 & & & & & & & & & & & & & & 196521292800 & 46071395395200\\
15 & & & & & & & & & & & & & & 40197104640 & 37300347786240\\
16 & & & & & & & & & & & & & & & 19394880744960\\
17 & & & & & & & & & & & & & & & 5775185848320\\
18 & & & & & & & & & & & & & & & 742119920640\\ \hline
Total & 9 & 73 & 529 & 3393 & 19081 & 93289 & 4361 & 209761 & 7985761 & 236605001 & 5376213193 & 60983761 & 12703477825 & 1854072020881 & 186029569786849\\ \hline
\end{tabular}
}
\caption{The values of $\#\PLR(r,s,8;m)$; continued in Table~\ref{ta:rs8a}.}\label{ta:rs8}
\end{table}

\begin{table}[h]
{\scriptsize
\resizebox{\textwidth}{!}{
\begin{tabular}{rrrrrrrr} \hline
\ & \multicolumn{5}{l}{$\#\PLR(r,s,8;m)$}\\ \cline{2-7}
\ & \multicolumn{5}{l}{$r.s.8$}\\ \cline{2-7}
$m$ & 4.4.8 & 4.5.8 & 4.6.8 & 5.5.8 & 5.6.8 & 6.6.8\\ \hline
0 & 1 & 1 & 1 & 1 & 1 & 1\\
1 & 128 & 160 & 192 & 200 & 240 & 288\\
2 & 7296 & 11600 & 16896 & 18400  & 26760 & 38880\\
3 & 245376 & 505920 & 905664 & 1035200  & 1844160 & 3275040\\
4 & 5440032 & 14863920 & 33150960 & 39944000  & 88164000 & 193314600\\
5 & 84155904 & 312224640 & 880629120 & 1123820160  & 3111238080 & 8513683200\\
6 & 938106624 & 4857854400 & 17614343040 & 23931230400 & 84205144800 & 290863660800\\
7 & 7674293760 & 57240046080 & 271706198400 & 395240496000  & 1792941696000 & 7913434233600\\
8 & 46492403328 & 517971847680 & 3284306156880 & 5147427465600  & 30572202805200 & 174666634178400\\
9 & 208994118144 & 3629706339840 & 31445326617600 &  53468641900800 & 422804060918400 & 3170532224025600\\
10 & 693958185984 & 19775273602560 & 240125846929920 & 446402183619840  & 4786206531503040 & 47813256027210240\\
11 & 1682575630848 & 83790652431360 & 1468289383142400 & 3010123202150400  & 44641974866227200 & 603661256780037120\\
12 & 2918423765376 & 275253395880960 & 7199937958106880 & 16435824668659200  & 344653334210505600 & 6417774203497977600\\
13 & 3499852769280 & 696318889996800 & 28293978724945920 & 72712045906752000  & 2208972817416960000 & 57702985271076096000\\
14 & 2737429309440 &  1341720153849600 &  88858361994393600 & 260310345154272000  & 11772041299608844800 & 440141451429993062400\\
15 & 1248707174400 & 1937078037135360 & 221915691372533760 & 751830550246218240  & 52182303997547888640 & 2854234197294902231040\\
16 & 250631700480 & 2044882328832000 & 437480320642485120 & 1743112037427264000  & 192252531989997360000 & 15755740879827344094720\\
17 & & 1520755813785600 & 673828360242831360 & 3220935488443008000  & 587611818336920832000 & 74072163230186875084800\\
18 & & 749652906240000 & 799632015238103040 & 4696696282529664000  & 1485550685585627136000 & 296515269372479241369600\\
19 & & 218552140800000 & 717288915035750400 & 5332415365638144000  & 3093542054560661760000 & 1009795856630569892352000\\
20 & & 28375521914880 & 473551206050119680 & 4627930686056294400  & 5277311702476213478400 & 2921119683107942455372800\\
21 & & & 221248992118210560 & 2991698111646720000  & 7323172199654814720000 & 7162134609634048893542400\\
22 & & & 68717472783482880 & 1386151005947904000 & 8192740332265767936000 & 14840974457028794640384000\\
23 & & & 12643342010449920 &  432666484604928000 & 7305506991725193216000 & 25896289714957972638720000\\
24 & & & 1036744153804800 & 81104713998336000  & 5116406448359066112000 & 37882633011208622775936000\\
25 & & &  &   & 2759923954130172641280 & 46209895877069315283271680\\
26  & & &  &  & 1116269083866463027200 & 46698712018213236924579840\\
27  & &  &  &  & 325482156403465420800 & 38792083589752166760038400\\
28  & &  &  &  & 64264706091590860800 & 26235781041173371579699200\\
29  & &  &  &  & 7655207985266688000 & 14276150397241050415104000\\
30  & &  &  &  & 413733776530636800 & 6157077288600135234355200\\
31  & &  & &  &  & 2063745518966035159449600\\
32  & &  & &  &  & 523353569391869239296000\\
33  & &  & &  &  & 96567369329870143488000\\
34  & &  & &  &  & 12182875723557568512000\\
35 & &   & &  &  & 937008615326102323200\\
36  & &  & &  &  & 33087582858697113600\\ \hline
Total & 13295767306401 & 8920365218163361 & 3753438773423308993 & 25624385022295308521  & 42914661462094545592201 & 271169169298945362007111849\\
\end{tabular}
}}
\caption{The values of $\#\PLR(r,s,8;m)$; continued from Table~\ref{ta:rs8}.  The present authors did not compute $\#\PLR(r,7,8;m)$ nor $\#\PLR(r,8,8;m)$.}\label{ta:rs8a}
\end{table}

\begin{table}[h]
{\tiny
\begin{center}
\begin{tabular}{rrrrrrr} \hline
\ & \multicolumn{5}{l}{$\#\mathrm{Isom}(n;m)$}\\ \cline{2-7}
\ & \multicolumn{6}{l}{$n$}\\ \cline{2-7}
$m$ & 1 & 2 & 3 & 4& 5 & 6\\ \hline
0 & 1 & 1 & 1 & 1 & 1 & 1\\
1 & 1 &  4 &  5 &  5 &  5& 5\\
2 & \ & 10 & 50 & 84 & 93& 94\\
3 & \ & 4  & 221 & 1120 & 2112& 2548\\
4 & \ & 1  & 525  & 10128 & 43955& 85234\\
5 & \ & \ & 651  & 60092  & 674957& 2508483\\
6 & \ & \ & 415 & 239302 & 7679384& 59110661\\
7 & \ & \ & 136 & 639098 & 65404265 & 1103309385\\
8 & \ & \ & 20 & 1148454 & 422142208&  16466869051\\
9 & \ & \ & 5 & 1374447 & 2080853035 & 198621450446\\
10 & \ & \ & \ & 1082019 & 7867483199 & 1953036511736\\
11 & \ & \ & \ & 548440 & 22843744418 & 15756857221135\\
12 & \ & \ & \ & 176137 & 50867669444 & 104784604156741\\
13 & \ & \ & \ & 35473 & 86544642569 & 576125696499417\\
14 & \ & \ & \ & 4696 & 111836743580 & 2623564948795633\\
15 & \ & \ & \ & 403 & 108882205792 & 9901507463165937\\
16 & \ & \ & \ & 35 & 79051125332 & 30959687376379661\\
17 & \ & \ & \ & \ & 42275685836 & 80100291981771263\\
18 & \ & \ & \ & \ & 16420711804  & 171118574787473668\\
19 & \ & \ & \ & \ & 4563456676 &  300957676311237853\\
20 & \ & \ & \ & \ & 894429087 &  434125855232450974\\
21 & \ & \ & \ & \ & 122238972 &  511227919780309083\\
22 & \ & \ & \ & \ & 11569016 &  488771341028032846\\
23 & \ & \ & \ & \ & 759296 &  376957644290919036\\
24 & \ & \ & \ & \ & 33736& 232788472371575258\\
25 & \ & \ & \ & \ & 1411&  114149339445885218\\
26 & \ & \ & \ & \ &  \ &  44033009520708974\\
27 & \ & \ & \ & \ &  \ &  13227534274721732\\
28 & \ & \ & \ & \ &  \ &  3061826358557444\\
29 & \ & \ & \ & \ &  \ &  540473537486248\\
30 & \ & \ & \ & \ &  \ &  72090555296085\\
31 & \ & \ & \ & \ &  \ &  7217657260917\\
32 & \ & \ & \ & \ &  \ &  540810639064\\
33 & \ & \ & \ & \ &  \ &  30364554576\\
34 & \ & \ & \ & \ &  \ &  1285684592\\
35 & \ & \ & \ & \ &  \ &  40649375\\
36 & \ & \ & \ & \ &  \ &  1130531\\ \hline
Total & 2 & 20 & 2029 & 5319934 & 534759300183 & 2815323435872410905\\ \hline
\end{tabular}
\end{center}
}
\caption{Number of isomorphism classes in $\PLS(n;m)$ for $n\leq 6$, according to weight $m$.}\label{tableIsom}
\end{table}

\begin{table}[h]
{
\resizebox{\textwidth}{!}{
\begin{tabular}{rrrrrrrrrrrrrrrr}  \hline
\ & \multicolumn{5}{l}{$\#\mathrm{Isot}(2,s,n;m)$}\\ \cline{2-16}
\ & \multicolumn{5}{l}{$2.s.n$}\\ \cline{2-16}
$m$ & 2.2.2 & 2.2.3 & 2.2.4 & 2.2.5 & 2.2.6 & 2.3.3 & 2.3.4 & 2.3.5 & 2.3.6 & 2.4.4 & 2.4.5 & 2.4.6 & 2.5.5 & 2.5.6 & 2.6.6\\ \hline
0 & 1 & 1 & 1 & 1 & 1 & 1 & 1 & 1 & 1 & 1 & 1 & 1 & 1 & 1 & 1\\
1 & 1 & 1 & 1 & 1 & 1 & 1 & 1 & 1 & 1 & 1 & 1 & 1 & 1 & 1 & 1\\
2 & 4 & 4 & 4 & 4 & 4 & 4 & 4 & 4 & 4 & 4 & 4 & 4 & 4 & 4 & 4\\
3 & 1 & 2 & 2 & 2 & 2 & 5 & 5 & 5 & 5 & 5 & 5 & 5 & 5 & 5 & 5\\
4 & 1 & 2 & 3 & 3 & 3 & 6 & 9 & 9 & 9 & 15 & 15 & 15 & 15 & 15 & 15\\
5 &   &   &   &   &   & 2 & 4 & 5 & 5 & 10 & 13 & 13 & 19 & 19 & 19\\
6 &   &   &   &   &   & 1 & 3 & 4 & 5 & 13 & 19 & 22 & 31 & 37 & 47\\
7 &   &   &   &   &   &   &   &   &   &  3 & 7 & 9 & 22 & 30 & 45\\
8 &   &   &   &   &   &   &   &   &   &  2 & 5 & 8 & 20 & 35 & 69\\
9 &   &   &   &   &   &   &   &   &   &    &   &   & 5   & 12 & 40\\
10 &  &  &  &  &  &  &  &  &  &  &  &  &              2  & 7 & 35\\
11 &  &  &  &  &  &  &  &  &  &  &  &  & &  & 7\\
12 &  &  &  &  &  &  &  &  &  &  &  &  & &  & 4\\ \hline
Total & 8 & 10 & 11 & 11 & 11 & 20 & 27 & 29 & 30 & 54 & 70 & 78 & 125 & 166 & 292\\ \hline
\end{tabular}
}}
\caption{Number of isotopism classes in $\PLR(2,s,n;m)$ for $2\leq s\leq n\leq 6$, according to weight $m$.}\label{tableIsot}
\end{table}

\begin{table}[h]
\small
\centering
\begin{tabular}{rrrrrrrrrrr} \hline
\ & \multicolumn{5}{l}{$\#\mathrm{Isot}(3,s,n;m)$}\\ \cline{2-11}
\ & \multicolumn{5}{l}{$3.s.n$}\\ \cline{2-11}
$m$ & 3.3.3 & 3.3.4 & 3.3.5 & 3.3.6 & 3.4.4 & 3.4.5 & 3.4.6 & 3.5.5 & 3.5.6 & 3.6.6\\ \hline
0 & 1 & 1 & 1 & 1 & 1 & 1 & 1 & 1 & 1 & 1\\
1 & 1 & 1 & 1 & 1 & 1 & 1 & 1 & 1 & 1 & 1\\
2 & 4 & 4 & 4 & 4 & 4 & 4 & 4 & 4 & 4 & 4\\
3 & 11 & 11 & 11 & 11 & 11 & 11 & 11 & 11 & 11 & 11\\
4 & 18 & 25 & 25 & 25 & 36 & 36 & 36 & 36 & 36 & 36\\
5 & 23 & 42 & 49 & 49 & 78 & 91 & 91 & 109 & 109 & 109\\
6 & 15 & 52 & 71 & 77 & 174 & 237 & 254 & 330 & 356 & 389\\
7 & 6 & 33 & 70 & 82 & 215 & 430 & 502 & 858 & 1012 & 1212\\
8 & 1 & 11 & 34 & 52 & 192 & 585 & 855 & 1770 & 2568 & 3782\\
9 & 1 & 4 & 13 & 23 & 91 & 491 & 962 & 2683 & 5168 &  10001\\
10 &  &  &  &  & 30 & 257 & 740 & 2689 & 7706 & 21857\\
11 &  &  &  &  & 4 & 71 & 298 & 1794 & 7988 &  35822\\
12 &  &  &  &  & 2 & 12 & 70 & 709 & 5446 & 42768\\
13 &  &  &  &  &  &  &  & 177 & 2301 & 34916\\
14 &  &  &  &  &  &  &  & 19 & 530 & 19078\\
15 &  &  &  &  &  &  &  & 3 & 62 & 6441\\
16 &  &  &  &  &  &  &  &  &  & 1315\\
17 &  &  &  &  &  &  &  &  &  & 133\\
18 &  &  &  &  &  &  &  &  &  &  16\\ \hline
Total & 81 & 184 & 279 & 325 & 839 & 2227 & 3825 & 11194 & 33299 & 177892\\ \hline
\end{tabular}
\caption{Number of isotopism classes in $\PLR(3,s,n;m)$ for $3\leq s\leq n\leq 6$, according to weight $m$.}\label{tableIsota}
\end{table}

\begin{table}[h]
\resizebox{\textwidth}{!}{
\begin{tabular}{rrrrrrrrrrr} \hline
\ & \multicolumn{5}{l}{$\#\mathrm{Isot}(r,s,n;m)$}\\ \cline{2-11}
\ & \multicolumn{5}{l}{$r.s.n$}\\ \cline{2-11}
$m$ & 4.4.4 & 4.4.5 & 4.4.6 & 4.5.5 & 4.5.6 & 4.6.6 & 5.5.5 & 5.5.6 & 5.6.6 & 6.6.6\\ \hline
0 & 1 & 1 & 1 & 1 & 1 & 1 & 1 & 1 & 1 & 1\\
1 & 1 & 1 & 1 &  1 & 1 & 1 & 1 & 1 & 1 & 1\\
2 & 4 & 4 & 4 & 4 & 4 & 4 & 4 & 4 & 4 & 4\\
3 & 11 & 11 & 11 & 11 & 11 & 11 & 11 & 11 & 11 & 11\\
4 & 52 & 52 & 52 & 52 & 52 & 52 & 52 & 52 & 52 & 52\\
5 & 139 & 160 & 160 & 187  & 187 & 187 & 221 & 221 & 221 & 221\\
6 & 507 & 668 & 707 & 882 & 935 & 997 & 1158 & 1227 & 1306 & 1396\\
7 & 1161 & 2103 & 2395 & 3713  & 4223 & 4826 & 6310  & 7127 & 8064  & 9130\\
8 & 2136 & 5678 & 7754 & 14266 & 19080 & 25524 & 33293 & 43322 & 56110 & 72145\\
9 & 2429 & 10739 & 19067 & 42940 & 72764 & 121508 & 150964 & 241958 & 380083 &  583339\\
10 & 2004 & 14881 & 36957 & 99301 & 230072 & 515040 & 554285 & 1174047 & 2384388 & 4627607\\
11 & 975 & 13865 & 50826 & 168900 & 565202 & 1797295 & 1594532 & 4699600 & 12974453 &  33362634\\
12 & 364 & 8970 & 50244 & 210285 & 1064946 & 5054807 & 3539461 & 15159299 & 59361654 & 210409407\\
13 & 72 & 3664 & 32727 & 187214 & 1498530 & 11135187 & 6017824 & 38833501 & 223569607 & 1129335392\\
14 & 18 & 995 & 13973 & 117985 & 1557518 & 19016101 & 7772366 & 78368607 & 686354327 & 5091624997\\
15 & 2 & 141 & 3268 &51094  & 1166309 & 24794117 & 7568187 & 123670028 & 1706058231 & 19140028219\\
16 & 2 & 22 & 411 & 14960 & 616603 & 24415585 & 5493206 & 151457082 & 3417379856 & 59761963636\\
17 &  &  &  & 2814 & 220158 & 17834146 & 2939617 & 142614087 & 5488132262 & 154544375137\\
18 &  &  &  & 332 & 50723 & 9492300 & 1141472 & 102078688 & 7025903964 & 330108625102\\
19 &  &  &  & 24  & 6591 & 3575605 & 317980 & 54746803 & 7119415871 & 580559388329\\
20 &  &  &  & 3 & 428 & 926317 & 62319 & 21601198 & 5662138638 & 837440466326\\
21 &  &  &  &  &  & 156463 & 8676 & 6121385 & 3498117999 & 986167409118\\
22 &  &  &  &  &  & 16759 & 823 & 1203460 & 1658503251 &  942850011453\\
23 &  &  &  &  &  & 960 & 69 & 155952 & 594594494 & 727157075193\\
24 &  &  &  &  &  & 56 & 6 & 12023 & 158425032 & 449054224783\\
25 &  &  &  &  &  &  & 2 & 486 & 30703736 & 220195944263\\
26 &  &  &  &  &  &  &  &  & 4220807 & 84941236104\\
27 &  &  &  &  &  &  &  &  & 396518 & 25516234965\\
28 &  &  &  &  &  &  &  &  & 24531 & 5906586539\\
29 &  &  &  &  &  &  &  &  & 886 & 1042616896\\
30 &  &  &  &  &  &  &  &  & 40 &  139114631\\
31 &  &  &  &  &  &  &  &  &  & 13928529\\
32 &  &  &  &  &  &  &  &  &  & 1048656\\
33 &  &  &  &  &  &  &  &  &  & 59130\\
34 &  &  &  &  &  &  &  &  &  &  2846\\
35 &  &  &  &  &  &  &  &  &  & 109\\
36 &  &  &  &  &  &  &  &  &  & 22\\ \hline
Total & 9878 & 61955 & 218558 & 914969 & 7074338 & 118883849 & 37202840  & 742190170 & 37349106398 & 5431010366323 \\ \hline
\end{tabular}}
\caption{Number of isotopism classes in $\PLR(r,s,n;m)$ for $4\leq r\leq s\leq n\leq 6$, according to weight $m$.}\label{tableIsot2}
\end{table}

\begin{table}[h]
\resizebox{\textwidth}{!}{
\begin{tabular}{rrrrrrrrrrrrrrrrr}  \hline
\ & \multicolumn{5}{l}{$\#\mathrm{MC}(r,s,n;m)$}\\ \cline{2-17}
\ & \multicolumn{5}{l}{$r.s.n$}\\ \cline{2-17}
$m$ & 2.2.2 & 2.2.3 & 2.2.4 & 2.2.5 & 2.2.6 & 2.3.3 & 2.4.4 & 2.5.5 & 2.6.6 & 3.3.3 & 3.3.4 & 3.3.5 & 3.3.6 & 3.4.4 & 3.5.5 & 3.6.6\\ \hline
0 & 1 & 1 & 1 & 1 & 1 & 1 & 1 & 1 & 1 & 1 & 1 & 1 & 1 & 1 & 1 & 1\\
1 & 1 & 1 & 1 & 1 & 1 & 1 & 1 & 1 & 1 & 1 & 1 & 1 & 1 & 1 & 1 & 1\\
2 & 2 & 3 & 3 & 3 & 3 & 3 & 3 & 3 & 3 & 2 & 3 & 3 & 3 & 3 & 3 & 3\\
3 & 1 & 2 & 2 & 2 & 2 & 4 & 4 & 4 & 4 & 5 & 8 & 8 & 8 & 8 & 8 & 8\\
4 & 1 & 2 & 3 & 3 & 3 & 5 & 11 & 11 & 11 & 8 & 18 & 18 & 18 & 24 & 24 & 24\\
5 &  &  &  &  &  & 2 & 8 & 14 & 14 & 9 & 28 & 33 & 33 & 49 & 67 & 67\\
6 &  &  &  &  &  & 1 & 10 & 22 & 32 & 7 & 34 & 46 & 50 & 104 & 191 & 224\\
7 &  &  &  &  &  &  & 3 & 17 & 32 & 4 & 23 & 46 & 54 & 128 & 477 & 667\\
8 &  &  &  &  &  &  & 2 & 16 & 48 & 1 & 9 & 24 & 36 & 116 & 963 & 2018\\
9 &  &  &  &  &  &  &  & 5 & 31 & 1 & 4 & 12 & 20 & 59 & 1444 & 5233\\
10 &  &  &  &  &  &  &  & 2 & 27 &  &  &  &  & 22 & 1452 & 11309\\
11 &  &  &  &  &  &  &  &  & 7 &  &  &  &  & 4 & 986 & 18435\\
12 &  &  &  &  &  &  &  &  & 4 &  &  &  &  & 2 & 411 & 21995\\
13 &  &  &  &  &  &  &  &  &  &  &  &  &  &  & 115 & 18028\\
14 &  &  &  &  &  &  &  &  &  &  &  &  &  &  & 16 & 9969\\
15 &  &  &  &  &  &  &  &  &  &  &  &  &  &  & 3 & 3468\\
16 &  &  &  &  &  &  &  &  &  &  &  &  &  &  &  & 765\\
17 &  &  &  &  &  &  &  &  &  &  &  &  &  &  &  & 94\\
18 &  &  &  &  &  &  &  &  &  &  &  &  &  &  &  & 14\\ \hline
Total  & 6 & 9 & 10 & 10 & 10 & 17 & 43 & 96 & 215 & 39 & 129 & 192 & 224 & 521 & 6162 & 92323\\ \hline
\end{tabular}
}
\caption{Number of main classes in $\PLR(r,s,n;m)$ for $2\leq r\leq s\leq n\leq 6$ with $r \leq 3$, according to weight $m$.  We omit the cases when $r$, $s$, and $n$ are pairwise distinct.}\label{tableMainClasses}
\end{table}

\begin{table}[h]
\begin{center}{\tiny
\begin{tabular}{rrrrrrrrrr}  \hline
\ & \multicolumn{5}{l}{$\#\mathrm{MC}(r,s,n;m)$}\\ \cline{2-10}
\ & \multicolumn{5}{l}{$r.s.n$}\\ \cline{2-10}
$m$ & 4.4.4 & 4.4.5 & 4.4.6 & 4.5.5 & 4.6.6 & 5.5.5 & 5.5.6 & 5.6.6 & 6.6.6\\ \hline
0 & 1 & 1 & 1 & 1 & 1 & 1 & 1 & 1 & 1\\
1 & 1 & 1 & 1 & 1 & 1 & 1 & 1 & 1 & 1\\
2 & 2 & 3 & 3 & 3 & 3 & 2 & 3  & 3 & 2\\
3 & 5 & 8 & 8 & 8 & 8 & 5 & 8 & 8 & 5\\
4 & 18 & 34 & 34 & 34 & 34 &  18 & 34 & 34 & 18\\
5 & 39 & 97 & 97 & 111 & 111 &  59 & 131 & 131 & 59\\
6 & 121 & 376 & 399 & 489 & 551 &  256 & 677 & 717 & 306\\
7 & 253 & 1135 & 1293 & 1969 & 2548 &  1224 & 3748 & 4225 & 1747\\
8 & 442 & 2987 & 4070 & 7392 & 13142 &  5997 & 22209 & 28670 & 12799\\
9 & 495 & 5579 & 9847 & 21936 & 61657 &  26188 & 122390 & 191855 & 99715\\
10 & 420 & 7694 & 18939 & 50420 & 259637 &  94479 & 590423 & 1197283 & 779295\\
11 & 218 & 7170 & 25943 & 85477 & 902847 &  269456 & 2356900 & 6500092 & 5583650\\
12 & 96 & 4686 & 25682 & 106351 & 2535116 &  595649 & 7593131 & 29710547 & 35131875\\
13 & 25 & 1944 & 16768 & 94754 & 5579487 &  1010706 & 19438925 & 111845936 & 188377998\\
14 & 8 & 561 & 7283 & 59910 & 9524578 &  1304319 & 39216773 & 343290367 & 848955581\\
15 & 2 & 88 & 1742 & 26146 & 12416456 &  1270356 & 61876720 & 853215848 & 3190714878\\
16 & 2 & 19 & 262 & 7790 & 12227832 &  923128 & 75774954 & 1708965453 & 9961645532\\
17 &  &  &  & 1533 & 8934309 &  495565 & 71353021 & 2744426741 & 25759586139\\
18 &  &  &  & 200 & 4758913 &  193531 & 51076956 & 3513371841 & 55021427957\\
19 &  &  &  & 19 & 1795458 &  54746 & 27401728 & 3560137618 & 96764408110\\
20 &  &  &  & 3 & 467000 &  11052 & 10816787 & 2831458432 & 139578978645\\
21 &  &  &  &  & 79692 &  1693 & 3069771  & 1749363542 & 164367335977\\
22 &  &  &  &  & 8815 &  192 & 604997 & 829461470 & 157147744329\\
23 &  &  &  &  & 566 &  26 & 79301 & 297418767 & 121198141862\\
24 &  &  &  &  & 44 &  4 & 6249 & 79274342 & 74846573994\\
25 &  &  &  &  & &  2 & 312 & 15377110 & 36702176578\\
26 &  &  &  &  & &   &  & 2119455 & 14158650257\\
27 &  &  &  &  & &   &  & 200664 & 4253618044\\
28 &  &  &  &  & &   &  & 12830 & 984869538\\
29 &  &  &  &  & &   &  & 527 & 173933415\\
30 &  &  &  &  & &   &  & 33 & 23245431\\
31 &  &  &  &  & &   &  &  & 2336988\\
32 &  &  &  &  &  &  &  &  & 179057\\
33 &  &  &  &  & &   &  &  & 10603\\
34 &  &  &  &  & &   &  &  & 640\\
35 &  &  &  &  & &   &  &  & 40\\
36 &  &  &  &  & &   &  &  & 12\\ \hline
Total  & 2148 & 32383 & 112372 & 464547 & 59568806 & 6239377 & 371406150 & 18677574543 & 905214521078\\ \hline
\end{tabular}
}\end{center}
\caption{Number of main classes in $\PLR(r,s,n;m)$ for $4\leq r\leq s\leq n\leq 6$, according to weight $m$.  We omit the cases when $r$, $s$, and $n$ are pairwise distinct.}\label{tableMainClasses1}
\end{table}

\begin{table}
\centering {\small
\begin{tabular}{rrr} \hline
$m$ & $\#\mathrm{Isot}(r,s,n;m)$ & $\#\mathrm{MC}(r,s,n;m)$ \\
\hline
0 & 1 & 1 \\
1 & 1 & 1 \\
2 & 4 & 2 \\
3 & 11 & 5 \\
4 & 52 & 18 \\
5 & 221 & 59 \\
6 & 1396 & 306 \\
7 & 9719 & 1861 \\
8 & 85145 & 15097 \\
9 & 860347 & 146893 \\
10 & 10071270 & 1693416 \\
11 & 133048009 & 22239872 \\
12 & ? & 327670703 \\
\hline
\end{tabular}}
\caption{Number of isotopism classes and main classes in $\PLR(r,s,n;m)$ when $r \geq m$, $s \geq m$, and $n \geq m$, according to weight $m$.  The second column is Sloane's \protect\url{oeis.org/A286317}.}
\label{ta:unbounded}
\end{table}

\end{document}